\documentclass[twoside]{article}
\usepackage{amsmath,amssymb,amsthm,amsfonts,amscd}
\usepackage{hyperref}
\usepackage[numeric]{amsrefs}

\usepackage{graphicx,Yau_style,titlesec}
\titleformat{\section}{\normalfont\Large\bfseries}{\thesection.}{0.5em}{}
\titleformat{\subsection}{\normalfont\large\bfseries}{\thesection.\thesubsection.}{0.5em}{}
\titleformat{\subsubsection}{\normalfont\bfseries}{\thesection.\thesubsection.\thesubsubsection.}{0.5em}{}

\input colordvi

\newtheorem{thm}[equation]{Theorem}

\newtheorem{lem}[equation]{Lemma}

\newtheorem{prop}[equation]{Proposition}

\newcommand{\thmref}[1]{theorem~\ref{#1}}

\newcommand{\lemref}[1]{lemma~\ref{#1}}
\newcommand{\secref}[1]{section~\ref{#1}}

\DeclareMathOperator{\Ad}{Ad}

\numberwithin{equation}{section}

\renewcommand\a{\alpha}

\newcommand\g{\gamma}
\renewcommand\d{\delta}

\renewcommand\l{\lambda}

\newcommand\s{\sigma}
\newcommand\D{\Delta}
\newcommand\fa{\mathfrak a}
\newcommand\fb{\mathfrak b}

\newcommand\fn{\mathfrak n}

\renewcommand\D{\Delta}
\newcommand\G{\Gamma}
\newcommand\f{\frac}
\newcommand\smallf[2]{{\textstyle{\frac{#1}{#2}}}}
\newcommand\srel[2]{\begin{smallmatrix} {#1} \\ {#2} \end{smallmatrix}}

\newcommand{\N}{{\mathbb{N}}}
\newcommand{\Z}{{\mathbb{Z}}}
\newcommand{\R}{{\mathbb{R}}}

\newcommand{\C}{{\mathbb{C}}}
\newcommand{\A}{{\mathbb{A}}}
\newcommand{\Q}{{\mathbb{Q}}}

\newcommand{\Sch}{{\mathcal{S}}}

\renewcommand\Re{\text{Re~}}

\newcommand{\supp}{\operatorname{supp}}

\renewcommand\i{^{-1}}
\renewcommand\({\left(}
\renewcommand\){\right)}
\newcommand{\ttwo}[4]{
\(\begin{smallmatrix}{#1} & {#2}
\\ {#3} & {#4} \end{smallmatrix}\)}
\newcommand{\tthree}[9]{
\(\begin{smallmatrix}{#1} & {#2} & {#3}
\\ {#4} & {#5} & {#6} \\ {#7} & {#8} & {#9} \end{smallmatrix}\)}

\newcommand{\sgn}{\operatorname{sgn}}

\newcommand{\gobble}[1]{}
  \newcommand{\rangeref}[2]{%
    \ref{#1}--\afterassignment\gobble\fam 0\ref{#2}%
  }
\makeatletter
\def\imod#1{\allowbreak\mkern10mu({\operator@font mod}\,\,#1)}
\makeatother

\begin{document}

\markboth{ A general Voronoi summation formula for  $GL(n,\Z)$} {Stephen D. Miller and Wilfried Schmid}

\title{\bf A general Voronoi summation formula for  $GL(n,\Z)$}
\author{Stephen D. Miller\thanks{Partially supported
by NSF grant DMS-0601009} \ and Wilfried Schmid\thanks{Partially supported by DARPA grant HR0011-04-1-0031 and NSF
grant DMS-0500922}}
\date{November 29, 2009}
\maketitle

\setcounter{page}{1}

\begin{abstract}\vskip 3mm\footnotesize
\noindent In \cite{voronoi} we derived an analogue of the classical Voronoi summation formula for automorphic forms
on $GL(3)$, by using the theory of automorphic distributions.  The purpose of the present paper is to apply this
theory to derive the analogous formulas for $GL(n)$.

\vskip 4.5mm


\end{abstract}

\vskip 12mm
\begin{flushright}
\emph{Dedicated to Shing-Tung Yau on the occasion of his 60th birthday}
\end{flushright}

\section{Introduction}\label{sec:intro}

The Voronoi summation formulas for $GL(2)$ and $GL(3)$ have had numerous applications to problems in analytic
number theory, perhaps most notably to recent subconvexity results (e.g. \cite{BHM,DFI,KMV,li,sarnak}).  The
formulas provide an identity between sums of the form
\begin{equation}\label{voronoiabstractform}
    \sum_{m\,\in\,\Z} \ a_m \, e^{2\pi i m \a}\, f(m) \ \ = \ \ \sum_{m\,\in\,\Z} \ \tilde{a}_m \, S(m,\a)\, F(m)\,
\end{equation}
where $a_m$ are Fourier coefficients of the automorphic form, $\a\in \Q$, $S(m,\a)$ an exponential sum, and $f$,
$F$ a pair of test functions related by an integral transformation.  Indeed, such a rubric covers the Poisson
summation formula, itself a cornerstone tool in analytic number theory. For $GL(2)$ the exponential sum is a single
exponential, whereas for $GL(3)$ it is a Kloosterman sum. One way to prove the $GL(2)$ formula is to use Mellin
inversion of the functional equation of the standard $L$-function with twists. An analytic variant of this method,
carried out by Duke-Iwaniec \cite{DuIwa}, involves an $n-1$-dimensional hyperkloosterman sum. That could be
regarded as predicting the appearance of a hyperkloosterman sum in the Voronoi formula for $GL(n)$. However, the
Mellin-inversion approach quickly runs into computational difficulties.

The argument that we follow here closely follows the one in \cite{voronoi}, except that it is done for general
$GL(n)$, and not just $n=3$ as was the case there.  In particular, all of the analytic details necessary to justify
that formula are essentially covered in \cite{voronoi}, though the formal aspects of the computation~-- while
philosophically identical~-- are more involved, of course.  We first execute the computation in classical terms,
entirely analogously to the $GL(3)$ argument in \cite{voronoi}, and then later perform the computation adelically.
The latter has the advantage of providing a formula for general congruence groups.

Fourier coefficients of automorphic forms on $GL(n)$ are indexed by $(n-1)$-tuples
$k=(k_1,\ldots,k_{n-1})\in\Z^{n-1}$; for a cusp form, the coefficients vanish unless $k\in (\Z_{\neq 0})^{n-1}$.
For full level Hecke cusp forms, the coefficients are uniquely determined by the $L$-function coefficients
$\{a_{(1,1,\ldots,1,k_{n-1})}\}$, or dually, by the coefficients $\{a_{(k_1,1,1,\ldots,1)}\}$.  That is a
consequence of fairly complicated Hecke relations. However, the Voronoi formulas we state hold even for non-Hecke
eigenforms, and our proof does not require the Hecke property. When the full-level assumption is dropped, there are
more Fourier coefficients to take into account.  For $GL(2)$ this is explained by the Jacquet-Langlands and
Atkin-Lehner theory, but in the absence of a satisfactory theory of this type for $GL(n)$, one cannot at present
pin down the general Fourier coefficients for $GL(n)$ in terms of the $L$-function coefficients. Thus we state the
formulas for full level, although our second, adelic proof produces a general formula once one has further
information of Atkin-Lehner type.

We say that $(\l,\d)\in\C^n\times (\Z/2\Z)^n$ is the representation parameter of a cusp form on $GL(n)$ if its
archimedean component embeds\begin{footnote}{The Casselman embedding theorem \cite{Casselman:1980} guarantees that
such an embedding exists.}\end{footnote} into the principal series representation
\begin{equation}
\label{Vld}
\begin{aligned}
    V_{\l,\d} \ \ &= \ \ \bigg \{\ f:GL(n,\R)\rightarrow\C \ \ \bigg|
    \\
    & \ \ \ \ f\!\(g\(\begin{smallmatrix} a_1 & 0 & 0 \\ \star & \ddots & 0 \\ \star & \star & a_n \\ \end{smallmatrix}\)\)=f(g)\cdot\!\!\prod_{1 \le j\le n}\(|a_j|^{\f{n+1}{2}-j-\l_j}\sgn(a_j)^{\d_j} \)\bigg\}\,.
\end{aligned}
\end{equation}
We do not assume that the archimedean component is a full principal series representation; in particular, it need
not belong to the spherical principal series. The parameter $(\l,\d)\in\C^n\times (\Z/2\Z)^n$ may not be uniquely
determined by the archimedean component:~indeed, when the archimedean component is an irreducible principal series
representation, the $\l_j\,,\d_j$ can be (simultaneously) freely permuted, a fact which provides useful flexibility
in the range of test functions allowed in the Voronoi formula. Except for that flexibility, our formula is
independent of the parameters up to permutation, as will be clear from its statement below.

Another ingredient of the Voronoi formula is the integral transform relating the two test functions; in the
prototypical example of the Poisson summation formula the Fourier transform plays that role.  We shall give two
descriptions here, entirely analogous to those in \cite{inforder} and \cite{voronoi}.  The first is more concise
but somewhat symbolic, in that it needs to be interpreted carefully to have meaning.  Let $e(u)=e^{2\pi i u}$, and
suppose that $f\in |x|^{\l_n}\sgn(x)^{\d_n}\Sch(\R)$, where $\Sch(\R)$ denotes the space of Schwartz functions on
the real line. The transform $F$ of $f$ is then defined as
\begin{equation}
\label{Ffromf}
    F(y) \ \ = \ \ \int_{\R^n}f\(\f{x_1\cdots x_n}{y}\)\cdot \prod_{1\leq j\le n}\big(e(-x_j)|x_j|^{-\l_j}\sgn(x_j)^{\d_j}dx_j\big),
\end{equation}
with an inversion formula that involves replacing the $\l_j$ with $-\l_j$.  Evidently this formula can be regarded
as a generalization of the Fourier transform. For details on how the formula needs to be interpreted see
\cite[\S5]{voronoi}.

The second, equivalent description of the transform $f\mapsto F$ relates the signed Mellin transform
\begin{equation}
\label{signedmell}
    M_\d f(s) \ =_{\text{\,def}} \ \int_\R f(x)|x|^{s-1}\sgn(x)^\d\,dx \qquad \big(\, \d \in \Z/2\Z\,\big)
\end{equation}
of $F(s)$ to that of $f(-s)$\,:
\begin{equation}
\label{MfMF}
   M_\d F(s) \ \ = \ \  (-1)^{n\d} \, \bigg(\, \prod_{1\leq j\le
n}G_{\d_j+\d}(s-\l_j+1)\bigg) M_\d f(-s)\,;
\end{equation}
here $G_\d(s)$, with $\d\in \Z/2\Z$, denotes the Gamma factor
\begin{equation}
\label{Gds}
    G_\d(s) \ \ = \ \ \begin{cases}2\,(2\pi)^{-s}\,\Gamma(s)\cos(\textstyle\f{\pi s}2)\ \ \ &\text{if}\ \ \d = 0 \\ 2\,i\,(2\pi)^{-s}\,\Gamma(s)\sin(\textstyle\f{\pi s}2)\ \ \ &\text{if}\ \ \d = 1\end{cases}
\end{equation} that was introduced in \cite[\S4]{inforder}. One can reconstruct the odd and even components of $F$ from the formula (\ref{MfMF}) by means of the signed Mellin inversion formula
\begin{equation}
\label{signedmellinv}
\begin{aligned}
    &\f {\big( f(x) + (-1)^\d f(-x)\big)}2 \ \ =
     \\
     &\qquad\qquad = \ \ \f{(\sgn x)^\d}{4\pi i }\! \int_{\Re{s}=s_0} \!\! M_\d f(s)|x|^{-s}\,ds \ \ \ \ \big(\,\Re{s_0}\gg 0\,\big)\,.
\end{aligned}
\end{equation}
This latter description of the transform $f \mapsto F$ is less suggestive than (\ref{Ffromf}), but more useful in
applications. One can argue as in \cite[\S4]{voronoi} that
\begin{equation}
\label{MfMF1} f\in |x|^{\l_n}\sgn(x)^{\d_n}\Sch(\R)\ \ \Longrightarrow \ \ F\in {\sum}_{1\leq j\le
n}\,|x|^{-\l_j}\sgn(x)^{\d_j}\Sch(\R)\,.
\end{equation}
In the singular cases where two or more of the $\l_j$ differ by an integer, the above formula must be interpreted
as including powers of $\log|x|$ -- see \cite[\S6]{inforder} for details.  It is important to note that this only
affects the asymptotics of the functions $f$ and $F$ near zero, a point at which they are never directly evaluated
in our formula below. We had remarked that the components $(\l_j,\d_j)$ of the parameter $(\l,\d)$ are freely
permutable when the archimedean representation is an irreducible principal series representation. In that situation
we can replace $|x|^{\l_n}\sgn(x)^{\d_n}\Sch(\R)$ on the left hand side of (\ref{MfMF1}) by
$|x|^{\l_j}\sgn(x)^{\d_j}\Sch(\R)$, and then sum over $j$. That makes the hypotheses on the test functions $f$ and
$F$ completely symmetric.

The $GL(n)$ summation formula requires one more ingredient, the $(n-1)$-dimensional hyperkloosterman sum
\begin{equation}
\label{hypkl}
\begin{aligned}
    &S(a,b\,;q,c,d)
    \ \ =
    \\
    & \sum_{x_j\in(\Z/\f{c_1\cdots c_jq}{d_1\cdots d_j}\Z)^*, \text{~for}\,
     j \le n-2}
     \!\!\!\!\!\!\!\!\!\!\!\!\!\!\!\!\!e\big(\textstyle
    \f{d_1x_1a}{q} \, + \, \f{d_2x_2\overline{x_1}}{\f{c_1q}{d_1}}
    \,+ \, \cdots \,+ \, \f{d_{n-2}x_{n-2}\overline{x_{n-3}}}{\f{c_1\cdots c_{n-3}q}{d_1\cdots d_{n-3}}} \,+\,\f{b\overline{x_{n-2}}}{\f{c_1\cdots c_{n-2}q}{d_1\cdots d_{n-2}}}
    \,\big)\,,
\end{aligned}
\end{equation}
where $e(u)$ is shorthand for $e^{2\pi i u}$ and $\overline{x_j}$, for $x_j\in (\Z/\f{qc_1\cdots c_j}{d_1\cdots
d_j}\Z)^*$, denotes the reciprocal of $x_j$ modulo $m_1\cdots m_j$\,.  The sum is used only when $d_1\cdots d_j$
divides $c_1\cdots c_jq$ for each $j\le n-2$.

\begin{thm}\label{mainthm}
Under the assumptions on a cusp form for $GL(n,\Z)$ stated above, with fixed values of $c_1,\ldots,c_{n-2}\in
\Z_{\neq 0}$ and relatively prime $a,q \in \Z$,
\[
\begin{aligned}
    &{\sum}_{r\neq 0}\, a_{c_{n-2},\ldots,c_1,r}\,\,e\big(-\textstyle{\f{ra}{q}}\big)\, f(r) \ \ =
    \\
    &\qquad = \ \ |q| \sum_{d_1|q c_1}\, \sum_{d_2|\f{qc_1c_2}{d_1}} \dots \!\! \sum_{d_{n-2}|\f{q c_1 \dots c_{n-2}}{d_1 \dots d_{n-3}}}\, \sum_{r\neq 0}\, {\textstyle\f{a_{r,d_{n-2},\ldots,d_1}}{|rd_1\cdots d_{n-2}|}}\ \times
    \\
    &\qquad\qquad\qquad\qquad\qquad \times \
    S(r,\bar{a};q,c,d)
    \ F\!\(\textstyle\f{r \, d_{n-2}^2\, d_{n-3}^3\cdots d_1^{n-1}}{q^n\, c_{n-2}\, c_{n-3}^2\cdots c_1^{n-2}}\)\,.
\end{aligned}
\]
\end{thm}

A Voronoi summation formula for $GL(n)$ appears in Goldfeld-Li \cite{gli1,gli2}. Our formula precedes the
Goldfeld-Li formula; see \cite{gli3}. It is also more general in two respects: it applies not only to spherical
principal series representations, and involves summation over the Fourier coefficients $a_{c_{n-1},\ldots,c_1,r}$
for arbitrary nonzero $c_1,\, c_2,\, \dots ,\,c_{n-2}\,$, not with $c_1 = c_2 = \dots = c_{n-2} = 1$ as in
\cite{gli1,gli2}. The freedom to specify arbitrary non-zero $c_j$ is potentially important; in the case of $GL(3)$,
for example, Li's convexity breaking result \cite{li} crucially depends on this additional freedom. Theorem
(\ref{mainthm}) applies to cusp forms. Ivic \cite{ivic} proves and uses a Voronoi formula for multiple divisor
functions, which corresponds to non-cusp forms. His proof uses Poisson summation. In fact, Voronoi formulas for
noncusp forms can always be derived using the formulas for cusp forms on smaller groups, though this procedure may
be complicated because of their nonzero constant terms.

One application of our formula is to give a new proof of the functional equation for the standard $L$-function of a
cusp form for $GL(n,\Z)$, and more generally those twisted by Dirichlet characters.  This was carried out in
\cite{voronoi,inforder} for $n\le 3$ by a general argument, an argument that also applies to our situation here, as
we shall argue presently. The key point is that our Voronoi formula can be applied to the test functions
\begin{equation}
\label{fpower}
 f(x) \ = \ |x|^{-s} \sgn(x)^\eta \, , \ \ \ F(x) \ = \ {\prod}_{j\,=\,1}^n\,G_{\eta+\d_j}(s+\l_j)\i \, \, |x|^{s} \sgn(x)^\eta \, ,
\end{equation}
even though they do not satisfy the hypotheses of the theorem as stated. They can, because of a deeper analytic
fact:~the relevant automorphic distributions in section~\ref{sec:background} vanish to infinite order at 0 and
$\infty$ in the sense of \cite{inforder} (see proposition~\ref{fliprelationsprop} below). The proof of this
vanishing to infinite order involves arguments related to the main mechanism of proof in the paper \cite{chm} of
Casselman-Hecht-Mili{\v c}i{\'c}. In effect, certain Fourier components of automorphic distributions are completely
determined, as distributions, by their restriction to an open Schubert cell, just as is the case for Whittaker
distributions in their paper. We could have relied on their result, but have chosen to develop the necessary tools
ourselves in section \ref{sec:inforder}, for the following reasons. Our tools are both stronger and more concrete
than the corresponding arguments in \cite{chm}, and we expect to use them in the future, in situations not covered
by \cite{chm}. In fact, we prove a slightly more general version of an important lemma -- \lemref{inforderlem3}
below -- than we need for the proof of the Voronoi formula. This more general version, which we also expect to use
in the future, cannot be reduced to the Whittaker case.

We now consider the formula in \thmref{mainthm} with the choice of functions (\ref{fpower}), taking $c_j$, $d_j$,
and $q$ to be positive~-- as we may, because the coefficient $a_k$ is insensitive to the signs of the entries of
$k$. That results in a general functional equation for additively twisted $L$-functions:
\begin{equation}
\label{addtwistlfn}
\aligned &\sum_{r\,\neq\,0}a_{c_{n-2},\ldots,c_1,r}\,e(-\smallf{ra}{q})\,|r|^{-s}\,\sgn(r)^\eta \ \
   = \ \ \prod_{j\,=\,1}^n\,G_{\eta+\d_j}(s+\l_j)\i\ \ \ \times
\\
&\qquad\qquad \times \ |q|^{1-n s}\sum_{\srel{d_j|\f{qc_1\cdots c_j}{d_1\cdots d_{j-1}}}{\text{for all }j\,\le\,n-2}}
\sum_{r\,\neq\,0} a_{r,d_{n-2},\ldots,d_1}\, S(r,\bar{a};q,c,d)\ \times
\\
&\qquad\qquad\qquad \times \ |r|^{s-1}\,\sgn(r)^\eta \({\prod}_{j\,=\,1}^{n-2}|c_j|^{-(n-1-j)s}\,|d_j|^{(n-j)s-1}\) ,
\endaligned
\end{equation}
or, in the special case of $c_1=\cdots=c_{n-2}=1$, more simply
\begin{equation}
\label{simpleraddtwistlfn}
\aligned & \sum_{r\,\neq\,0}\, a_{1,\ldots,1,r}\,e(-\smallf{ra}{q})\,|r|^{-s}\,\sgn(r)^\eta \ \ = \ \
\prod_{j\,=\,1}^n\,G_{\eta+\d_j}(s+\l_j)\i\ \ \times
\\
&\qquad \times \ \ |q|^{1-n s}\, \sum_{\srel{d_j|\f{q}{d_1\cdots d_{j-1}}}{\text{for all
}j\,\le\,n-2}} \sum_{r\,\neq\,0}a_{r,d_{n-2},\ldots,d_1}\ \ \times \\
& \qquad\qquad \ \ \times\  \   S(r,\bar{a};q,(1,\ldots,1),d)\, |r|^{s-1}\,\sgn(r)^\eta
{\prod}_{j\,=\,1}^{n-2}\,|d_j|^{(n-j)s-1}\, .
\endaligned
\end{equation}
Now let $\chi$ be a primitive Dirichlet character modulo $q$, such that $\chi(-1)=(-1)^\eta\,$. A basic identity
asserts
\begin{equation}
\label{gausssum} {\sum}_{a\in \Z/q\Z} \ \chi(a)\,e(-\smallf{ra}{q}) \ \ = \ \ \left\{\
\begin{array}{ll}
\ 0\ \ & \hbox{~~if $(r,q)>1$} \\
\ \chi(-r)\i\, g_\chi \ \ & \hbox{~~if $(r,q)=1$\,,}
\end{array}
\right.
\end{equation}
where $g_\chi$ denotes the Gauss sum $\sum_{a\in \Z/q\Z} \chi(a)e(\smallf{a}{q})$ for $\chi$. Thus, summing
(\ref{simpleraddtwistlfn}) over the residue classes modulo $q$, one derives the functional equation of the standard
$L$-function twisted by the Dirichlet character $\tau$\,,
\begin{equation}
\label{twistedfe1} \!\sum_{r\,=\,1}^\infty\, a_{r,1,\ldots,1}\,\overline{\chi(r)}\,r^{s-1}  \ = \ q^{ns}\,g_\chi^{-n}
\prod_{j=1}^n\,G_{\eta+\d_j}(s+\l_j) \sum_{r\,=\,1}^\infty \, a_{1,\ldots,1,r}\,\chi(r)\,r^{-s}\,.
\end{equation}
\bigskip

It is a pleasure to express our gratitude to James Cogdell and Herve Jacquet. Both of them helped us with detailed
information about the state of the  literature relevant to the appendix.

\section{Automorphic Distributions}\label{sec:background}

We derive the Voronoi summation formula not from the datum of an automorphic form, but from the essentially
equivalent datum of an automorphic distribution. For the connection between the two we refer the reader to
\cite{voronoi}.

We kept the definition (\ref{Vld}) of the principal series representation with para\-meter $(\l,\d)$ purposely
vague: when the functions $f$ in the definition are required to be smooth, one denotes the resulting space by
$V^{\infty}_{\l,\d}$, and when these ``functions" are only required to be distributions, one obtains
$V^{-\infty}_{\l,\d}$, the larger space of distribution vectors. Of course these are not the only choices of a
topology for $V_{\l,\d}$. In all cases the group $GL(n,\R)$ acts by left translation. By definition an automorphic
distribution for an arithmetic subgroup $\G \subset GL(n,\Q)$, with representation parameter $(\l,\d)$, is a
$\G$-invariant vector $\tau $ in $V^{-\infty}_{\l,\d}$ -- in other words,
\begin{equation}
\label{autod}
\begin{aligned}
    &\tau \ \in \ C^{-\infty}\big(GL(n,\R)\big) \ \ \ \text{such that}
    \\
    & \ \ \ \ \ \ \tau\ \(\g\,g\(\begin{smallmatrix} a_1 & 0 & 0 \\ \star & \ddots & 0 \\ \star & \star & a_n \\ \end{smallmatrix}\)\)\ =\ \tau(g)\cdot \! {\prod}_{1\le j\le n}\(|a_j|^{\f{n+1}{2}-j-\l_j}\sgn(a_j)^{\d_j} \)
\end{aligned}
\end{equation}
for all $\g \in \G$ and $a_1,\,a_2,\,\dots,\,a_n \in \R^*$. In dealing with distributions, we adopt the same
convention as in our other papers: ``distributions transform like functions"~-- i.e., they are naturally dual to
smooth, compactly supported measures.

Let $N(\R)\subset G(\R)=_{\text{def}}GL(n,\R)$ denote the subgroup of unipotent upper triangular matrices, and
$N'(\R)=[N(\R),N(\R)]$, $N''(\R)=[N(\R),N'(\R)]$ its first two derived subgroups. In the classical approach we
shall work exclusively at full level; we therefore change notation from $\G$ to $G(\Z)$. Analogously we let
$N(\Z)$, $N'(\Z)$, $N''(\Z)$ denote the groups of integral points in $N(\R)$ and its derived subgroups. Since we
are working at full level, no nonzero automorphic distributions can exist unless $\d_1+ \cdots + \d_n\equiv
0\imod2$.  We shall also assume $\l_1+\cdots +\l_n=0$; this can be arranged by multiplying $\tau$ by an appropriate
character of the center. For emphasis,
\begin{equation}\label{sumtozero}
    \sum_{i\,=\,1}^n \l_i \ \ = \ \ 0 \ \ \ \ \text{and} \ \ \ \sum_{i\,=\,1}^n
    \d_i \ \ \equiv \ \ 0 \imod 2\,.
\end{equation}
The first of these assumptions is not necessary, but helps to simplify formulas and arguments.

Our computations involve integrating translates of $\tau$, sometimes multiplied by smooth functions, over compact
sets. This makes sense: let $\ell(g)$ denote the left $g$-translate of $\tau$, formally
\begin{equation}
\label{ellofg} (\ell(g)\tau)(g_1)\ \ = \ \ \tau(g^{-1}g_1)\,;
\end{equation}
then $g\mapsto \ell(g)\tau$ is a smooth function on $G(\R)$ with values in the closed subspace
$V^{-\infty}_{\l,\d}$ of $C^{\infty}(\R)$, from which it inherits the structure of complete, locally convex,
Hausdorff topological vector space. One can integrate measurable functions with values in such a topological vector
space over any finite measure space. For example, since $\tau$ is $G(\Z)$-invariant under left translation, $n
\mapsto \ell(n)\tau$ induces smooth, $V^{-\infty}_{\l,\d}$-valued functions on the compact spaces $N'(\R)/N'(\Z)$,
$N''(\R)/N''(\Z)$. Thus
\begin{equation}
\label{tau'} \tau' \ = \ \int_{N'(\R)/N'(\Z)}\ell(n)\tau\,dn \ \ \ \ \ \text{and} \ \ \ \ \ \tau'' \ = \
\int_{N''(\R)/N''(\Z)}\ell(n)\tau\,dn
\end{equation}
are well defined vectors in $V^{-\infty}_{\l,\d}$, by construction invariant under, respectively, $N'(\R)$ and
$N''(\R)$\,:
\begin{equation}
\label{tau'1} \tau' \in \big(V^{-\infty}_{\l,\d}\big)^{N'(\R)}\,, \ \ \ \ \ \tau'' \in
\big(V^{-\infty}_{\l,\d}\big)^{N''(\R)}\,.
\end{equation}
Both are also $N(\Z)$-invariant,
\begin{equation}
\label{tau'2} \tau'\,,\,\, \tau'' \,\ \in \,\ \big(V^{-\infty}_{\l,\d}\big)^{N(\Z)}\,,
\end{equation}
since $N(\Z)$ leaves $\tau$ invariant, normalizes all of the four groups $N'(\R)$, $N'(\Z)$, $N''(\R)$, $N''(\Z)$,
and preserves the measures on $N'(\R)$ and $N''(\R)$.

Let $B_-(\R)\subset G(\R)$ denote the subgroup of lower triangular matrices. In view of (\ref{autod}), every
distribution in $V^{-\infty}_{\l,\d} \subset C^{-\infty}(N(\R))$ behaves in a $C^\infty$ manner under right
translation by elements of $B_-(\R)$. Since $N(\R)\cdot B_-(\R)$ is open in $G(\R)$, we can restrict such
distributions from $G(\R)$ to the subgroup $N(\R)$ \cite{voronoi}. In particular this applies to $\tau'$; we shall
refer to its restriction to $N(\R)$ as $\tau_{\text{abelian}}$\,. In view of (\rangeref{tau'1}{tau'2}), and because
$N'(\R) \subset N(\R)$ is normal,
\begin{equation}
\label{tauabelian} \tau_{\text{abelian}}\ =_{\text{def}} \ \tau' \big|_{N(\R)}\ \in \ C^{-\infty}\big(N(\Z)\backslash
N(\R)/N'(\R)\big)\,.
\end{equation}
The entries on the first superdiagonal provide coordinates $(x_1,\,\dots,\,x_{n-1})$ for $N'(\R)\backslash N(\R)
\cong \R^{n-1}$. Under this identification, the image of $N(\Z)$ in $N'(\R)\backslash N(\R)$ corresponds to
$\Z^{n-1}$. Thus $\tau_{\text{abelian}}$\,, as distribution on $N(\Z)\backslash N(\R)/N'(\R)\cong
\R^{n-1}/\Z^{n-1}$,\,\ has the Fourier expansion
\begin{equation}
\label{tauabelexp} \tau_{\text{abelian}}(x)\ \ = \ \ {\sum}_{k\in(\Z_{\neq 0})^{n-1}}\ c_k\,e(k\cdot x)\,.
\end{equation}
Here $e(u) = e^{2\pi i u}$ as before, and $k\cdot x$ stands for the sum $k_1x_1+\cdots + k_{n-1}x_{n-1}$. Note that
the summation does not involve terms for which at least one of the $k_j$ equals zero -- this reflects the
assumption of cuspidality; cf. \cite{voronoi}.

When there are several choices of the representation parameter, as is usually the case, the coefficients $c_k$ in
the expansion (\ref{tauabelexp}) do depend on $(\l,\d)$. However, they are related to the Fourier coefficients
$a_k$ of any automorphic form associated to $\tau$ by the formula
\begin{equation}
\label{ctoa} a_{k_1,k_2,\dots,k_{n-1}}\, = \, \prod_{j=1}^{n-1}\left( (\sgn k_j)^{\d_1 + \d_2 + \dots + \d_j}
\,|k_j|^{\l_1 + \l_2 + \dots + \l_j}\right) c_{k_1,k_2,\dots,k_{n-1}}\,.
\end{equation}
The $a_k$ are independent of $(\l,\d)$ -- in fact, they coincide with the Hecke eigenvalues in the case of full
level, as we are assuming. Alternatively one can show that the $a_k$ do not depend on $(\l,\d)$ by calculating the
effect of the intertwining operators between the different principal series representations into which our
automorphic representation can be embedded.

For $x\in \R^{n-1}$ and $y\in \R^{n-2}$, define
\begin{equation}\label{nxy}
    n_{x,y} \ \ =  \ \  \(\begin{smallmatrix}
                                           1 & x_1 & y_1 & 0 & 0 & 0 \\
                                            0 & 1 & x_2 & y_2 & 0 & 0 \\
                                            0 & 0 & \ddots  & \ddots & \ddots & 0 \\
                                            0 & 0 & 0 & 1 & x_{n-2} & y_{n-2} \\
                                            0 & 0 & 0 & 0 & 1 & x_{n-1} \\
                                            0 & 0 & 0 & 0 & 0 & 1
                                          \end{smallmatrix}\).
\end{equation}
Every element of $N$ can be uniquely decomposed as either $n''n_{x,y}$ or as $n_{x,y}n''$, for some $x \in
\R^{n-1}$, $y \in \R^{n-2}$, and $n''\in N''(\R)$; only the factor $n''$ depends on which order is chosen. Thus
\begin{equation}
\label{tau'3} \tau' \ \ = \ \ \int_{(\R/\Z)^{n-2}}\ell(n_{0,y})\tau''    \, dy_1\cdots dy_{n-2}\,.
\end{equation}
Corresponding to the datum $(j,m,k)$, with $1\le j \le n-2$, $m\in \Z_{\neq 0}$, and $k\in \Z^{n-1}$, we define
\begin{equation}
\label{Rjkmtau}
\begin{aligned}
&R_{j,m,k}\tau \ = \ \int_{(\R/\Z)^{n-2}}\int_{\left\{\begin{smallmatrix}{x\in (\R/\Z)^{n-1}} \\
{x_{j+1}\,=\,0 \ \ \ }\end{smallmatrix}\right\}} e(k\cdot x\,+\, my_j)\,\, \ell(n_{x,y})\tau''\, dx\, dy\,,
\\
&S_{j,m,k}\tau \ = \ \int_{(\R/\Z)^{n-2}}\int_{\left\{\begin{smallmatrix}{x\in (\R/\Z)^{n-1}}
\\ {x_j\,=\,0\, \ \ \ \ \ \ }\end{smallmatrix}\right\}} e(k\cdot x\,+\, my_j)\,\, \ell(n_{x,y})\tau''\, dx\, dy\,,
\end{aligned}
\end{equation}
and we also define $R_{0,m,k}\tau$ and $S_{n-1,m,k}\tau$, but only corresponding to $m=1$\,:
\begin{equation}
\label{Rjk1tau}
\begin{aligned}
&R_{0,1,k}\tau \ = \ \int_{(\R/\Z)^{n-2}}\int_{\left\{\begin{smallmatrix}{x\in (\R/\Z)^{n-1}} \\
{x_{1}\,=\,0\ \ \ \ \ \ \ }\end{smallmatrix}\right\}} e(k\cdot x)\,\, \ell(n_{x,y})\tau''\, dx\, dy\,,
\\
&S_{n-1,1,k}\tau \ = \ \int_{(\R/\Z)^{n-2}}\int_{\left\{\begin{smallmatrix}{x\in (\R/\Z)^{n-1}}
\\ {x_{n-1}\,=\,0 \ \ \ \ }\end{smallmatrix}\right\}} e(k\cdot x)\,\, \ell(n_{x,y})\tau''\, dx\, dy\,.
\end{aligned}
\end{equation}
In these equations $dx$ is shorthand for $dx_1 \dots dx_{j}\, dx_{j+2}\dots dx_{n-1}$ in the case of
$R_{j,m,k}\tau$, and for $dx_1 \dots dx_{j-1}\, dx_{j+1}\dots dx_{n-1}$ in the case of $S_{j,m,k}\tau$; $dy$ stands
for $dy_1 \dots dy_{n-2}\,$ in all cases.

The $\,R_{j,m,k}\tau\,$ and $\,S_{j,m,k}\tau\,$ are integrals of continuous $\,V^{-\infty}_{\l,\d}$-valued
functions over tori, hence
\begin{equation}
\label{Rjkmtau3} R_{j,m,k}\tau\,,\ S_{j,m,k}\tau \ \in \ V^{-\infty}_{\l,\d}\,.
\end{equation}
Large subgroups of $\,N(\Z)\,$ leave them invariant -- see \lemref{RSlem} below. The $\,G(\Z)$-invariance of
$\,\tau\,$ implies certain relations among the $\,R_{j,m,k}\tau\,$ and $\,S_{j,m,k}\tau\,$. To state them, we need
to consider the embeddings
\begin{equation}
\label{sl2emb} \Phi_j \ : \ SL(2,\R)\ \ \hookrightarrow \ \ G(\R)\ = \ GL(n,\R) \qquad\ \ \  (\, 1 \le j \le n-1\,)
\end{equation}
into the $2 \times 2$ diagonal blocks with ``vertices" $(j,j)$, $(j+1,j+1)$. On the infinitesimal level this means
\begin{equation}
\label{sl2emb2} {\Phi_j}_* \(\begin{smallmatrix} 0 & 1 \\ 0 & 0
\end{smallmatrix} \)\, = \, E_{j,j+1}\,,\ \, {\Phi_j}_* \!\(\begin{smallmatrix} 0 & 0 \\ 1 & 0 \end{smallmatrix} \)\, = \, E_{j+1,j}\,,\, \ {\Phi_j}_* \! \(\begin{smallmatrix} 1 &
0 \\ 0 & -1
\end{smallmatrix} \)\, = \, E_{j,j} -  E_{j+1,j+1}\,,
\end{equation}
with $\,E_{r,s}\,$= matrix with $(r,s)$ entry equal to one, and the other entries equal to zero. The image of
$\Phi_j$ normalizes
\begin{equation}
\label{Nj}
\begin{aligned}
N_j(\R) \ \ = \ \ \text{subgroup of}\,\ G(\R)\,\ \text{whose Lie algebra is spanned}
\\
\text{by} \ \ \{\, E_{r,s}\, \big| \, 1\le r < s \le n\,, \ (r,s) \neq (j,j+1) \,\}\,.
\end{aligned}
\end{equation}
In fact, $\,N_j(\R)\,$ is the unipotent radical of a parabolic which contains $\,\operatorname{Im} \Phi_j\,$ as the
semisimple part of its Levi component.

\begin{lem}
\label{RSlem} Both $\,R_{0,1,k}\tau\,$ and $\,S_{n-1,1,k}\tau\,$ are $\,N(\Z)$-invariant, and
\[
R_{j,m,k}\tau \ \in \ \big( V^{-\infty}_{\l,\d}\big)^{N_{j+1}(\Z)}\,,\ \ \ \ S_{j,m,k}\tau \ \in \ \big(
V^{-\infty}_{\l,\d}\big)^{N_j(\Z)} \qquad (\, 1 \le j \le n-2\,).
\]
For $\left(\begin{smallmatrix} a & b \\ c & d
\end{smallmatrix}\right) \in SL(2,\Z)\,$,
\[
\begin{aligned}
&d\,m = c\,k_2\ \ \ \Longrightarrow \ \ \ \ell\big(\Phi_1\left(\begin{smallmatrix} d & -b \\ -c & a \end{smallmatrix}
\right)\big) R_{0,1,\widetilde k}\,\tau\ = \ S_{1,m,k}\,\tau \,,
\\
&\qquad\qquad \ \ \text{with}\,\ \ \widetilde k_{2} = a\,k_{2} - b\,m\,, \, \ \widetilde k_i = k_i \ \
\text{otherwise\,;}
\end{aligned}
\]
similarly
\[
\begin{aligned}
&d\,m = -c\,k_{n-2}\ \ \ \Longrightarrow \ \ \ \ell\big( \Phi_{n-1}\left(\begin{smallmatrix} d & -b \\ -c & a
\end{smallmatrix}\right)\big) S_{n-1,1,\widetilde k}\,\tau \ = \ R_{n-2,m, k} \tau \,,
\\
&\qquad\qquad  \text{with}\ \ \widetilde k_{n-2} = a\,k_{n-2} + b\,m\,, \, \ \widetilde k_i = k_i \ \
\text{otherwise\,;}
\end{aligned}
\]
and, for $2 \le j \le n-2$\,,
\[
\begin{aligned}
&a\,m = - c\,k_{j+1}\ \ \ \Longrightarrow\ \ \ \ell\big( \Phi_{j} \left(\begin{smallmatrix} d & -b \\ -c & a
\end{smallmatrix} \right)\big) S_{j,m,k}\tau \ = \ R_{j-1,\widetilde m,\widetilde k}\,\tau \,,\, \ \text{with}
\\
&\qquad \widetilde m = - c\,k_{j-1}\,,\,\ \widetilde k_{j-1} =  d\, k_{j-1}\,, \, \ \widetilde k_{j+1} =  d\,k_{j+1}+
b\,m\,, \, \ \widetilde k_i = k_i \ \ \text{otherwise\,.}
\end{aligned}
\]
Finally, for $1 \le j \le n-2$,
\[
\begin{aligned}
&\ell\big( \Phi_{j+1} \left(\begin{smallmatrix} 1 & 1 \\ 0 & 1 \end{smallmatrix} \right)\big) R_{j,m,k}\tau \ = \
R_{j,m,\widetilde k}\,\tau \ \ \ \text{with}\,\ \widetilde k_i =k_i + \d_{i,j}\,m\ ,
\\
&\qquad\qquad \ell\big( \Phi_j \left(\begin{smallmatrix} 1 & 1 \\ 0 & 1 \end{smallmatrix} \right)\big) S_{j,m,k}\tau \
= \ S_{j,m,\widetilde k}\,\tau \ \ \ \text{with}\,\ \widetilde k_i =k_i - \d_{i,j+1}\,m\ .
\end{aligned}
\]
\end{lem}\medskip

In the relationship between $\,S_{j,m,k}\tau\,$ and $\,R_{j-1,\widetilde m,\widetilde k}\,\tau\,$ the indices
$(\widetilde m , \widetilde k)$ and $(m,k)$ appear to play non-symmetric roles, since the former are defined in
terms of the latter. In fact, it is possible to express any $\,R_{j-1,m,k}\tau\,$ in terms of $\,S_{j,\widetilde
m,\widetilde k}\,\tau\,$, with suitably chosen $(\widetilde m,\widetilde k)$. One can see this either directly, by
inverting the map $(m,k) \mapsto (\widetilde m , \widetilde k)$, or applying the automorphism (\ref{Rjk1.2tau}),
which interchanges the roles of the $\,R_{j,m,k}\tau\,$ and the $\,S_{n-j,m,k}\tau\,$.\bigskip

\noindent {\it Proof\/} of lemma \ref{RSlem}. The passage from $\tau$ to  $R_{j-1,m,k}\tau$ and $S_{j,m,k}\tau$
involves two integrations, first over $N''(\R)/N''(\Z)$, then over $\R^{2n -4}/\Z^{2n-4}$. They can be combined
into one integration against characters of $N_j(\R)$ which are trivial on $N_j(\Z)$\,:
\begin{equation}
\label{Rjkmtau2}
\begin{aligned}
&R_{j-1,m,k}\tau \ = \ \int_{N_j(\R)/N_j(\Z)}\, \chi^R_{j-1,m,k}(n)\,\ell(n)\tau\,dn\,,
\\
&\qquad\qquad\qquad S_{j,m,k}\tau \ = \ \int_{N_j(\R)/N_j(\Z)}\, \chi^S_{j,m,k}(n)\,\ell(n)\tau\,dn\,;
\end{aligned}
\end{equation}
the characters $\,\chi^R_{j-1,m,k}\,,\,\chi^S_{j,m,k}\,:\,N_j(\R) \longrightarrow \C^*$ are determined by the
equations
\begin{equation}
\label{chijkmtau}
\begin{aligned}
&\chi^R_{j-1,m,k} \big|_{N''(\R)}\ \equiv \ 1\ ,\ \ \ \chi^S_{j,m,k} \big|_{N''(\R)}\ \equiv \ 1\,,
\\
&\qquad\qquad \text{and for $\,n_{x,y}\in N_j(\R)$\,,\,\ or equivalently $\,x_j = 0$}\,,
\\
&\chi^R_{j-1,m,k}(n_{x,y})\ \ = \ \ \begin{cases}\ e(k\cdot x + m\,y_{j-1})\ \ &\text{if}\,\ 2 \le j\le n-1 \\  \
e(k\cdot x )\ \ &\text{if}\,\ j = 1 \, \ \text{and}\,\ m=1 \ , \end{cases}
\\
&\chi^S_{j,m,k}(n_{x,y})\ \ = \ \ \begin{cases}\ e(k\cdot x + m\,y_{j})\ \ &\text{if}\,\ 1 \le j\le n-2 \\  \ e(k\cdot
x )\ \ &\text{if}\,\ j = n-1 \, \ \text{and}\,\ m=1 \ . \end{cases}
\end{aligned}
\end{equation}
The kernels of both $\,\chi^R_{j-1,m,k}\,$ and $\,\chi^S_{j,m,k}\,$ contain $\,N_j(\Z)\,$, and even all of
$\,N(\Z)\,$ in the exceptional cases of $\,\chi^R_{0,1,k}\,$ and $\,\chi^S_{n-1,1,k}\,$. That implies the initial
assertions.

The $\Phi_j$-image of $SL(2,\Z)$ lies in $G(\Z)$, it normalizes $N_j(\R)$ and $N_j(\Z)$, and conjugation by it
preserves Haar measure on $N_j(\R)$. For $\g \in SL(2,\Z)$, define
\begin{equation}
\label{chijkmtau1} A_\g \ : \ N_j(\R)\ \longrightarrow\ N_j(\R)\,,\ \ \ \ A_\g(n)\ = \ \Phi_j(\g)\, n \,
\Phi_j(\g^{-1})\,.
\end{equation}
Using the change of variables $\,n\rightsquigarrow A_\g(n)\,$, we find
\begin{equation}
\label{chijkmtau2}
\begin{aligned}
\ell\big(\Phi_j(\g^{-1})\big)\,R_{j-1,m,k}\tau \ &= \ \int_{N_j(\R)/N_j(\Z)}\,
\chi^R_{j-1,m,k}(n)\,\ell\big(A_{\g^{-1}}( n)\big)\tau\,dn
\\
&= \ \int_{N_j(\R)/N_j(\Z)}\, \chi^R_{j-1,m,k}\big(A_\g(n)\big)\,\ell(n)\tau\,dn\ ,
\end{aligned}
\end{equation}
and analogously
\begin{equation}
\label{chijkmtau3} \ell\big(\Phi_j(\g^{-1})\big)\,S_{j,m,k}\tau \ = \ \int_{N_j(\R)/N_j(\Z)}\,
\chi^S_{j,m,k}\big(A_\g(n)\big)\,\ell(n)\tau\,dn \ .
\end{equation}
Now suppose $\g = \left(\begin{smallmatrix} a & b \\ c & d
\end{smallmatrix}\right) \in SL(2,\Z)\,$, $\,(x,y) \in
\R^{n-1}\times \R^{n-2}$ with $x_j=0$, so that $n_{x,y}\in N_j(\R)$. A straightforward matrix computation shows
\begin{equation}
\label{chijkmtau4}
\begin{aligned}
&A_\g(n_{x,y})\ \equiv \ n_{\widetilde x,\widetilde y}\ \ \ \text{modulo}\ \
\big(\operatorname{Ker}\chi^R_{j-1\,,\,{\textstyle{\cdot}}\,,\,{\textstyle{\cdot}}\,}\big) \cap
\big(\operatorname{Ker}\chi^S_{j\,,\,{\textstyle{\cdot}}\,,\,{\textstyle{\cdot}}\,}\big) \,\ \ \ \ \text{with}
\\
&\ \ \ \ \widetilde x_{j-1} = d\,x_{j-1} - c\,y_{j-1}  \ ,\ \ \widetilde x_{j+1} = d\,x_{j+1} + c\,y_j \ ,\ \
\widetilde x_i = x_i \ \ \text{otherwise}\ ,
\\
&\ \ \ \ \ \ \ \ \widetilde y_{j-1} = a\,y_{j-1} - b\,x_{j-1}    \ ,\ \ \widetilde y_j =  a\,y_j + b\,x_{j+1} \ ,\ \
\widetilde y_i = y_i \ \ \text{otherwise}\ .
\end{aligned}
\end{equation}
These identities remain correct in the exceptional cases of $j=1$ and $j=n-1$ when terms with out-of-range indices
are disregarded. We conclude
\begin{equation}
\label{chijkmtau5}
\begin{aligned}
&\chi^R_{j-1,m,k}\big(A_\g(n_{x,y})\big)\ =  \  e\big(\, \textstyle{\sum_{|i-j|\ge 2}}\ k_i\,x_i \ + \ (d k_{j-1}
-bm)x_{j-1}\ +
\\
&\qquad\qquad\qquad\qquad\qquad + \ (am-ck_{j-1})y_{j-1} + c\,k_{j+1}\,y_j + dk_{j+1}x_{j+1} \big),
\\
&\chi^S_{j,m,k}\big(A_\g(n_{x,y})\big)\ =  \  e\big(\, \textstyle{\sum_{|i-j|\ge 2}}\ k_i\,x_i \ + \ dk_{j-1}x_{j-1} \
- \ ck_{j-1}\,y_{j-1}\ +
\\
&\qquad\qquad\qquad\qquad\qquad\qquad + \ (d k_{j+1} + bm)x_{j+1}\ + \ (am+ck_{j+1})y_{j} \big)\,.
\end{aligned}
\end{equation}
These identities, too, remain valid in the exceptional cases when properly interpreted. In particular,
\begin{equation}
\label{chijkmtau6}
\begin{aligned}
&a\,m = c\,k_{j-1} \ \ \ \ \Longrightarrow \ \ \ \ \chi^R_{j-1,m,k}\big(A_\g(n)\big)= \chi^S_{j,\widetilde m,
\widetilde k}(n)\ \ \text{with}\,\ \widetilde m = c\,k_{j+1}\,,
\\
&\qquad\qquad\qquad \widetilde k_{j-1} =  d\, k_{j-1} -b \, m\,, \ \ \widetilde k_{j+1} = d\,k_{j+1}\,, \ \ \widetilde
k_i = k_i \ \ \text{otherwise\,;}
\\
&a\,m = -c\,k_{j+1} \ \ \, \Longrightarrow \, \ \ \chi^S_{j,m,k}\big(A_\g(n)\big)= \chi^R_{j-1,\widetilde m, \widetilde
k}(n)\ \ \text{with}\,\ \widetilde m = - c\,k_{j-1}\,,
\\
&\qquad\qquad\qquad \widetilde k_{j-1} =  d\, k_{j-1}\,, \ \ \widetilde k_{j+1} =  d\,k_{j+1}+ b\,m\,, \ \ \widetilde
k_i = k_i \ \ \text{otherwise\,.}
\end{aligned}
\end{equation}
Once more this must be properly interpreted in the exceptional cases. The first of the five equalities in the lemma
follows from the second half of (\ref{chijkmtau6}), with $\,\g^{-1}\,$ in place of $\,\g\,$, the second from the
first half of (\ref{chijkmtau6}), again with $\,\g^{-1}\,$ in place of $\,\g\,$, and third follows directly from
(\ref{chijkmtau6}); in all three cases we also appeal to (\ref{Rjkmtau2}) and either (\ref{chijkmtau2}) or
(\ref{chijkmtau3}). For the last two equalities, we apply (\ref{chijkmtau2}), replacing $j$ with $j+1$, as well as
(\ref{chijkmtau3}) and (\ref{chijkmtau5}), in both cases with $a=d=1$, $b=-1$, $c=0$. \hfill$\square$\medskip

Recall the definition of the elementary matrices $\,E_{r,s}\,$ below (\ref{sl2emb2}). It will be convenient to use
the notation
\begin{equation}
\label{hj} h_j(t)\ = \ \Phi_j\big( \begin{smallmatrix} 1 & t \\ 0 & 1 \end{smallmatrix} \big) \ = \ \exp( t \,
E_{j,j+1}) \,.
\end{equation}
Lemma \ref{RSlem} asserts the $\,N_j(\Z)$-invariance of $\,R_{j-1,m,k}\tau\,$ and $\,S_{j,m,k}\tau\,$. Together
with $\,N_j(\Z)\,$, $\,h_j(1)\,$ generates all of $\,N(\Z)\,$. In effect, our next lemma clarifies the obstacle to
$\,N(\Z)$-invariance for $\,R_{j-1,m,k}\tau\,$ and $\,S_{j,m,k}\tau\,$.

\begin{lem}
\label{RSlem2} The $R_{j,m,k}\tau$ do not depend on $k_{j+1}$ and the $S_{j,m,k}\tau$ do not depend on $k_j$. For
$\,1 \le j \le n-2\,$,\,\ $\,\ell\big(h_{j+1}(-\frac{k_{j}}{m})\big)R_{j,m,k}\tau\,$ depends on $\,k_j\,$ only
modulo $\,m\,$, and $\,\ell \big(h_{j} (\frac{k_{j+1}}{m}) \big)S_{j,m,k}\tau\,$ depends on $\,k_{j+1}\,$ only
modulo $\,m\,$.
\end{lem}

\begin{proof}
In the integral (\rangeref{Rjkmtau}{Rjk1tau}) defining $R_{j,m,k}\tau$, the variable $x_{j+1}$ is set equal to
zero, so effectively the exponential factor does not involve $k_{j+1}$. That makes $R_{j,m,k}\tau$ independent of
$k_{j+1}$. According to lemma \ref{RSlem}, for $1 \le j \le n-2$, increasing the index $k_j$ by $m$ has the same
effect on $\,R_{j,m,k}\tau\,$ as a translation by $\,h_{j+1}(1)\,$, so
$\,\ell\big(h_{j+1}(-\frac{k_{j}}{m})\big)R_{j,m,k}\tau\,$ depends on $\,k_j\,$ only modulo $\,m\,$. The analogous
assertions about the $S_{j,m,k}\tau$ are proved the same way.
\end{proof}

The assertions of \lemref{RSlem} become more transparent when stated in terms of the re-normalized quantities
$\,\ell (h_{j+1}(-\frac{k_{j}}{m}) )R_{j,m,k}\tau\,$ and $\,\ell (h_{j} (\frac{k_{j+1}}{m}) )S_{j,m,k}\tau\,$:

\begin{lem}
\label{RSlem3} Define $\,m_j = \gcd(m,k_j)\,$, the greatest common divisor of $\,m\,$ and $\,k_j\,$, and choose
$\,\overline {k}_j\in \Z\,$ so that $\,\overline{k}_j \,k_j \equiv m_j \imod m\,$. Then
\[
\ell \left( \Phi_1 \big( \begin{smallmatrix} 0 \, & \, m_2/m \\  -m/m_2 \, & \, \overline k_2 \end{smallmatrix}
\big)\right) R_{0,1,(k_1,m_2,k_3,\dots,k_{n-1})}\tau \ \ = \ \ \ell\big(h_1(\smallf{k_2}m)\big)S_{1,m,k}\tau \,,
\]
and
\[
\begin{aligned}
&\ell\left( \Phi_{n-1}\big(\begin{smallmatrix} 0 & -m_{n-2}/m \\ m/m_{n-2} & \overline k_{n-2}
\end{smallmatrix}\big)\right) S_{n-1,1,(k_1,\dots,k_{n-3},m_{n-2},k_{n-1})}\tau \ \ =
\\
&\qquad\qquad\qquad\qquad\qquad\qquad\qquad = \ \ \ell \big(h_{n-1}(-\smallf{k_{n-2}}{m}) \big)R_{n-2, m, k}\,\tau \,.
\end{aligned}
\]Finally, for $2 \le j \le n-2$\,,
\[
\begin{aligned}
&\ell\left( \Phi_j\big(\begin{smallmatrix} 0 & -m_{j+1}/m \\ m/m_{j+1} & 0 \end{smallmatrix}\big)\right) \left( \ell
\big(h_{j} (\smallf{k_{j+1}}{m}) \big) S_{j,m,k}\tau \right) \ = \ \,\ell \big(h_j(-\smallf{\widetilde
k_{j-1}}{\widetilde m}) \big) R_{j-1,\widetilde m,\widetilde k}\,\tau \,,
\\
&\text{with}\, \ \widetilde m  =  m\,k_{j-1}/m_{j+1}\,, \ \widetilde k_{j-1} = k_{j-1}\,\overline{k}_{j+1}\,, \
\widetilde k_{j+1} =  m_{j+1}\,, \ \widetilde k_i  = k_i  \ \text{otherwise}\,.
\end{aligned}
\]
\end{lem}

\begin{proof}
Each of these identities follows from its counterpart in \lemref{RSlem}. Let $c = \f m{m_2}$, $d= \f{k_2}{m_2}$,
and $a= \overline k_2$. By definition of $\overline k_2$, there exists $\,b\in \Z\,$ such that $ad = k_2
\,\overline k_2 / m_2 = 1 + b \, m/m_2 = 1 + bc\,$. Then $\,d\,m = c\,k_{2}\,$, so the first identity in lemma
\ref{RSlem} asserts
\begin{equation}
\label{firstrelation} \ell \left( \Phi_1 \big( \begin{smallmatrix} k_2/m_2 \, & \, -b \\  -m/m_2 \, & \, \overline k_2
\end{smallmatrix}\big) \right) R_{0,1,(k_1,m_2,k_3,\dots,k_{n-1})}\tau \ \ = \ \ S_{1,m,k}\tau \,.
\end{equation}
The first assertion of the current lemma follows, since
\begin{equation}
\label{firstrelation1} \big( \begin{smallmatrix} 1 \, & \, k_2/m \\  0 \, & \, 1 \end{smallmatrix} \big) \big(
\begin{smallmatrix} k_2/m_2 \, & \, -b \\  -m/m_2 \, & \, \overline k_2 \end{smallmatrix}\big)\ = \ \big(
\begin{smallmatrix} 0 \, & \, m_2/m \\  -m/m_2 \, & \, \overline k_2 \end{smallmatrix} \big)\ .
\end{equation}
For the verification of the second assertion we let $\,a = \overline k_{n-2}\,$, $\,c = -m/m_{n-2}\,$, $\,d =
k_{n-2}/m_{n-2}\,$, and choose $\,b\,$ so that $\,ad = k_{n-2}\overline k_{n-2}/m_{n-2} = 1 - b m /m_{n-2} = 1 +
bc\,$. Then $\,d\,m = -c\,k_{n-2}\,$, and the second identity in lemma \ref{RSlem} implies
\begin{equation}
\label{secondrelation} \ell \left(\Phi_{n-1}\big( \begin{smallmatrix} k_{n-2}/m_{n-2} \, & \, -b \\  m/m_{n-2} \, & \,
\overline k_{n-2}
\end{smallmatrix} \big) \right) S_{n-1,1,(k_1,\dots,k_{n-3},m_{n-2},k_{n-1})}\tau\, =\,  R_{n-2,m,k}\tau \,.
\end{equation}
We obtain the second assertion by applying $\,\ell(h_{n-1}(-\frac{k_{n-2}}{m}))\,$ to both sides.

For the verification of the third assertion we suppose that $\,2 \le j \le n-2\,$, set $\,a = k_{j+1}/m_{j+1}\,$,
$\,c= - m/m_{j+1}\,$, $\, d = \overline k_{j+1}\,$, and we choose $\,b\,$ so that $\,ad = k_{j+1} \overline
k_{j+1}/m_{j+1} = 1 -b m/m_{j+1} = 1 + b c\,$. Then $\,a\,m\,$ does equal $\,-c \, k_{j+1}\,$, so we can apply the
third identity in lemma \ref{RSlem}, with
\begin{equation}
\label{thirdrelation} \widetilde m \, = \, m \, k_{j-1} / m_{j+1}\,,\ \ \ \widetilde k_{j-1}\, = \, k_{j-1}\,\overline
k_{j+1}\,,\ \ \ \widetilde k_{j+1} \, =  \, m_{j+1}\,,
\end{equation}
which then reads as follows:
\begin{equation}
\label{thirdrelation1} \ell \left( \Phi_j\big( \begin{smallmatrix} \overline k_{j+1} \, & \, -b \\  m/m_{j+1} \, & \,
k_{j+1}/m_{j+1} \end{smallmatrix} \big) \right) S_{j,m,k}\ \ = \ \ R_{j-1,\widetilde m,\widetilde k}\,\tau \,.
\end{equation}
At this point, the matrix identity
\begin{equation}
\label{thirdrelation2} \big( \begin{smallmatrix} 1 \, & \, - \widetilde k_{j-1}/\widetilde m \\  0 \, & \, 1
\end{smallmatrix} \big) \big( \begin{smallmatrix} \overline k_{j+1} \, & \, -b \\  m/m_{j+1} \, & \, k_{j+1}/m_{j+1}
\end{smallmatrix}\big) \big( \begin{smallmatrix} 1 \, & \, -k_{j+1}/m \\  0 \, & \, 1 \end{smallmatrix} \big) \ = \
\big( \begin{smallmatrix} 0 \, & \, - m_{j+1}/m \\  m/m_{j+1} \, & \, 0 \end{smallmatrix} \big)
\end{equation}
completes the verification of the third assertion.
\end{proof}

\begin{lem}
\label{RSlem4} For $\,1 \le j \le n-2\,$, $\,m \in \Z_{\neq 0}\,$, $\,k \in \Z^{n-1}\,$,
\[
\begin{aligned}
&\ell\big(h_{j+1}({\textstyle{-\frac{k_{j}}{m}}})\big)R_{j,m,k}\tau \ =
\\
&\ \ \ = \, {\sum}_{r\in\Z/m\Z} \ e\left(\smallf{r\,k_j}{m}  \right)\,\int_\R \! \ell(h_{j}(t))\, \ell \big(h_{j}
(\smallf {r}{m})\big) S_{j,m,(k_1,\,\dots,\,k_j, r,k_{j+2},\,\dots,\,k_{n-1})}\tau\, dt\,.
\end{aligned}
\]
The integrals converge in the strong distribution topology and depend on the index $\,r\,$ only modulo $\,m\,$, as
indicated by the notation.
\end{lem}

\begin{proof}
In order to relate the $\,R_{j,m,k}\tau\,$ to the $\,S_{j,m,k}\tau\,$, we express them both in terms of a
projection $\,P_{j,k,m}\tau\,$,\,\ whose definition is similar to that of $\,R_{j,k,m}\tau\,$ and
$\,S_{j,k,m}\tau\,$\,:
\begin{equation}
\label{Pjkmtau}
P_{j,m,k}\tau \ = \ \int_{(\R/\Z)^{n-2}}\int_{\left\{\begin{smallmatrix}{x\in (\R/\Z)^{n-1}} \\
{x_j = x_{j+1} = 0\ }\end{smallmatrix}\right\}} e(k\cdot x\,+\, my_j)\,\, \ell(n_{x,y})\tau''\, dx\, dy\,.
\end{equation}
Evidently $P_{j,m,k}\tau$ does not depend on $k_j$ or $k_{j+1}$. Recall the definition of the elementary matrices
$\,E_{r,s}\,$ above (\ref{hj}), and note that $\,\tau''\,$ is invariant under translation by $\,N''(\R)\,$, which
is normal in $N(\R)$. Since $\exp(u E_{j,j+2})\,n_{x,y} \equiv n_{x, \widetilde y}\,$ modulo $N''(\R)$, with
$\widetilde y_i = y_i + \d_{i,j}\, u\,$,
\begin{equation}
\label{Pjkmtau1}
\begin{aligned}
&\ell\big(\exp(u\,E_{j,j+2})) P_{j,m,k}\tau \ =
\\
&\qquad =\ \int_{(\R/\Z)^{n-2}}\int_{\left\{\begin{smallmatrix}{x\in (\R/\Z)^{n-1}}\\
{x_j = x_{j+1} = 0\ }\end{smallmatrix}\right\}} e(k\cdot x\,+\, m(y_j-u))\,\, \ell(n_{x,y})\tau''\, dx\, dy
\\
&\qquad = \ e(-m\,u)\,P_{j,m,k}\tau\,.
\end{aligned}
\end{equation}
A simple computation, which can be reduced to the case of $GL(3,\R)$, shows that
\begin{equation}
\label{hjhj} h_j(u)\,h_{j+1}(v)\ = \ h_{j+1}(v)\,h_j(u)\,\exp(u\,v\,E_{j,j+2})\,.
\end{equation}
Comparing (\ref{Rjkmtau}) to (\ref{Pjkmtau}), we find
\begin{equation}
\label{Pjkmtau2} R_{j,m,k}\tau \ = \ \int_0^1 e(k_j\,t)\, \ell(h_j(t))\,P_{j,m,k}\tau\, dt\,,
\end{equation}
and similarly
\begin{equation}
\label{Pjkmtau3} S_{j,m,k}\tau \ = \ \int_0^1 e(k_{j+1}\,t)\, \ell(h_{j+1}(t))\,P_{j,m,k}\tau\, dt\,.
\end{equation}
This last identity exhibits $\,S_{j,m,k}\tau\,$ as one of the Fourier components of $\,P_{j,m,k}\tau\,$ with
respect to the action of a circle group -- recall that $\,P_{j,m,k}\tau\,$ does not depend on $\,k_{j+1}\,$ whereas
$\,S_{j,m,k}\tau\,$ does depend on $\,k_{j+1}$ -- and consequently
\begin{equation}
\label{Pjkmtau4} P_{j,m,k}\tau \ = \ {\sum}_{i\in\Z} \ S_{j,m,(k_1,\,\dots,\,k_j,\,i,k_{j+2},\,\dots,\,k_{n-1})}\tau
\end{equation}
is the sum of all the Fourier components. This sum converges in the strong distribution topology since the circle
action is continuous with respect to that topology.

Using, in order, (\ref{Pjkmtau2}), (\ref{hjhj}), (\ref{Pjkmtau1}), (\ref{Pjkmtau4}) and the transformation behavior
of the Fourier component $\,S_{j,m,k}\tau\,$ under the action of $\,h_{j+1}(t)\,$, we find
\begin{equation}
\label{Pjkmtau5}
\begin{aligned}
&\ell\big(h_{j+1}({\textstyle{-\frac{k_{j}}{m}}})\big)R_{j,m,k}\tau \ = \, \int_0^1 \! e(k_j\,t) \,
\ell\big(h_{j+1}({\textstyle{-\frac{k_{j}}{m}}})\,h_{j}(t)\big) P_{j,m,k}\tau\, dt \ =
\\
&\ \ \ \ = \, \int_0^1 \! e(k_j\,t)\,\ell\big(h_{j}(t)\,
h_{j+1}({\textstyle{-\frac{k_{j}}{m}}})\,\exp(\smallf{k_j\,t}{m}\,E_{j,j+2})\big)  P_{j,m,k}\tau\, dt
\\
&\ \ \ \ = \, \int_0^1 \ell\big(h_{j}(t)\,h_{j+1}({\textstyle{-\frac{k_{j}}{m}}})\big)  P_{j,m,k}\tau\, dt
\\
&\ \ \ \ = \, {\sum}_{i\in\Z} \, \int_0^1 \ell \big(h_{j}(t)\, h_{j+1}({\textstyle{-\frac{k_{j}}{m}}})\big)  \,
S_{j,m,(k_1,\,\dots,\,k_j,\,i,k_{j+2},\,\dots,\,k_{n-1})}\tau\, dt
\\
&\ \ \ \ = \, {\sum}_{i\in\Z} \, \int_0^1 e\left(\smallf{i\,k_j}{m}  \right)\,\ell(h_{j}(t))\,
S_{j,m,(k_1,\,\dots,\,k_j,\,i,k_{j+2},\,\dots,\,k_{n-1})}\tau\, dt
\\
&\ \ \ \ = \, \sum_{r=0}^{m-1}\, {\sum}_{i\in\Z} \int_0^1 \! e\left(\smallf{r\,k_j}{m}  \right)\,\ell(h_{j}(t))\,
S_{j,m,(k_1,\,\dots,\,k_j, r+im,k_{j+2},\,\dots,\,k_{n-1})}\tau\, dt
\\
&\ \ \ \ = \, {\sum}_{r\in\Z/m\Z} \ e\left(\smallf{r\,k_j}{m}  \right)\,\int_\R \! \ell(h_{j}(t))\,
S_{j,m,(k_1,\,\dots,\,k_j, r,k_{j+2},\,\dots,\,k_{n-1})}\tau\, dt\,.
\end{aligned}
\end{equation}
At the final step we have used the identity
\begin{equation}
\label{Pjkmtau6} \ell(h_j(1))S_{j,m,k}\tau\ = \ S_{j,m,(k_1,\,\dots,\,k_j, \,k_{j+1} -
m,\,k_{j+2},\,\dots,\,k_{n-1})}\tau\,,
\end{equation}
which amounts to a restatement of the final assertion of lemma \ref{RSlem}. In particular the integral in the last
line of (\ref{Pjkmtau5}) depends only on the class of $\,r\,$ modulo $\,m\,$, as claimed. The integral converges in
the strong distribution topology because the sum (\ref{Pjkmtau4}) does. Finally, shifting the variable of
integration by $\,\frac{k_{j+1}}{m}\,$, allows us to replace $\,S_{j,m,(\dots,r,\dots)}\tau\,$ by its
$\,\ell(h_j(\frac{r}{m}))$-translate.
\end{proof}

By construction $\,R_{j-1,m,k}\tau\,$ and $\,S_{j,m,k}\tau\,$ transform under the left action of $N_j(\R)$
according to characters of that group -- recall (\ref{Rjkmtau2}). Like any vector in $V^{-\infty}_{\l,\d}$ they
also transform under the right action of $\,B_-(\R)\,$ according to the inducing character; cf. (\ref{autod}).
Since $\,N_j(\R)\cdot\operatorname{Im}\Phi_j\cdot B_-(\R)\,$ is open in $G(\R)$, we can legitimately restrict
$\,R_{j-1,m,k}\tau\,$ and $\,S_{j,m,k}\tau\,$ to the image of $\Phi_j$\,. In analogy to (\ref{Vld}), pairs
$(\nu,\eta)\in \C \times \Z/2\Z\,$ parameterize principal series representations of $SL(2,\R)$\,:
\begin{equation}
\label{Wmueta}
\begin{aligned}
W_{\nu,\eta}\ = \ \{\,f : SL(2,\R)\to \C \, \big| \,f\big( g \left(
\begin{smallmatrix} a & 0 \\ c & a^{-1}
\end{smallmatrix}\right)\big) = |a|^{1-\nu} (\sgn a)^\eta
f(g)\,\}\,;
\end{aligned}
\end{equation}
$W^{-\infty}_{\nu,\eta}$ shall denote the space of distribution vectors. Under right translation by elements of
$\,\operatorname{Im}\Phi_j\cap B_-(\R)\,$, $R_{j-1,m,k}\tau$ and $S_{j,m,k}\tau$ transform according to the
(restriction of) the inducing character. Thus
\begin{equation}
\label{RSrestr} (R_{j-1,m,k}\tau )\circ \Phi_j \,,\ (S_{j,m,k}\tau ) \circ \Phi_j \ \ \in \ \
W^{-\infty}_{\l_{j}-\l_{j+1}\,,\,\d_{j}-\d_{j+1}} \,.
\end{equation}
One can restrict $\,R_{j-1,m,k}\tau\circ \Phi_{j}\,$ and $\,S_{j,m,k}\tau\circ \Phi_j\,$ to the upper triangular
unipotent subgroup of $\,SL(2,\R)\,$, just as it is legitimate to restrict $\,\tau\,$ to $\,N(\R)\,$. Thus we can
define distributions of one variable $\,\rho_{j,m,k}\,$, $\,\sigma_{j,m,k}\,$ by the equations
\begin{equation}
\label{rhosigma}
\begin{aligned}
&\rho_{j,m,k}(x) \ = \ \left( \ell\big(h_{j+1} (- \smallf{k_{j}}{m} ) \big)R_{j,m,k}\tau \right) \circ \Phi_{j+1}\big(\begin{smallmatrix} \, 1 \, & \, -x \, \\
\, 0 \, & \, 1 \,
\end{smallmatrix}\big)\,,
\\
&\qquad \sigma_{j,m,k}(x) \ = \ \left( \ell\big(h_{j} ( \smallf{k_{j+1}}{m} ) \big)S_{j,m,k}\tau \right) \circ \Phi_j
\big(\begin{smallmatrix} 1 \, &  x \\ 0 & \, 1
\end{smallmatrix}\big)\,;
\end{aligned}
\end{equation}
in the exceptional cases (\ref{Rjk1tau}) the integers $k_0$, $k_n$ should be interpreted as zero.

At the extremes, $\rho_{0,1,k}$ and $\sigma_{n-1,1,k}$ are expressible in terms of the Fourier coefficients
$\,c_k\,$ of $\,\tau_{\text{abelian}}\,$,\,\ and hence, via (\ref{ctoa}), also in terms of the $\,a_k\,$:

\begin{lem}\label{Rjk1taulem}
$\ \rho_{0,1,k}(x)\ = \ \sum_{r \neq 0}\, c_{r, k_2, k_3,\dots , k_{n-1}}\, e(- r x)\ $ and $\ \sigma_{n-1,1,k}(x)\
= \ \sum_{r \neq 0}\, c_{k_1,\dots , k_{n-2}, r}\, e(r x)\,$.
\end{lem}

\begin{proof}
The integrations with respect to the $y$ variables in (\ref{Rjk1tau}) convert $\tau''$ into $\tau'$, and the
integrations with respect to the remaining variables turn the series (\ref{tauabelexp}) into a Fourier series in
the first, respectively last, variable by fixing the remaining indices.
\end{proof}

The Voronoi formula amounts to a connection between the Fourier coefficients of the $\rho_{0,1,k}$ to those of the
$\sigma_{n-1,1,k}$. Our proof establishes that connection by relating the $\rho_{j-1,m,k}$ to the $\sigma_{j,m,k}$
and the $\sigma_{j,m,k}$ to the $\rho_{j,m,k}$. Lemma \ref{RSlem3} embodies a weak form of the former; we
strengthen it to a useful version in section \ref{sec:inforder}. Implicitly lemma \ref{RSlem4} relates the
$\sigma_{j,m,k}$ to the $\rho_{j,m,k}$. Our next proposition~-- which also collects some additional information~--
makes that quite explicit.

We normalize the Fourier transform $\,\mathcal Ff = \widehat f\,$ of a Schwartz function $f \in \Sch(\R)$ by the
formula
\begin{equation}
\label{ftschwartz} \widehat f(x)\ \ = \ \ \int_\R f(y)\,e(-xy)\,dy\,.
\end{equation}
The same formula expresses the Fourier transform of any $f \in \Sch'(\R)$, the space of tempered distributions, to
which the Fourier transform extends via the duality between functions and distributions. The proposition also
involves the finite Fourier transform
\begin{equation}
\label{ftfinite} \widehat a_k\ \ = \ \ {\sum}_{\ell \in \Z/m\Z}\ e\big( \textstyle\f {k \ell}m \big)\,a_\ell \qquad (\,
a = (a_k)_{k\in \Z/m\Z}\, )\,,
\end{equation}
of functions on $\,\Z/m\Z\,$, normalized following a common convention.

\begin{prop}\label{sigmarhoftprop} The $\rho_{j,m,k}$ and $\sigma_{j,m,k}$ are tempered distributions. The $\rho_{j,m,k}$ do not depend on
$k_{j+1}$ and the $\sigma_{j,m,k}$ do not depend on $k_j$. For $1 \le j \le n-2$ $\rho_{j,m,k}$ depends on $k_j$
only modulo $m$, and $\sigma_{j,m,k}$ depends on $k_{j+1}$ only modulo $m$. Still for $\,1 \le j \le n-2\,$, and
for any $\, a = (a_k)_{k\in \Z/m\Z}\,$,
\[
\begin{aligned}
&{\sum}_{\ell\in \Z/m\Z}\ a_\ell\,\rho_{j,m,(k_1,\ldots,\,k_{j-1},\,\ell,\,k_{j+1},\ldots,\,k_{n-1})}(x)\ =
\\
&\qquad\qquad =\ {\sum}_{\ell\in\Z/m\Z}\ \widehat{a}_\ell\,
    \widehat{\sigma}_{j,m,(k_1,\ldots,\,k_{j},\,\ell,\,k_{j+2},\ldots,\,k_{n-1})}(mx)\,,
\end{aligned}
\]
or equivalently,
\[
\rho_{j,m,k}(x) \ = \ {\sum}_{\ell\in\Z/m\Z}\ e\left(\smallf{k_j\ell}{m}\right) \,
\widehat{\sigma}_{j,m,(k_1,\ldots,k_j,\ell,k_{j+2},\ldots,k_{n-1})}(mx)\,.
\]
\end{prop}

\begin{proof}
By construction the $\rho_{j,m,k}$ and $\sigma_{j,m,k}$ are restrictions, to the upper triangular unipotent group
in $SL(2,\R)$, of vectors in various $W^{-\infty}_{\nu,\eta}$, and such distribution vectors can be paired
continuously against vectors in the dual representations $W^{\infty}_{-\nu,-\eta}$. On the other hand, any $f \in
\mathcal S(\R)$ naturally extends to a vector in $W^{\infty}_{-\nu,-\eta}$, and thus pairs naturally and
continuously against restrictions to $\R$ of vectors in $W^{-\infty}_{\nu,\eta}$. This proves the temperedness of
the $\rho_{j,m,k}$ and $\sigma_{j,m,k}$. The assertions about the dependence on the $\,k_i\,$ follow directly from
the corresponding statements in \lemref{RSlem2}.

In order to relate the $\rho_{j,m,k}$ to the $\sigma_{j,m,k}$, we evaluate both sides of the identity in
\lemref{RSlem4} on  $\,h_{j+1}(-x)\,$. In view of (\ref{hjhj}), with $\,u = -t-r/m\,$, $\,v= -x\,$, and
(\ref{Pjkmtau1}),
\begin{equation}
\label{rhosigma2}
\begin{aligned}
& \big(\ell(h_{j}(t))\, \ell \big(h_{j} (\smallf {r}{m})\big) S_{j,m,(\dots, r,\dots )}\tau \big)\big( h_{j+1}(-x)
\big)\ =
\\
&\ \ \ \ =\ e\big(x(r+mt)\big)\,\big( \ell \big(h_{j} (\smallf {r}{m})\big) \ell(h_{j+1}(x))S_{j,m,(\dots, r,\dots
)}\tau \big) \big( h_{j}(-t) \big)
\\
&\ \ \ \ =\ e(m\,t)\,\big( \ell \big(h_{j} (\smallf {r}{m})\big) S_{j,m,(\dots, r,\dots )}\tau \big) \big( h_{j}(-t)
\big)\,.
\end{aligned}
\end{equation}
The second step uses the identity $\,\ell(h_{j+1}(x))S_{j,m,(\dots, r,\dots )}\tau\, = e(-r\,x)S_{j,m,(\dots,
r,\dots )}\tau\,$, which follows from (\ref{Pjkmtau3}). Coupled with \lemref{RSlem4}, (\ref{rhosigma2}) allow us to
conclude
\begin{equation}
\label{rhosigma3}
\begin{aligned}
&\rho_{j,m,k}(x) \ = \ \ell\big(h_{j+1}({\textstyle{-\frac{k_{j}}{m}}})\big)\, R_{j,m,k}\tau\big(h_{j+1}( - x)\big) \ =
\\
&\ \ \ = \ {\sum}_{r\in\Z/m\Z} \ e\left(\smallf{r\,k_j}{m}  \right) \int_\R \!e(m\,x\,t)\,
\sigma_{j,m,(k_1,\,\dots,\,k_j, r,k_{j+2},\,\dots,\,k_{n-1})}(-t)\, dt
\\
&\ \ \ = \ {\sum}_{r\in\Z/m\Z} \ e\left(\smallf{r\,k_j}{m}  \right) \,\widehat\sigma_{j,m,(k_1,\,\dots,\,k_j,
r,k_{j+2},\,\dots,\,k_{n-1})}\,(m\,x)\,.
\end{aligned}
\end{equation}
That is the second, equivalent statement about the connection between the $\rho_{j,m,k}$ and the Fourier transforms
of the $\,\sigma_{j,m,k}$.
\end{proof}

\section{Vanishing to infinite order}\label{sec:inforder}

Let $\,I \subset \R\,$ be an open interval, and $\,x_0\,$ a point in $\,I\,$. In \cite{inforder} we introduced the
notion of a distribution $\,\sigma \in C^{-\infty}(I)\,$ vanishing at $\,x_0\,$ to infinite order. When
$\,\sigma\,$ happens to be a $\,C^\infty\,$ function, this coincides with the usual notion of vanishing to infinite
order. We shall not repeat the details of the definition here; instead we summarize the features that are relevant
for this paper. First of all,
\begin{equation}
\label{inf}
\begin{aligned}
&\text{if $\,\sigma_1\,,\, \sigma_2 \in C^{-\infty}(I)\,$ both vanish to infinite order at $\,x_0\,$, and}
\\
&\qquad \text{if $\,\sigma_1\,$ and $\,\sigma_2\,$ agree on $\,I-\{x_0\}\,$, then $\,\sigma_1 = \sigma_2\,$ on all of
$\,I\,$.}
\end{aligned}
\end{equation}
Thus, if $\,\sigma_0 \in C^{-\infty}(I-\{x_0\})\,$ can be extended to a distribution $\,\sigma \in
C^{-\infty}(I)\,$ which vanishes to infinite order at $\,x_0\,$, that extension is uniquely determined; we then
call $\,\sigma\,$ the canonical extension of $\,\sigma_0\,$ across $\,x_0\,$. The terminology ``canonical
extension" can be justified:
\begin{equation}
\label{inf1}
\begin{aligned}
&\text{the property of vanishing to infinite order at $\,x_0\,$ is preserved by $\,C^\infty$}
\\
&\text{coordinate changes, by differentiation, and by multiplication with $C^\infty$}
\\
&\text{functions or with $|x-x_0|^\nu(\sgn(x-x_0))^\eta$,\,\ for any $(\nu,\eta)\in \C\times \Z/2\Z\,$};
\end{aligned}
\end{equation}
consequently these operations commute with the process of canonical extension. Everything that has been said also
applies to distributions defined on an open neighborhood $\,I\,$ of $\,\infty\,$ in $\,\R\mathbb P^1 \cong \R \cup
\{\infty\}$, via the coordinate change $\,x \rightsquigarrow 1/x\,$. The Fourier transform provides a connection
between vanishing to infinite order at the origin and canonical extension across infinity:
\begin{equation}
\label{inf2}
\begin{aligned}
&\text{$\,\sigma\in \mathcal S'(\R)$ vanishes to infinite order at the origin}
\\
&\qquad \Longrightarrow \qquad \text{$\,\widehat \sigma\,$ has a canonical extension across $\,\infty\,$}.
\end{aligned}
\end{equation}
In particular,
\begin{equation}
\label{inf3}
\begin{aligned}
&\text{any periodic distribution $\,\sigma = \textstyle\sum_{r\neq 0}\, a_r \, e(rx)\,$ with zero}
\\
&\qquad\text{constant term has a canonical extension across $\,\infty\,$},
\end{aligned}
\end{equation}
as follows from (\ref{inf2}) since $\,\widehat \sigma = {\sum}_{r \neq 0}\, a_{r}\, \d_r(x)\,$, with $\,\d_r(x)$
denoting the delta function at $\,r$, vanishes identically near the origin, and that is a much stronger condition
than vanishing to infinite order.

For the statements of the next proposition, we fix $k\in \Z^{n-1}$, $m\in \Z_{\neq 0}$, and $j$, $1\le j \le n-2$,
as before. We define
\begin{equation}
\label{mudef} \mu_j\ :\ \R^* \ \longrightarrow\ \C^*\,,\ \ \ \ \ \ \mu_j(x)\ = \
|x|^{\l_j-\l_{j+1}-1}\sgn(x)^{\d_j+\d_{j+1}}\,,
\end{equation}
and we again use the notational conventions of \lemref{RSlem3}\,: $\,m_j = \gcd(m,k_j)\,$ is the greatest common
divisor of $m$ and $k_j$, and $\,\overline {k_j}\in \Z/m\Z\,$ an integer such that $\,\overline{k_j} \,k_j \equiv
m_j \imod m\,$.

\begin{prop}
\label{fliprelationsprop} All the $\rho_{j,m,k}$ and $\sigma_{j,m,k}$ extend canonically to distributions on the
compactified real line $\,\R \cup \{\infty\}\,$,\,\ and vanish to infinite order at the origin. Both
$\,\sigma_{1,m,k}\,$ and $\,\rho_{n-2,m,k}\,$ even vanish to infinite order at every rational point. In terms of
the notational conventions just introduced,
\[
\begin{aligned}
&\sigma_{1,m,k}(x) \ = \ \mu_1(\smallf{m\,x}{m_2})\,{\sum}_{\ell\neq 0}\,\, c_{\ell,m_2,k_3,\ldots,k_{n-1}}\,\,
e\!\(\smallf{\ell m_2}{m}( \overline{k_2} -\smallf{m_2}{m\, x})\) ,
\\
&\rho_{n-2,m,k}(x) \ = \ \mu_{n-1}(\smallf{m\,x}{m_{n-2}})\,{\sum}_{\ell\neq 0}\,\,
c_{k_1,\ldots,k_{n-3},m_{n-2},\ell}\,\, e\!\(\!\smallf{\ell m_{n-2}}{m} (\smallf{m_{n-2}}{m\, x} -
\overline{k_{n-2}})\) ,
\end{aligned}
\]
and, for $2\le j \le n-2$\,,
\[
\begin{aligned}
&\sigma_{j,m,k}(x) \ = \ \mu_{j}(\smallf{m\,x}{m_{j+1}})\, \rho_{j-1, \widetilde m,\,
(k_1,\ldots,\,k_{j-2},\,\widetilde k_{j-1}\,,\,k_{j},\, m_{j+1},\,k_{j+2},
    \ldots,k_{n-1})}\!\(\smallf{m_{j+1}^2}{m^2\,x}\)
\\
&\qquad\qquad\qquad\qquad\qquad\text{with} \ \ \widetilde m = \smallf{m\,k_{j-1}}{m_{j+1}}\ \ \text{and}\ \ \widetilde
k_{j-1} = k_{j-1}\,\overline{k_{j+1}}\,.
\end{aligned}
\]
In all three cases these are identities of distributions on $\,\R \cup \{\infty\}\,$.
\end{prop}

The proof will occupy the remainder of this section. We begin with a remark on the action of $SL(2,\R)$ on
$W^{-\infty}_{\nu,\eta}$. If $a,b,c,d \in \R$ and $ad - bc =1$,
\begin{equation}
\label{sl2identity}
\begin{pmatrix} a \, & \, b \\ c \, & \, d \end{pmatrix} \begin{pmatrix} 1 \, & \, x \\ 0 \, & \, 1 \end{pmatrix} \
= \ \begin{pmatrix} 1 \, & \, \f{ax+b}{cx+d} \\ \, 0 & \, 1 \end{pmatrix} \begin{pmatrix} \f 1{cx+d} \, & \, 0 \\ c \,
& \, cx + d & \,\end{pmatrix}\,.
\end{equation}
Hence, for $\sigma \in W^{-\infty}_{\nu,\eta}$,
\begin{equation}
\label{sl2action} \left( \ell \begin{pmatrix} d \, & \, -b \\ -c \, & \, a \end{pmatrix}\sigma \right) \begin{pmatrix} 1 \, & \, x \\
0 \, & \, 1
\end{pmatrix}\ \ = \ \ |cx+d|^{\nu -1}\, (\sgn (cx+d))^\eta\ \sigma \! \begin{pmatrix} 1 \, & \, \f{ax+b}{cx+d} \\ \, 0 & \, 1 \end{pmatrix}\,.
\end{equation}
Although this appears to be an equality of distributions on $\R - \{-\f dc\}$, it can be given meaning on $\R \cup
\{\infty\}$\,: if $\sigma_0 \in C^{-\infty}(\R)$ is the restriction of some $\sigma \in W^{-\infty}_{\nu,\eta}$ to
the upper triangular unipotent subgroup of $SL(2,\R)$ -- identified with $\,\R\,$ in the usual manner -- then
$\,(\sgn x)^\eta |x|^{\nu- 1} \sigma(1/x)\,$ extends across the origin, and $\sigma_0$ along with this extension
completely determines $\sigma$. The extension is not unique since one can add any finite linear combination of
derivatives of the delta function; in other words, $\,\sigma_0\,$ does not determine $\,\sigma\,$ completely.
However, according to (\ref{inf1}), if $\,\sigma_0\,$ has a canonical extension across infinity when viewed as a
distribution in the usual sense, then $\,\sigma_0\,$ can also be canonically extended to a vector in $\,
W^{-\infty}_{\nu,\eta}\,$.

When the three identities in \lemref{RSlem3} are restricted to the images of, respectively, $\,\Phi_1\,$,
$\,\Phi_{n-1}\,$, $\,\Phi_j\,$, the equalities asserted by proposition \ref{fliprelationsprop} follow from
(\ref{sl2action}), but initially only as equalities of distributions on $\,\R-\{0\}\,$. To complete the proof of
the proposition, we shall show that the $\,R_{j,m,k}\tau\circ \Phi_{j+1}\,$ and {$\,{S_{j,m,k}\tau \circ
\Phi_j}\,$} vanish to infinite order at infinity: in that case, the renormalized quantities $\,\ell
(h_{j}(-\frac{k_{j-1}}{m}) )R_{j-1,m,k}\tau\circ \Phi_j\,$ and $\,\ell (h_{j} (\frac{k_{j+1}}{m})
)S_{j,m,k}\tau\circ \Phi_j\,$ also vanish to infinite order at $\infty$. Since \lemref{RSlem3} relates the behavior
of one near $\infty$ to the behavior of the other at the origin, both vanish to infinite order also at the origin.
We can then conclude that the $\,\rho_{j,k,m}\,$ and $\,\sigma_{j,m,k}\,$ have canonical extensions across
infinity, that they vanish to infinite order at the origin if $\,1 \le j \le n-2\,$,\,\ and that the equalities
asserted by the proposition are valid even at $\,x=0\,$ and $\,x=\,\infty\,$,\,\ as asserted. The second of the
three identities relates the behavior of $\,\rho_{n-2,m,k}\,$ near $\,\infty\,$ to that of the periodic
distribution
\begin{equation}
\label{Rjk1.1tau} \sigma_{n-1,1,(k_1,\ldots,k_{n-3},m_{n-2},k_{n-1})}(x)\ = \ {\sum}_{\ell\neq 0}\,
c_{k_1,\ldots,k_{n-3},m_{n-2},\ell}\,\, e(\ell x)
\end{equation}
near $\,x=-\f{\overline{k_{n-2}}}{m/m_{n-2}}\,$.\,\ Since $m$ and $k_{n-2}$ are multiples of $m_{n-2}$, we may
write them as $m=am_{n-2}$ and $k=bm_{n-2}$.  In this parametrization, $\,\overline{k_{n-2}}\,$ represents the
inverse $\,\bar{b}\,$ of $\,b\,$ modulo $\,a\,$, and the previous fraction is equal to $-{\bar{b}}/{a}$.  Thus,
when one considers all pairs of integers $\,m\neq 0\,$ and $\,k_{n-2}\,$ having greatest common divisor
$\,m_{n-2}\,$, one obtains every rational number as such a fraction. It follows that both the Fourier series
(\ref{Rjk1.1tau}) and the $\,\rho_{n-2,m,k}\,$ vanish to infinite order also at every rational point. The same
assertion about the $\,\sigma_{1,m,k}\,$ can be proved the same way, of course.

In short, to prove proposition \ref{fliprelationsprop} it suffices to show that all the $\,S_{j,m,k}\tau \circ
\Phi_j\,$ and $\,R_{j,m,k}\tau\circ \Phi_{j+1}\,$ vanish to infinite order at $\,\infty\,$. The outer automorphism
\begin{equation}
\label{Rjk1.2tau} g\ \ \mapsto \ \ w_{\text{long}}\,(g^{-1})^t\,w_{\text{long}}^{-1}\ ,\ \ \ \text{with}\ \ \
w_{\text{long}}\ = \ \left(\begin{smallmatrix} & & & & 1 \\ & & & \cdot & \\ & & \cdot & & \\ & \cdot &  & & \\ 1 & & &
&\end{smallmatrix}\right) \ ,
\end{equation}
of $\,G(\R)\,$ relates the automorphic distribution $\,\tau\,$ to its own contragredient $\,\widetilde \tau\,$, and
further relates the quantities $\,S_{j,m,k}\tau\,$ for $\,\tau\,$ to the $\,R_{j,m,k}\tau\,$ corresponding to
$\,\widetilde\tau\,$. We therefore only need to treat the case of the $\,R_{j,m,k}\tau\,$.

The verification of the vanishing to infinite order requires global arguments~-- we need to regard $\,\tau\,$ and
the various quantities related to it as sections of a line bundle on the (real) flag variety
\begin{equation}
\label{flag} X \ \ = \ \ G(\R)/B_-(\R)\,;
\end{equation}
here $\,B_-(\R)\,$ refers to the group of lower triangular matrices, as before. We let $\,\mathcal L_{\l,\d}
\rightarrow X\,$ denote the equivariant line bundle~-- i.e., line bundle on which $G(\R)$ acts, compatibly with its
action on $\,X\,$~-- on whose fiber at the identity coset $\,B_-(\R)\,$ operates via the character
\begin{equation}
\label{lb1} \(\begin{smallmatrix} a_1 & 0 & 0 \\ \star & \ddots & 0 \\ \star & \star & a_n \\ \end{smallmatrix}\)\ \
\mapsto \ \ {\prod}_{1\le j\le n}\(|a_j|^{\l_j + j - \f{n+1}{2}}\sgn(a_j)^{\d_j} \) \,.
\end{equation}
The representation space in which $\,\tau\,$ lies coincides with the space of distribution sections of $\,\mathcal
L_{\l,\d}\,$,
\begin{equation}
\label{lb2} V^{-\infty}_{\l,\d}\ \ = \ \ C^{-\infty}(X,\mathcal L_{\l,\d})\,,
\end{equation}
on which $\,G(\R)\,$ acts via left translation, as it does in the case of $\,V^{\infty}_{\l,\d}$; cf.
\cite{voronoi}. The analogous description applies to the representation space $\,W^{-\infty}_{\nu,\eta}\,$ of the
group $\,SL(2,\R)$, whose flag variety is equivariantly embedded in $\,X\,$ via each of the $\,\Phi_j\,$.

Let $\,\it o \in X\,$ denote the identity coset $\,e\,B_-(\R)$. The upper unipotent group $\,N(\R)\,$ acts freely
on the $\,N(\R)$-orbit through $\,\it o\,$, and
\begin{equation}
\label{flag2} N(\R) \ \ \cong \ \ X_0 \ \ =_{\text{def}}\ \ N(\R)\cdot \it o\ \ \subset \ \ X
\end{equation}
is the open Schubert cell. The isotropy subgroup $\,B_-(\R)\,$ at $\,\it o\,$ intersects $\,N(\R)\,$ only in the
identity, and that makes the line bundle $\,\mathcal L_{\l,\d}\,$ canonically trivial over the open Schubert
cell~-- another, equivalent way of identifying the restriction $\,\sigma|_{X_0}\,$ of any $\,\sigma \in
V^{-\infty}_{\l,\d}\,$ to the open Schubert cell with a scalar distribution. A simple computation in $\,SL(2,\R)\,$
shows that, as $\,t\to\infty\,$,\,\ the curve $\,h_j(t)\,\it o\,$ converges to
\begin{equation}
\label{flag3} {\lim}_{t\to\infty} \, h_j(t)\, \it o \ = \ s_j\,\it o\,,
\end{equation}
the translate of the base point $\,\it o\,$ by
\begin{equation}
\label{flag3.5} s_j \ \ = \ \ \Phi_j \( \begin{smallmatrix} \, 0 \, & \, 1 \, \\ \, -1 \, & \, 0 \, \end{smallmatrix}
\) \qquad (\, 1 \le j \le n-1\,)\,,
\end{equation}
which normalizes the diagonal subgroup of $G(\R)$ and represents the $j$-th simple Weyl reflection. The
$\,N(\R)$-orbit through $\,s_j\,\it o\,$,
\begin{equation}
\label{flag4} C_j\ \ =_{\text{def}} \ \ N(\R)\cdot s_j\, \it o \ \subset \ X
\end{equation}
has codimension one. In fact, the $\,C_j\,$, $\,1 \leq j \leq n-1\,$, are exactly the codimension one Schubert
cells.

The codimension one subgroup $\,N_j(\R) \subset N(\R)$, defined in (\ref{Nj}), acts freely at $\,s_j\,\it o\,$,\,\
hence
\begin{equation}
\label{flag5} N_j(\R) \ \ \cong \ \ C_j \ \ = \ \ N_j(\R)\cdot s_j\,\it o\,.
\end{equation}
As a subgroup of $\,N(\R)\,$,\,\ $\,N_j(\R)\,$ acts freely also at the base point $\,\it o\,$,\,\ and the resulting
orbit
\begin{equation}
\label{flag6} N_j(\R) \ \ \cong \ \ {\{x_j=0\}} \ \ =_{\text{def}}\ \ N_j(\R)\cdot \it o
\end{equation}
lies in the open Schubert cell $\,X_0\,$\,\ as a closed, codimension one submanifold. Since $\,s_j \in \Phi_j\big(
SL(2,\R)\big)$ normalizes $\,N_j(\R)$,\,\ left translation by $\,s_j\,$ relates the two orbits,
\begin{equation}
\label{flag5.5}
\begin{CD}
\ell(s_j) \ : \ \{x_j=0\} \ @>{\thicksim}>> \ C_j\,.
\end{CD}
\end{equation}
Note that $\,s_j^2\,$ lies in the diagonal subgroup of $G(\R)$, and thus fixes the point $\,\it o\,$.

The group $\,\Phi_j\big( SL(2,\R)\big)$ normalizes $\,N_j(\R)$,\,\ and $\,N_j(\R) \times \R \cong N(\R)\,$ via
$\,(n,t)\mapsto n\,h_j(t)\,$. Consequently the $\,N_j(\R)$-translates of
\begin{equation}
\label{flag7}
\begin{CD}
\Phi_j\big( SL(2,\R)\big)\cdot {\it} o \ \ \cong \ \ SL(2,\R)/\,\Phi_j^{-1}\!\big(B_-(\R)\big) \ \ \cong \ \ \R\mathbb
P^1
\end{CD}
\end{equation}
sweep out an $\,N_j(\R)$-equivariant fibration
\begin{equation}
\label{flag8}
\begin{CD}
X_0  @>{\subset}>> X_0 \cup C_j @<{\supset}<< \ \Phi_j\big( SL(2,\R)\big)\cdot {\it} o   \\
@VV{\R}V         @VV{\R\mathbb P^1}V   @VV{\R\mathbb P^1}V  \\
\{x_j=0\}\ @= \ \{x_j=0\} \ @<{\supset}<< \{{\it o}\} \ \ .
\end{CD}
\end{equation}
Note that $\,X_0 \cup C_j\,$ is open in $\,X\,$ since its complement, i.e., the union of all non-open Schubert
cells other than $\,C_j\,$, is closed.

We shall need to consider $\,P_{j,n-j}(\R)\,$,\,\ the standard upper parabolic subgroup of type $(j,n-j)$ -- in
other words, the parabolic subgroup of $\,GL(n,\R)\,$ generated by $\,N(\R)\,$ and the diagonally embedded copies
of $\,GL(j,\R)\,$, $\,GL(n-j,\R)\,$ placed, respectively, into the top left $\,j\times j\,$ and bottom right
$\,(n-j) \times (n-j)\,$ corners. Further notation: for any element $\,w\,$ of the normalizer of the diagonal
subgroup,
\begin{equation}
\label{schw}
C_w \ \ = \ \ N(\R)\cdot w\, {\it o}
\end{equation}
is the Schubert cell containing the point $\,w\,{\it o}\,$; it depends only on the coset of $\,w\,$ modulo the
diagonal subgroup. In particular, $\,C_{s_j} = C_j\,$ as previously defined, and $\,C_e = X_0\,$.

\begin{lem}
\label{inforderlem} The orbit $\,P_{j,n-j}(\R)\cdot s_j\,{\it o}\,$ is a locally closed algebraic subma\-nifold of
$\,X\,$, of codimension one. It contains $\,C_j\,$ as a Zariski open subset. Define, by downward induction,
\[
C_j^j \, = \, C_j\,,\ \ \ C_j^{i} \, = \, s_{i}\,C_j^{i+1} \cup C_j^{i+1}\ \ \ \ \text{for} \ \ 1 \leq i \leq j-1 \,.
\]
Then $\,C_j^i\,$ is a Zariski open, $\,N(\R)$-invariant subset of $\,P_{j,n-j}(\R)\cdot s_j\,{\it o}\subset X\,$,
which coincides with the union of the Schubert cells $\,C_{s_{j_1}s_{j_2}\dots s_{j_r}s_j}\,$, $\,i \leq j_1 <
\dots < j_r < j\,$, $\,0 \leq r \leq j - i +1\,$. It is also the smallest $\,N(\R)$-invariant subset of
$\,P_{j,n-j}(\R)\cdot s_j\,{\it o}\,$ containing $\,s_i\,C^{i+1}_j\,$.
\end{lem}

\begin{proof}
We let $\,\{e_1,e_2,\dots,e_n\}\,$ denote the standard basis of $\,\R^n\,$ and $\,\{f_1,f_2,\dots,f_n\}\,$ the dual
basis. The function
\begin{equation}
\label{cicj} G(\R) \ \ni \ g \ \mapsto \ \langle \, f_1 \wedge f_2 \wedge \dots \wedge f_j \, , \, g^{-1}( e_1 \wedge
e_2 \wedge \dots \wedge e_j )\, \rangle
\end{equation}
transforms according to a character under right translation by $\,B_-(\R)\,$, is left invariant under the action of
$\,N(\R)\,$, and does not vanish at the identity. Its vanishing locus, which we can and shall regard as a
subvariety of $\,X = G(\R)/B_-(\R)$, therefore consists of a union of lower dimensional $\,N(\R)$-orbits. The
action of $\,s_j\,$ on the standard basis interchanges $\,e_j\,$ and $\,e_{j+1}\,$,\,\ up to sign, but leaves the
other basis ele\-ments alone. The function (\ref{cicj}) therefore vanishes at $\,s_j\,$,\,\ which implies that the
vanishing locus contains the $\,N(\R)$-orbit of $\,s_j\,\it o\,$,\,\ i.e., the codimension one Schubert cell
$\,C_{s_j} = C_j\,$. Arguing similarly one sees that the vanishing locus does not contain any of the other
codimension one Schubert cells $\,C_i\,$, $\,i \neq j\,$, and must therefore coincide with the closure of
$\,C_j\,$.\,\  Since $\,P_{j,n-j}(\R)\,$ preserves the line spanned by $\,e_1 \wedge e_2 \wedge \dots \wedge
e_j\,$, the vanishing locus contains $\,P_{j,n-j}(\R)\cdot s_j\,{\it o}\,$. But $\,N(\R) \subset P_{j,n-j}(\R)\,$,
so $\,C_j = N(\R)\cdot s_j\,{\it o}\,$ is contained, and necessarily Zariski open, in $\,P_{j,n-j}(\R)\cdot
s_j\,{\it o}\,$. Since $\,P_{j,n-j}(\R)\,$ acts transitively on $\,P_{j,n-j}(\R)\cdot s_j\,{\it o}\,$, which
contains the locally Zariski closed submanifold $\,C_j\,$ as a Zariski open subset, the orbit must also be a
locally closed submanifold.

By construction, the $\,C^i_j\,$ are finite unions of translates of $\,C_j\,$ by elements of $\,P_{j,n-j}(\R)\,$,
and are therefore Zariski open in $\,P_{j,n-j}(\R)\cdot s_j\,{\it o}\,$. The $\,N(\R)$-invariance of $\,C^i_j\,$ is
not immediately obvious, but is a consequence, of course, of its description as a union of Schubert cells. This
description follows inductively from the following assertion: let $\,w\,$ be an element of the normalizer of the
diagonal subgroup, $\,1 \le i \le n-1\,$, and let $\,\fn\,$ denote the Lie algebra of $\,N(\R)\,$; then
\begin{equation}
\label{flag9}
w^{-1} E_{i,i+1}w \in \fn \ \ \ \ \Longrightarrow \ \ \ \ s_i \, C_w \, \cup C_w\ = \
C_w\,\cup\,C_{s_iw}\,.
\end{equation}
To see this, recall that every $\,n\in N(\R)\,$ can be expressed uniquely as $\,n = n_i h_i(t)\,$, with $\,n_i \in
N_i(\R)\,$, $\,t\in \R\,$. Since $\,s_i\,$ normalizes $\,N_i(\R)\,$,
\begin{equation}
\label{flag10}
\begin{aligned}
C_w\ &= \ N(\R)\cdot w \, {\it o}\ = \ N_i(\R)\cdot \{h_i(t)\, w \, {\it o}\, \mid \, t\in \R\}\,,
\\
s_i \,C_w\ &= \ s_i \, N(\R)\cdot w \,{\it o}\ = \ N_i(\R)\cdot \{ s_i \, h_i(t)\, w\, {\it o}\, \mid \, t\in \R\}\,, \
\ \ \text{and}
\\
C_{s_iw}\ &= \ N(\R)\cdot s_i\, w {\it o}\ = \ N_i(\R)\cdot \{h_i(t)\, s_i \, w\, {\it o}\, \mid \, t\in \R\}\,.
\end{aligned}
\end{equation}
The hypothesis of (\ref{flag9}) implies that the isotropy subgroup of $\,\Phi_i(SL(2,\R))\,$ at the point $\,w {\it
o}\,$ coincides with the $\,\Phi_i$-image of the lower triangular subgroup of $\,SL(2,\R)\,$. Thus (\ref{flag10})
reduces (\ref{flag9}) to the corresponding statement about the flag variety of $\,SL(2,\R)\,$~-- i.e., about
$\,\R\mathbb P^1\,$~-- which is essentially obvious. As was remarked before, (\ref{flag9}) implies the description
of $\,C^i_j\,$ as a union of Schubert cells, by downward induction on $\,i\,$. The final assertion of the lemma
also follows: with $\,w\,$ and $\,i\,$ as in (\ref{flag9}), $\,N(\R)\,s_i\,C_w\,$ contains $\,N_i(\R)\cdot
\{h_i(t_1) s_i h_i(t_2) w {\it o} \mid  t_1,t_2\in \R\}\,$, hence both $\,C_{s_iw}\,$ and $\,C_w\,$.
\end{proof}

The notions of vanishing to infinite order and canonical extension can be defined for distributions on manifolds;
the locus along which the vanishing or canonical extension takes place must be a locally closed submanifold
\cite{inforder}. When the submanifold has codimension one, as is the case for the $\,C_j\,$,\,\ this can be thought
of as the one variable case with parameters. Recall that the $\,R_{j,m,k}\tau\circ \Phi_{j+1}\,$ can be regarded as
a distribution section of a line bundle over the flag variety of $\,SL(2,\R)$, i.e., over the compactified real
line $\,\R\mathbb P^1 = \R \cup \{\infty\}\,$.

\begin{lem}
\label{inforderlem1} For $\,2 \le j \le n-1\,$, $\,m \neq 0\,$,\,\ and $\,k \in \Z^{n-1}$, the following are
equivalent:\newline \noindent {\rm a)}  The
$\,R_{j-1,m,(k_1,\dots,k_{j-2},\ell,k_{j},\dots,k_{n-1})}\tau\,\big|_{X_0}$, indexed by a set of representatives
$\,\ell\,$ modulo $\,m\,$, have canonical extensions across $\,C_j\,$;\newline \noindent {\rm b)}  The
$\,R_{j-1,m,(k_1,\dots,k_{j-2},\ell,k_{j},\dots,k_{n-1})}\tau\circ \Phi_{j}\big|_\R\,$, indexed by a set of
representatives $\,\ell\,$ modulo $\,m\,$, have canonical extensions across $\,\infty\,$;\newline \noindent {\rm
c)} $\,P_{j-1,m,k}\tau\,\big|_{X_0}$ has a canonical extension across $\,C_j\,$.\newline The conditions analogous
to {\rm a) - c)}, with ``vanish to infinite order" in place of ``have canonical extensions" are also equivalent to
each other. Moreover, {\rm a)} and {\rm b)} are equivalent even when $\,j=1\,$, as are the analogous conditions
involving vanishing to infinite order. In all cases, these equivalences preserve uniformity in $\,m\,$ and $\,k\,$,
in the sense of \cite[definition 7.1]{inforder}.
\end{lem}

\begin{proof}
The fibration (\ref{flag8}) is $\,N_j(\R)$-equivariant. We can therefore identify the total space $\,X_0 \cup
C_j\,$ with $\,N_{j}(\R)\times \R\mathbb P^1$. Since $\,N_j(\R)\,$ acts on $\,R_{j-1,m,k}\,$ according to the
character $\,\chi^R_{j-1,m,k}\,$, $\,R_{j-1,m,k}\,$ can be regarded as the product of its restriction to the fiber
with the character:
\begin{equation}
\label{Pjkmtau7}
R_{j-1,m,k}\tau \,\big|_{X_0\cup C_j} \ =  \ \chi^R_{j-1,m,k} \times R_{j-1,m,k}\tau \circ \Phi_{j}
\,.
\end{equation}
Thus b) implies a), both in the ``canonical extension" and the ``vanishing to infinite order" version. The converse
is almost equally obvious.  We can use the coordinates on $\,N_{j}(\R)\,$, along with $\,x_{j}\,$, to establish
vanishing to infinite order. The vector fields $\,\ell(E_{i_1,i_2})\,$, with $\,1\leq i_1 < i_2\leq n\,$,
$(i_1,i_2) \neq (j,j+1)\,$, generate the linear differential operators on $\,N_{j}(\R)\,$, and each of them acts on
$\,R_{j,m,k}\tau\,$ as multiplication by a constant. Hence, when an expression verifying a) is written in terms of
the $\,\ell(E_{i_1,i_2})\,$ and the partial derivative with respect to $\,x_{j+1}\,$, one can restrict it to the
fiber $\,\R\mathbb P^1$ in the product $\,N_{j}(\R)\times \R\mathbb P^1$ and conclude b), and this implication also
applies to both versions of a) and b). In view of \lemref{RSlem2}, increasing the index $\,\ell\,$ by $\,m\,$ has
the effect of translating $\,R_{j-1,m,(k_1,\dots,k_{j-2},\ell,k_{j},\dots,k_{n-1})}\tau\circ \Phi_{j}\,$ by
$\,h_j(1)\,$. Hence b) implies that the $\,R_{j-1,m,(k_1,\dots,k_{j-2},\ell,k_{j},\dots,k_{n-1})}\tau\circ
\Phi_{j}\,$, for all $\,\ell\in\Z\,$, have canonical extensions, or vanish to infinite order, uniformly in
$\,\ell\,$, in the sense of \cite[definition 7.1]{inforder}. An application of \cite[Lemma 7.2]{inforder} then
implies that the sum of the terms in (\ref{Pjkmtau7}), for all $\,\ell\in\Z\,$, has a canonical extension across
$\,C_j\,$, respectively vanishes to infinite order along $\,C_j\,$. But the sum is $\,P_{j-1,m,k}\tau\,$, since
\begin{equation}
\label{Pjkmtau8} P_{j-1,m,k}\tau \ = \ {\sum}_{\ell\in\Z} \
R_{j-1,m,(k_1,\dots,k_{j-2},\ell,k_{j},\,\dots,\,k_{n-1})}\tau\,,
\end{equation}
in complete analogy to (\ref{Pjkmtau4}). Thus b) implies c), again in both versions. This last argument can be
reversed because \cite[Lemma 7.2]{inforder} asserts an equivalence between the two relevant conditions. In these
arguments, bounds embodying the uniformity in $\,m\,$ and $\,k\,$ are carried along, so uniformity is preserved by
all of the equivalences.
\end{proof}

In practice, we shall prove that a particular distribution vanishes to infinite order along a locally closed
submanifold $\,C \subset X\,$ by first showing that its restriction to the complement of $\,C\,$ has a canonical
extension across $\,C\,$. The difference between the distribution and the canonical extension is a
distribution\begin{footnote}{Technically a distribution not on $\,X\,$, but rather on any open subset $\,U\subset
X\,$ which contains $\,C\,$ as a closed submanifold.}\end{footnote} supported on $\,C\,$. That distribution is then
shown to vanish by finding a contradiction between the support condition and other properties it is known to
possess.

To make this concrete, let us consider a non-open Schubert cell $\,C_w\,$, as defined in (\ref{schw}), attached to
an element $\,w \neq e\,$ of the normalizer of the diagonal subgroup. Since $\,C_w\,$ is locally closed in $\,X\,$,
we can choose a Zariski open subset $\,U_w \subset X\,$ which contains $\,C_w\,$ as a Zariski closed subset. We
suppose that a distribution section $\,\s\,$ of $\,\mathcal L_{\l,\d}\,$ is given, with support in $\,C_w\,$,
\begin{equation}
\label{schw1} \s \in C^{-\infty}(U_w,\mathcal L_{\l,\d})\,,\ \ \ \ \supp \s \subset C_w\,.
\end{equation}
A particular example would be a distribution section on $\,C_w\,$, not of the line bundle $\,\mathcal L_{\l,\d}\,$
itself, but of its tensor product with $\,\wedge^{\text{top}}\, T_{C_w}X\,$, the top exterior power of the normal
bundle of $\,C_w\,$ in $\,X\,$. Distributions with values in $\,\mathcal L_{\l,\d}\,$ are naturally dual to smooth
measures with values in the dual line bundle $\,\mathcal L_{-\l,\d}\,$; the shift by $\,\wedge^{\text{top}}\,
T_{C_w}X\,$ compensates for the discrepancy between the transformation under coordinate changes of smooth measures
on $\,U_w\,$ on the one hand, and smooth measures on $\,C_w\,$ on the other. We should remark that the line bundle
$\,\wedge^{\text{top}}\, T_{C_w}X\,$ on $\,C_w\,$ extends to a $\,G(\R)$-equivariant line bundle on $\,X\,$, so the
shift by $\,\wedge^{\text{top}}\, T_{C_w}X\,$ merely amounts to a shift of the parameters $\,\l\,$ and $\,\d\,$.

We shall call a distribution section $\,\s\,$ of the particular type just discussed a distribution section of
``normal degree zero". On general principle, one can express an $\,\mathcal L_{-\l,\d}$-valued distribution with
support on $\,C_w\,$ as a linear combination of normal derivatives, applied to a $\,\mathcal L_{-\l,\d}$-valued
distribution of normal degree zero, though in general such expressions can be given only locally. In our
applications $\,\s\,$ will sometimes be invariant under a subgroup that acts on $\,C_w\,$ with a compact
fundamental domain. In that case such an expression exists globally, and there is a well defined ``normal order" --
i.e., the maximum number of normal derivatives that are required. Otherwise the normal order may be well defined
only locally.

In the following, we let $\,\fn\,$, $\,\fn_-\,$, $\,\fb_-\,$, and $\,\fa\,$  denote the real Lie algebras of,
respectively, $\,N(\R)\,$, the lower triangular unipotent subgroup $\,N_-(\R)\,$, the lower triangular Borel
subgroup $\,B_-(\R)\,$, and the diagonal subgroup. The unipotent group $\,N(\R)\,$ is the pointwise product, in
either order, of $\,N(\R)\cap wN(\R)w^{-1}$ and $\,N(\R)\cap wN_-(\R)w^{-1}$; the latter fixes the point $\,w\,\it
o\,$, and the former acts freely at $\,w\,\it o\,$. Thus
\begin{equation}
\label{schw2} N(\R)\cap wN(\R)w^{-1} \ \cong \ C_w \ \ \ \text{via}\ \ \ \ N(\R)\ \ni n \ \mapsto \ n\, w\, {\it o}\,.
\end{equation}
Since $\,N(\R)\cap wN(\R)w^{-1}\,$ is closed in $\,wN(\R)w^{-1}\,$, we can let its orbit at the point $\,w \it o\,$
play the role of $\,U_w\,$:
\begin{equation}
\label{schw2.5} U_w \ = \ w\,X_0 \ = \ w\,N(\R)\cdot {\it o}\ = \ \text{$wN(\R)w^{-1}$-\,orbit at} \ w\,{\it o}\,.
\end{equation}
The tangent space to $\,U_w\,$ at $\,w\,\it o\,$ is naturally isomorphic to $\,w\fn w^{-1}$, and that of $\,C_w\,$
is naturally isomorphic to $\,\fn \cap w\fn w^{-1}$;\,\ cf. (\ref{schw2}). Conjugation by the group ${\,N(\R)\cap
wN(\R)w^{-1}}$ preserves both $\,w\fn w^{-1}$ and $\,\fn \cap w\fn w^{-1}$, so these are isomorphisms not just at
$\,w\,\it o\,$, but at any point $\,n\,w\,\it o\,$ with $\,n \in N(\R)\cap wN(\R)w^{-1}$. We conclude that
\begin{equation}
\label{schw3} w\fn w^{-1}/(\fn \cap w\fn w^{-1})\ \cong \ \text{normal space to}\ \,C_w\, \ \text{at} \ \,n\, w\, {\it
o}\,,
\end{equation}
for any $\,\,n \in N(\R)\cap wN(\R)w^{-1}$. Let $\,Z_1\,, \, Z_2 \,, \, \dots \,,\, Z_M\,$ be a basis of
$\,\fn_-\cap w \fn w^{-1}$, and let $\,\ell(Z_j)\,$ denote the vector field generated by $\,Z_j\,$ under
infinitesimal left translation. Since $\,\fn_-\cap w \fn w^{-1}$ is a linear complement to $\,\fn \cap w\fn w^{-1}$
in $\,w\fn w^{-1}$,
\begin{equation}
\label{schw4} \text{the $\,\ell(Z_j)\,$,\,\ $\,1 \leq j \leq M$,\,\ generate the normal space to}\ \,C_w\, \ \text{at each
point.}
\end{equation}
We conclude: an $\,\mathcal L_{\l,\d}$-valued distribution $\,\s\,$ on $\,U_w\,$ with support on $\,C_w\,$, as in
(\ref{schw1}), can be expressed locally as
\begin{equation}
\label{schw5} \s \ \ = \ \ {\sum}_L \ \ell(Z^L)\, \s_L \,;
\end{equation}
here $\,L\,$ runs over all $\,M$-tuples $(\ell_1, \ell_2, \dots , \ell_M)$ of nonnegative integers of total length
up to the normal degree of $\,\s$, $\,\ell(Z^L)\,$ is shorthand for the ordered product
$\,\ell(Z_1)^{\ell_1}\ell(Z_2)^{\ell_2}\dots \ell(Z_N)^{\ell_N}\!$,\,\ and the $\,\s_L\,$ are $\,\mathcal
L_{\l,\d}$-valued distributions supported on $\,C_w\,$, of normal degree zero.

The functions $\,f\in C^\infty(U_w)\,$ that vanish on $\,C_w\,$ constitute an ideal $\,I_w\,$. One calls a vector
field -- always understood to have $\,C^\infty$ coefficients -- tangential to $\,C_w\,$ if the one parameter group
of diffeomorphisms it generates preserves $\,C_w\,$, or entirely equivalently,
\begin{equation}
\label{schw5.5} \text{a vector field $\,V\,$ is tangential to $\,C_w\,$ if $\,\,V I_w \subset I_w\,$}.
\end{equation}
As before, we suppose that $\,\sigma \in C^{-\infty}(U_w,\mathcal L_{\l,\d})\,$ has support in $\,C_w\,$. Then
\begin{equation}
\label{schw6}
\begin{aligned}
\!\!\text{multiplication by any $f\in I_w$ reduces the normal degree of $\sigma$ by one, and}
\\
\!\!\text{multiplication by any $\,f\in C^\infty(U_w)\,$ does not increase the normal degree.}
\end{aligned}
\end{equation}
In particular, if $\,\sigma\,$ has normal degree zero and $\,f \in I_w\,$, $\,f\,\sigma\,$ vanishes. This follows
easily from the definitions.

By infinitesimal left translation, any element $\,Z\,$ of the Lie algebra of $\,G(\R)\,$ acts on sections of the
$\,G(\R)$-equivariant line bundle $\,\mathcal L_{\l,\d}$. Arbitrary vector fields, on the other hand, do not
obviously act on sections of $\,\mathcal L_{\l,\d}$. Like any line bundle, $\,\mathcal L_{\l,\d}$ can be locally
trivialized, and any two local trivializations are related by multiplication with a $\,C^\infty\,$ function without
zeroes. In view of (\ref{schw6}), multiplication by such a function does not affect the normal degree. Therefore,
without loss of generality, we may as well suppose that $\,\sigma\,$ is a scalar distribution. We note:
\begin{equation}
\label{schw7}
\begin{aligned}
&\text{if the vector field $\,V\,$ is tangential to $\,C_w\,$, the normal}
\\
&\text{degree of $\,V\sigma\,$ does not exceed the normal degree of $\,\sigma\,$.}
\end{aligned}
\end{equation}
Like (\ref{schw6}) this follows directly from the definitions. The Schubert cell $\,C_w\,$ is not only an
$\,N(\R)$-orbit, but is also invariant under the diagonal subgroup, which fixes the point $\,w\,{\it o}\,$ and
normalizes $\,N(\R)$. Thus
\begin{equation}
\label{schw8} \text{the vector field $\ell(Z)$, for any $\,Z\in \fa \oplus \fn\,$, is tangential to $\,C_w\,$,}
\end{equation}
since the one parameter group generated by $\,\ell(Z)$ preserves $\,C_w\,$. For lack of a better term, we shall
call a vector field $\,V\,$ {\it hypertangential\/} to $\,C_w\,$ if
\begin{equation}
\label{hypertan} V I_w^k \ \subset \ I_w^{k+1}\ \ \ \text{for all $\,k \geq 0\,$};
\end{equation}
this is not standard terminology, however.

\begin{lem}
\label{hypertanlem} A vector field $\,V\,$ is hypertangential to $\,C_w\,$ if and only if both $\,V\,$ and its
commutator $\,[V,W]\,$ with any other vector field $\,W\,$ are tangential to $\,C_w\,$. If the vector field $\,V$
is hypertangential to $\,C_w\,$, $\,V\sigma\,$ has strictly lower normal degree than $\,\sigma\,$;\,\ in
particular, when $\,\sigma\,$ has normal degree zero, then $\,V \sigma\,$ must vanish.
\end{lem}

\begin{proof}
We choose local coordinates $\,x_1, x_2, \dots , x_m, y_1, \dots , y_n\,$ on $\,X\,$ so that $\,C_w\,$ is the set
of common zeroes of the $\,y_j\,$, $\,1 \leq j \leq n\,$. When we express $\,V\,$ as
\begin{equation}
\label{hypertan1} V \ \ = \ \ {\sum}_{1\leq i \leq m}\, a_i\, {\textstyle{\frac {\partial \ }{\partial x_i}}} \ + \
{\sum}_{1\leq j \leq n}\, b_j\, {\textstyle{\frac {\partial \ }{\partial y_j}}} \,,
\end{equation}
tangentiality is characterized by the condition $\,b_j \in I_w\,$ for all $\,j\,$, whereas hypertangentiality
translates into the conditions $\,a_i \in I_w\,$, $\,b_j \in I^2_w\,$, for all indices $\,i\,$ and $\,j\,$. Using
these characterizations, one obtains the alternative description of hypertangentiality by computing the commutators
of $\,V\,$ with the $\,{\frac {\partial \ }{\partial y_\ell}}\,$ and the $\,x_k{\frac {\partial \ }{\partial
y_\ell}}\,$. Since
\begin{equation}
\label{hypertan2} V \ \ = \ \ {\sum}_{1\leq i \leq m}\, \big( {\textstyle{\frac {\partial \ }{\partial x_i}}} \circ a_i
- {\textstyle{\frac {\partial a_i}{\partial x_i}}}\big) \ + \ {\sum}_{1\leq j \leq n}\, \big( {\textstyle{\frac
{\partial \ }{\partial y_j}}} \circ b_j - {\textstyle{\frac {\partial b_j}{\partial y_j}}}\big)\,,
\end{equation}
and since $\,\frac {\partial \ }{\partial x_i} I_w \subset I_w\,$,\,\ a hypertangential vector field $\,V\,$ can be
expressed as a linear combination $\,\sum_\ell W_\ell \circ f_\ell + g\,$ with $\,f_\ell, g \in I_w\,$.  At this
point (\ref{schw6}) implies the second assertion of the lemma.
\end{proof}

We had remarked earlier that any $\,n\in wN(\R)w^{-1}\,$ can be expressed uniquely as a product $\,n = n_1\,n_2\,$,
with $\,n_1 \in wN(\R)w^{-1}\cap N(\R)\,$, $\,n_2 \in wN(\R)w^{-1}\cap N_-(\R)\,$. Thus, for any $\,Z\,$ in the Lie
algebra of $\,G(\R)$, we can define a vector field $\,m(Z)\,$ on $\,U_w = w N(\R)\cdot {\it o}\,$ by the formula
\begin{equation}
\label{mvf}
\begin{aligned}
\big( m(Z)f \big)(n_1\,n_2\,w\,{\it o}) \ = \ {\textstyle{\frac{\partial \ }{\partial t}}}\, f(n_1
\exp(-t\,Z)\,n_2\,w\,{\it o}) \big|_{t=0}\ \ \text{if $\,f \in C^\infty(U_w)\,$}
\\
\text{and $\,n_1 \in wN(\R)w^{-1}\cap N(\R)\,$, $\,n_2 \in wN(\R)w^{-1}\cap N_-(\R)\,$}.
\end{aligned}
\end{equation}
By construction,
\begin{equation}
\label{mvf1} \text{the vector fields $\,m(Z)\,$ are $\,\big(wN(\R)w^{-1}\cap N(\R)\big)$-invariant.}
\end{equation}
Since $\,\ell(Z)f(n_1 n_2 w {\it o})\,$ is the derivative at $\,t=0\,$ of $\,f(\exp(-t Z)n_1 n_2 w {\it o})\,$, we
can describe the value $\,m(Z)|_{n_1 n_2 w {\it o}}\,$ of $\,m(Z)\,$ at the point $\,n_1 n_2 w {\it o}\,$ as
follows:
\begin{equation}
\label{mvf2} m(Z)\big|_{n_1 n_2 w {\it o}} \ = \ \ell(\widetilde Z)\big|_{n_1 n_2 w {\it o}} \ \ \ \text{with} \ \
\widetilde Z = \Ad(n_1) Z \,;
\end{equation}
here, as before, $\,n_1 \in wN(\R)w^{-1}\cap N(\R)\,$, $\,n_2 \in wN(\R)w^{-1}\cap N_-(\R)\,$.

\begin{lem}
\label{hypertanlem2} {\rm {\bf a)}}\, \ $m(Z)\,$ is tangential to $\,C_w\,$ if $\,Z \in \fa \oplus \fn\,$ or $\,Z
\in w\,\fn_-\,w^{-1}\,$.\newline \noindent {\rm {\bf b)}}\, \ $Z_1, Z_2 \in w\,\fn \, w^{-1} \cap \fn_- \ \ \
\Longrightarrow \ \ \ [m(Z_1),m(Z_2)]\, = \, m([Z_1,Z_2])\,$.\newline \noindent {\rm {\bf c)}}\, \ $Z_1, Z_2 \in
w\,\fn \, w^{-1} \cap \fn \ \ \ \Longrightarrow \ \ \ [m(Z_1),m(Z_2)]\, = \, -\,m([Z_1,Z_2])\,$.\newline \noindent
{\rm {\bf d)}}\, \ $Z_1 \in w\,\fn \, w^{-1} \cap \fn_- \,,\ \ Z_2 \in w\,\fn \, w^{-1} \cap \fn \ \ \
\Longrightarrow \ \ \ [m(Z_1),m(Z_2)]\, = \, 0\,$.\newline \noindent {\rm {\bf e)}}\, \ $Z_1 \in w\,\fn_- \, w^{-1}
\cap \fn \,,\ \ Z_2 \in w\,\fn \, w^{-1} \cap \fn_- \ \ \ \ \Longrightarrow \ \ \ \ $the vector field\newline
$\left. \ \ \ \ \,\, [m(Z_1),m(Z_2)]- m([Z_1,Z_2])\right.$\,\ is tangential to $\,C_w\,$.\newline \noindent {\rm
{\bf f)}}\, \ $Z_1 \in w\,\fn_- \, w^{-1} \cap \fn \,,\ \ [\,Z_1 \, , \,w\,\fn w^{-1} \cap \fn_-\,] \ \subset \ \fa
+ \fn + w\,\fn_-\,w^{-1} \ \ \ \ \Longrightarrow$ \newline $\left. \ \ \ \ \,\, m(Z_1)\right.$\,\ is
hypertangential to $\,C_w\,$.
\end{lem}

\begin{proof}
The one parameter group $\,\exp(-tZ)\,$, with $\,Z \in w\,\fn_-\,w^{-1}\,$, fixes the point $\,w\,{\it o}\,$. We
conclude that the value $\,m(Z)|_{n_1 w {\it o}}\,$ of the vector field $\,m(Z)\,$ at any point $\,n_1w {\it o}\in
C_w\,$ --\,\ i.e., when $\,n_2\! =\! e\,$ --\,\ vanishes. That is an even stronger condition than tangentiality.
When $\,Z \in \fa \oplus \fn\,$, on the other hand, the values $\,m(Z)|_{n_1 w {\it o}}\,$ along $\,C_w\,$ are
tangential to $\,C_w\,$ but possibly non-zero; in that case, too, $\,m(Z)\,$ is tangential to $\,C_w\,$. That
implies a). For $\,Z_1 \in w\,\fn \, w^{-1} \cap \fn_-\,$ and $\,n_2 \in wN(\R)w^{-1}\cap N_-(\R)\,$, the product
$\,\exp(-tZ_1)n_2\,$ also lies in $\,wN(\R)w^{-1}\cap N_-(\R)\,$. Hence, for any $\,Z_2\,$ in the Lie algebra of
$\,G(\R)$, any $\,n_1 \in wN(\R)w^{-1}\cap N(\R)\,$ and any $\,f\in C^\infty(U_w)$,
\begin{equation}
\label{mvf3}
\begin{aligned}
\big( m(Z_1) m(Z_2)f \big)(n_1\,n_2\,w\,{\it o}) \ = \ {\textstyle{\frac{\partial \ }{\partial t}}}\,\big(m(Z_2) f
\big)(n_1 \exp(-t\,Z_1)\,n_2\,w\,{\it o}) \big|_{t=0}
\\
=\ \ {\textstyle{\frac{\partial^2 \, \ }{\partial t \partial s}}}\,f(n_1 \exp(-s Z_2)\exp(-t\,Z_1)\,n_2\,w\,{\it o})
\big|_{s=t=0}\,.
\end{aligned}
\end{equation}
Similarly, with $\,Z_1 \in w\,\fn \, w^{-1} \cap \fn\,$ and $\,n_1 \in wN(\R)w^{-1}\cap N(\R)\,$,
$\,n_1\exp(-tZ_1)\,$ lies in $\,wN(\R)w^{-1}\cap N(\R)\,$. Hence, for any $\,Z_2\,$ in the Lie algebra of
$\,G(\R)$, any $\,n_2 \in wN(\R)w^{-1}\cap N_-(\R)\,$ and any $\,f\in C^\infty(U_w)$,
\begin{equation}
\label{mvf4}
\begin{aligned}
\big( m(Z_1) m(Z_2)f \big)(n_1\,n_2\,w\,{\it o}) \ = \ {\textstyle{\frac{\partial \ }{\partial t}}}\,\big(m(Z_2) f
\big)(n_1 \exp(-t\,Z_1)\,n_2\,w\,{\it o}) \big|_{t=0}
\\
=\ \ {\textstyle{\frac{\partial^2 \, \ }{\partial t \partial s}}}\,f(n_1 \exp(-t Z_1)\exp(-s\,Z_2)\,n_2\,w\,{\it o})
\big|_{s=t=0}\,.
\end{aligned}
\end{equation}
These two identities imply b) - d).

For the proof of e), we fix $Z_1 \in w\,\fn_- \, w^{-1} \cap \fn \,$ and $\,Z_2 \in w\,\fn \, w^{-1} \cap \fn_-\,$.
As $\,Y_j\,$ runs over a basis of $\,w\,\fn\,w^{-1}\,$, the values of the vector fields $\,m(Y_j)\,$ at any point
of $\,U_w\,$ span the tangent space. We can therefore write
\begin{equation}
\label{mvf5}
\begin{aligned}
&m(Z_1)\ = \ {\sum}_i \ a_i\, m(Y^+_i) \ + \ {\sum}_j \ b_j\, m(Y^-_j) \,,
\\
&\ \ \ \text{with}\ \ \ Y^+_i \in w\,\fn \, w^{-1} \cap \fn \,,\ \ Y^-_j \in w\,\fn \, w^{-1} \cap \fn_- \,,\ \ \ a_i ,
b_j \in C^\infty(U_w) \,;
\end{aligned}
\end{equation}
this expression becomes unique when we assume, as we may, that the $\,Y^+_i\,,\, Y^-_j\,$ are linearly independent.
Then
\begin{equation}
\label{mvf6}
\begin{aligned}
\left. [ m(Z_1),m(Z_2)] \  \right. &= \ {\sum}_i \ \big( a_i\, [m(Y^+_i),m(Z_2)] - (m(Z_2)a_i) \, m(Y^+_i) \,\big)
\\
&\ \ \ \ \ + \ {\sum}_j \ \big( b_j\, [m(Y^-_j),m(Z_2)] - (m(Z_2)b_j)\, m(Y^-_j) \,\big)\,.
\end{aligned}
\end{equation}
On the other hand, for $\,f\in C^\infty(U_w)$,
\begin{equation}
\label{mvf7}
\begin{aligned}
&(m([Z_1,Z_2] )f )(w \, {\it o})\ = \ {\textstyle{\frac{\partial \ }{\partial t}}} f(\exp(-t[Z_1,Z_2])\,w\,{\it
o})\big|_{t=0}
\\
&= \ {\textstyle{\frac{\partial^2 \, \ }{\partial t \partial s}}} \big( f(\exp(-s Z_2)\exp( -t Z_1)\,w\,{\it o}) -
f(\exp(-t Z_1)\exp( -s Z_2) \,w\,{\it o})\big)\big|_{s=t=0} \!\!\!\!\!\!\!\!\!
\\
&= \ -\, {\textstyle{\frac{\partial^2 \, \ }{\partial t \partial s}}} \,f(\exp(-t Z_1)\exp( -s Z_2) \,w\,{\it
o})\big|_{s=t=0}
\\
&= \ -\, {\textstyle{\frac{\partial \ }{\partial s}}}\big( m(Z_1) f\big)(\exp(-s\,Z_2)\,w\,{\it o})\big|_{s=0}
\\
&= \ -\, \big( m(Z_2)m(Z_1) f\big)(w\,{\it o}) \,;
\end{aligned}
\end{equation}
at the second step we have used the identity
\begin{equation}
\label{mvf8} \exp(sZ_2)\exp(tZ_1) = \exp ( sZ_2 + t Z_1 + {\textstyle{\frac 12}} st[Z_2,Z_1] + s^2 \cdots + t^2 \cdots
+ \cdots )\,,
\end{equation}
at the third, the fact that $\,\exp(-tZ_1)w{\it o} \equiv w {\it o}\,$, and at the last two steps the definition of
$\,m(Z)\,$. We now substitute the expression (\ref{mvf5}) for $\,m(Z_1)\,$ and note that $\,m(Z_1)|_{w{\it o}} =
0\,$, hence $\,a_i(w{\it o})= 0\,$ and $\,b_j(w{\it o})= 0\,$:
\begin{equation}
\label{mvf9}
\begin{aligned}
(m([Z_1,Z_2] f )(w \, {\it o})\ = \ &- {\sum}_i \, \big( m(Z_2)\,a_i \big) (w \, {\it o})\big(m(Y^+_i)f\big)(w \, {\it
o})
\\
& \qquad - {\sum}_j \, \big( m(Z_2)\,b_j\big)(w \, {\it o})\big(m(Y^-_j)f\big)(w \, {\it o}) \,.
\end{aligned}
\end{equation}
Comparing this to (\ref{mvf6}), and using the vanishing of the $\,a_i\,$ and $\,b_j\,$ at $\,w{\it o}\,$, we see
that the values at $\,w{\it o}\,$ of the two vector fields $\,[m(Z_1),m(Z_2)]\,$ and $\,m([Z_1,Z_2])\,$ coincide.
Both vector fields are invariant under $\,w N(\R) w^{-1} \cap N(\R)\,$, which acts transitively on $\,C_w\,$, so
their difference vanishes along $\,C_w\,$. It is therefore tangential to $\,C_w\,$, as asserted by e).

We use the criterion in \lemref{hypertanlem} to verify f). By a) $\,m(Z_1)\,$ is tangential to $\,C_w\,$. It
remains to be shown that the commutators of $\,m(Z_1)\,$ with $\,a\,m(Y^+)\,$ and $\,b\,m(Y^-)\,$ are tangential to
$\,C_w\,$, for any $\,Y^+ \in w \fn w^{-1} \cap \fn\,$, $\,Y^- \in w \fn w^{-1} \cap \fn_-\,$, and $\,a, b \in
C^\infty(U_w)\,$. Both $\,m(Z_1)\,$ and $\,a\,m(Y^+)\,$ are tangential to $\,C_w\,$ by a), hence so is their
commutator. Because of a), e) and the hypotheses on $\,Z_1\,$, the commutator of $\,m(Z_1)\,$ and $\,m(Y^-)\,$ is
also tangential to $\,C_w\,$. It remains to be shown that $\,\big( m(Z_1) b\big) m(Y^-)\,$ is tangential to
$\,C_w\,$, or equivalently that $\,m(Z_1) b \,$ vanishes along $\,C_w\,$ --\,\ or, in view of the
$\,(wN(\R)w^{-1}\cap N(\R))$-invariance, that $\,(m(Z_1) b)(w{\it o}) = 0 \,$. But this is clear: $\,Z_1 \in w
\fn_- w^{-1}\,$= isotropy subalgebra at $\,w{\it o}\,$, so the value $\,m(Z_1)|_{w{\it o}}\,$ of the vector field
$\,m(Z_1)\,$ at $\,w{\it o}\,$ is zero.
\end{proof}

\begin{prop}
\label{hypertanprop} We consider a distribution section $\,\sigma\,$ of $\,\mathcal L_{\l,\d}\,$ with support in
$\,C_w\,$, as in (\ref{schw}), and some $\,Z \in w \fn_- w^{-1} \cap \fn\,$, subject to the following
hypotheses:\newline \noindent {\rm {\bf i)}}\ \ \ $[\,Z \, , \,w\,\fn w^{-1} \cap \fn_-\,] \ \subset \ \fa + \fn +
w\,\fn_-\,w^{-1}\,${\rm ;}\newline \noindent {\rm {\bf ii)}}\,\,\  $\ell(Z) - m(Z)\,$ annihilates $\,\sigma\,$.
\newline \noindent Then $\,\ell(Z)\,$ lowers the normal degree of $\,\sigma\,$ by one.
\end{prop}

The hypothesis ii) requires explanation. The meaning of the action of $\,\ell(Z)\,$ on $\,\sigma\,$ is clear, not
just for the specific $\,Z\,$ in the proposition, but for any $\,Z\,$ in the Lie algebra of $\,G(\R)\,$. The
formula (\ref{mvf}), in contrast, only defines $\,m(Z)\,$ as a vector field on $\,U_w\,$. To attach meaning to
$\,m(Z)\sigma\,$, we express $\,m(Z)\,$ as a linear combination $\,m(Z) = \sum_j \, a_j\,\ell(Z_j)\,$, with
$\,Z_j\,$ running over a basis of $\,w\fn w^{-1}\,$, and with coefficients $\,a_j \in C^\infty(U_w)\,$; we
interpret $\,m(Z)\sigma\,$ as equivalent to $\,\sum_j \, a_j\,\ell(Z_j)\sigma\,$.

In our applications, $\,\sigma\,$ has invariance properties that rule out a strictly lower normal degree for
$\,\ell(Z)\sigma\,$, and that will allow us to conclude $\,\sigma\,$ must vanish. As mentioned in the introduction,
this is closely related to the main mechanism of proof in the paper \cite{chm} of Casselman-Hecht-Mili{\v c}i{\'c}.

\begin{proof}
We remarked earlier that we may suppose, without loss of generality, that $\,\sigma\,$ is a scalar distribution.
According to \lemref{hypertanlem2} $\,m(Z)\,$ is hypertangential to $\,C_w\,$. Thus, by \lemref{hypertanlem},
$\,\ell(Z)\,\sigma\,$ has strictly lower normal degree than $\,\sigma\,$.
\end{proof}

We return to our earlier notation, with $\,\tau \in C^{-\infty}(X,\mathcal L_{\l,\d})^{G(\Z)}$, and with the
$\,R_{j,m,k}\tau$ and $\,S_{j,m,k}\tau$ as defined in section \ref{sec:background}. Recall the definition of the
codimension one Schubert cells $\,C_j = C_{s_j}\,$.

\begin{lem}
\label{inforderlem3} For $\,1 \leq j \leq n-2\,$ and all choices of $\,m\,$ and $\,k\,$, $\,R_{j,m,k}\tau$ vanishes
to infinite order along the codimension one Schubert cells $\,C_{\tilde j}\,$, provided $\,\tilde j\,$ equals
neither $\,j\,$ nor $\,j+1\,$. In the case of $\,j=0\,$, $\,R_{0,1,k}\tau$ vanishes to infinite order along all the
$\,C_{\tilde j}\,$. In both cases, the vanishing to infinite order is uniform in $\,m\,$ and the multi-index
$\,k\,$, in the sense of \cite[definition 7.1]{inforder}.
\end{lem}

\begin{proof}
We begin with an auxiliary result. In all cases covered by the statement of the lemma, with the single exception of
$\,j=0\,$, $\,\tilde j=1\,$,
\begin{equation}
\label{cuspi} k_{\tilde j} \ = \ 0 \ \ \ \ \Longrightarrow\ \ \ \ R_{j,m,k}\tau \ = \ 0 \,;
\end{equation}
in the remaining exceptional case, $\,\ell(h_1(1))\,$ acts as the identity on $\,R_{0,1,k}\tau\,$, and
\begin{equation}
\label{cuspi0} \int_{\R/\Z} \ell(h_1(t))\, R_{0,1,k}\tau \, dt  \ = \ 0 \,.
\end{equation}
Indeed, by (\ref{Rjkmtau}),
\begin{equation}
\label{cuspii}
\begin{aligned}
&\!\!R_{j,m,k}\tau \ = \ \int_{\left\{\begin{smallmatrix}{(x,y)\in(\R/\Z)^{2n-3}} \\
{x_{j+1}\,=\,0\, \ \ \ \ \ \ \ \ }\end{smallmatrix}\right\}} e(k\cdot x + my_j)\,\, \ell(n_{x,y})\tau''\, dx\, dy
\\
&\ = \, \int_{\left\{\begin{smallmatrix}{(x,y)\in(\R/\Z)^{2n-3}} \\
{x_{j+1}\,=\,0\, \ \ \ \ \ \ \ \ }\end{smallmatrix}\right\}} \int_{N''(\R)/N''(\Z)} e(k\cdot x + my_j)\,\,
\ell(n_{x,y}\,n'')\tau \, dn'' \,dx\, dy \,.
\end{aligned}
\end{equation}
Let $\,U_{\tilde j,n-\tilde j}(\R)\subset N(\R)\,$ denote the unipotent radical of the standard upper para\-bolic
of type $\,\tilde j \times (n-\tilde j)\,$. The double integral in (\ref{cuspii}) can be combined into a single
integral. When $\,(j,\tilde j )\neq (0,1)\,$, the resulting integration can be performed by first integrating over
$\,U_{\tilde j,n-\tilde j}(\R)/U_{\tilde j,n-\tilde j}(\Z)\,$, then over the remaining variables. The hypotheses,
including the assumption that $\,k_{\tilde j}=0\,$, ensure that the character $\,e(k\cdot x + my_j)\,$ is
identically equal to one on $\,U_{\tilde j,n-\tilde j}(\R)\,$. Hence the cuspidality of $\,\tau\,$ implies
\begin{equation}
\label{cuspiii}
\begin{aligned}
\int_{U_{\tilde j,n-\tilde j}(\R)/U_{\tilde j,n-\tilde j}(\Z)}\  e(k\cdot x + my_j)\,\, \ell(n)\tau\, dn \ \ &=
\\
= \ \ \int_{U_{\tilde j,n-\tilde j}(\R)/U_{\tilde j,n-\tilde j}(\Z)}\ \ell(n)\tau\, dn \ \ &= \ \ 0\,.
\end{aligned}
\end{equation}
and that, in turn, implies the vanishing of the integral (\ref{cuspii}). When $\,j=0\,$, $\,\tilde j = 1\,$, the
same argument applies, provided we first average $\,\ell(h_1(t))\, R_{0,1,k}\tau\,$ over $\,\R/\Z\,$.

As the first step towards the proof of vanishing to infinite order, we show that the restrictions of the
$\,R_{j,m,k}\tau$ to the open Schubert cell $\,X_0\,$ have canonical extensions across the $\,C_{\tilde j}\,$,
$\,\tilde j \neq j,\,j+1\,$, and in the case of $\,R_{0,1,k}\tau$ across all the $\,C_{\tilde j}\,$. The argument
is slightly different, and also slightly simpler, for $\,R_{0,1,k}\tau$. We use (\ref{flag8}) to identify $\,X_0\,$
with $\,N_{\tilde j}(\R)\times \R\,$ and simultaneously, $\,X_0\cup C_{\tilde j}\,$ with $\,N_{\tilde j}(\R)\times
\R\mathbb P^1$. In view of (\ref{tauabelian}), (\ref{tauabelexp}), (\ref{tau'3}), and (\ref{Rjk1tau}),
$\,R_{0,1,k}\tau \big|_{X_0}\,$ can then be regarded as a Fourier series on $\,N_{\tilde j}(\R)\times \R\,$ --
which happens to be constant in the entries corresponding to $\,N'(\R)\,$ -- whose expansion in the variable
$\,x_{\tilde j}\,$,\,\ corresponding to the factor $\,\R\,$, has zero constant term. Thus \cite[proposition
2.19]{inforder} applies directly: $\,R_{0,1,k}\tau \big|_{X_0}\,$ has a canonical extension across $\,C_{\tilde
j}\,$, at least as a scalar distribution. However, the discrepancy between a distribution section of $\,\mathcal
L_{\l,\d}\,$ and a scalar distribution on the complement of $\,C_{\tilde j}\,$, in terms of coordinates valid along
$\,C_{\tilde j}\,$, is a factor of the type $\,|x_{\tilde j}|^\nu\,(\sgn x_{\tilde j})^\eta$, which does not affect
the notion of vanishing to infinite order; cf. (\ref{inf1}). The uniformity in the sense of \cite{inforder},
finally, follows from the fact that the Fourier coefficients of the distribution $\,\tau_{\text{abelian}}\,$ --
like those of any periodic distribution -- are bounded by some polynomial in the length $\,\|k\|\,$ of the
multi-index $\,k\,$.

We now suppose $\,1 \leq j \leq n-2\,$, $\,1 \leq \tilde j \leq n-1\,$, $\,\tilde j \neq j,\,j+1\,$. Under these
conditions we want to show that $\,R_{j,m,k}\tau \big|_{X_0}\,$ has a canonical extension across $\,C_{\tilde
j}\,$. Just as in the case of $\,R_{0,1,k}\tau\,$, we may as well regard $\,R_{j,m,k}\tau \big|_{X_0}\,$ as a
scalar distribution. We consider the fibration (\ref{flag8}) with $\,j\,$ replaced by $\,j+1\,$, which then allows
us to identify
\begin{equation}
\label{schw9} R_{j,m,k}\tau \big|_{X_0}\ \cong \ \chi^R_{j,m,k} \times R_{j,m,k}\tau \circ \Phi_{j+1} \big|_\R
\end{equation}
as in (\ref{Pjkmtau7}). Except for a translation and change of sign, $\,R_{j,m,k}\tau \circ \Phi_{j+1} \big|_\R\,$
coincides with $\,\rho_{j,m,k}\,$ (\ref{rhosigma}), which is a tempered distribution by proposition
\ref{sigmarhoftprop}. We can therefore express $\,R_{j,m,k}\tau \circ \Phi_{j+1} \big|_\R\,$ as a sufficiently high
derivative of a continuous function of polynomial growth:
\begin{equation}
\label{schw10} R_{j,m,k}\tau \circ \Phi_{j+1} \big|_\R\ = \ {\textstyle\frac{d^r\ }{dx^r}}\,f \qquad \text{with}\ \ \
|f(x)|\, = \, O(|x|^s)\ \ \ \text{as}\ \ |x| \to \infty\,,
\end{equation}
for some $\,r,s \in \N\,$. The variable $\,x\,$ in this identity is really $\,x_{j+1}\,$, the $(j+1,j+2)$ matrix
entry of $\,n_{x,y}\,$ (\ref{nxy}), and differentiation in this direction is the vector field
$\,\ell(-E_{j+1,j+2})\,$, i.e., infinitesimal translation by $\,-E_{j+1,j+2}\,$\,\ --\,\ the inverse in
(\ref{ellofg}) accounts for the minus sign. The diffeomorphism
\begin{equation}
\label{schw11} X_0 \ \cong \ N(\R)\ \cong N_{j+1}(\R) \times \R\,,
\end{equation}
which underlies the identification (\ref{schw9}), involves left multiplication by the $\,N_{j+1}(\R)$ factor. Hence
(\ref{schw10}) can be re-written as follows:
\begin{equation}
\label{schw12} R_{j,m,k}\tau (n_{x,y}n'') \ = \ \ell(-\Ad(n_{x,y}n'')E_{j+1,j+2})^r \(\chi^R_{j,m,k} \times
f\)(n_{x,y}n'')\,,
\end{equation}
for $\,(x,y)\in \R^{n-1}\times \R^{n-2}\,$  and $\,n''\in N''(\R)\,$. A simple calculation shows
\begin{equation}
\label{schw13} \Ad(n_{x,y}n'')E_{j+1,j+2}\ \ \equiv \ \ E_{j+1,j+2} + x_j\,E_{j,j+2}\ \ \ \text{modulo}\ \
\operatorname{Ker}\chi^R_{j,m,k} \,.
\end{equation}
Since $\,E_{j+1,j+2}\,$ and $\,E_{j,j+2}\,$ commute, and since $\,\ell(E_{j,j+2})\,$ acts on $\,\chi^R_{j,m,k}\,$
as multiplication by $-2\,\pi\,i\,m\,$,
\begin{equation}
\label{schw14} R_{j,m,k}\tau (n_{x,y}n'') \ = \ (2\,\pi\,i\,m\,x_j \,- \, \ell(E_{j+1,j+2}))^r \(\chi^R_{j,m,k} \times
f\)(n_{x,y}n'')\,.
\end{equation}
Since $\,{\tilde j} \neq j\,$, the function $\,x_j\,$ is smooth along $\,C_{\tilde j}\,$ -- cf. (\ref{flag8}), with
$\,{\tilde j}\,$ in place of $\,j\,$. That makes $\,(2\,\pi\,i\,m\,x_j - \ell(E_{j+1,j+2}))\,$ a linear
differential operator with $C^\infty$ coefficients on some neighborhood\begin{footnote}{On the complement of the
closure of $\,C_j\,$ in $\,X\,$, in fact.}\end{footnote} of $\,C_{\tilde j}\,$. The action of such a differential
operator does not affect vanishing to infinite order along $\,C_{\tilde j}\,$. Hence, to show that
$\,R_{j,m,k}\tau,$ has a canonical extension across $\,C_{\tilde j}\,$, it suffices to establish the same fact for
$\,\chi^R_{j,m,k} \times f\,$.

We may suppose $\,k_{\tilde j} \neq 0\,$ by (\ref{cuspi}), and the vector field $\,\ell(E_{{\tilde j},{\tilde
j}+1})\,$ acts on $\,\chi^R_{j,m,k}\,$ as multiplication by the factor $\,- 2\,\pi\,i\,k_{\tilde j}\,$. Thus, for
any $\,r \in \N\,$,
\begin{equation}
\label{schw15} \chi^R_{j,m,k}\times f\ \ = \ \ (- 2\,\pi\,i\,k_{\tilde j})^{-r}\, \ell(E_{{\tilde j},{\tilde
j}+1})^r\,\chi^R_{j,m,k}\times f\,.
\end{equation}
We use (\ref{flag8}), with $\,{\tilde j}\,$ in place of $\,j\,$, to identify $\,X_0 \cup C_{\tilde j} \cong
N_{\tilde j} \times \R\mathbb P^1$. Then $\,x_{\tilde j}\,$ becomes a coordinate on the fiber, which takes the
value $\,\infty\,$ exactly along $\,C_{\tilde j}\,$. Arguing as we did between (\ref{schw10}) and (\ref{schw12}),
we find
\begin{equation}
\label{schw16} {\textstyle\frac{\partial\ }{\partial x_{\tilde j}}} \ \ = \ \ \ell(-\Ad(n_{x,y}n'')E_{{\tilde
j},{\tilde j}+1}) \ \ \ \ \ (\, n'' \in N''(\R)\,).
\end{equation}
Since $\,\tilde j \neq j\,,j+1\,$, $\,Ad(n_{x,y}n'')E_{{\tilde j},{\tilde j}+1} \equiv E_{{\tilde j},{\tilde
j}+1}\,$ modulo the kernel of $\,\chi^R_{j,m,k}\,$. That makes (\ref{schw15}) equivalent to
\begin{equation}
\label{schw17} \chi^R_{j,m,k}\times f\ \ = \ \ (2\,\pi\,i\,k_{\tilde j})^{-r}\, {\textstyle\frac{\partial^r\ }{\partial
x_{\tilde j}^r}}\(\chi^R_{j,m,k}\times f\)\,.
\end{equation}
According to (\ref{schw10}) the function $\,\chi^R_{j,m,k} \times f\,$ is locally bounded along $\,C_{\tilde j}\,$.
Since $\,\frac{\partial\ }{\partial x_{\tilde j}}\,$ vanishes to second order\begin{footnote}{In terms of the
coordinate change $\,t = 1/x_{\tilde j}\,$,\,\ $\,\frac{\partial\ }{\partial x_{\tilde j}}= - t^2\,\frac{\partial\
}{\partial t}$.}\end{footnote} on $\,\{x_{\tilde j}=\infty\} = C_{\tilde j}\,$, the criterion of \cite[definition
2.4]{inforder} applies: $\,R_{j,m,k}\tau\,$ has a canonical extension across $\,C_{\tilde j}\,$, as asserted.

To establish the uniformity in the sense of \cite{inforder}, we only need to argue that the bound implicit in
(\ref{schw10}) holds uniformly in $\,m\,$ and $\,k\,$, with a bounding constant which depends polynomially on
$\,m\,$ and $\,\|k\|\,$. But this is a consequence of the fact that the $\,\chi^R_{j,m,k}\,$ are the Fourier
coefficients of a globally defined distribution.

To finish the proof we must show that $\,R_{j,m,k}\tau \,$ vanishes to infinite order on $\,C_{\tilde j}\,$, or
equivalently, that the difference
\begin{equation}
\label{schw18} s_{j,m,k,\tilde j} \ \ = \ \ R_{j,m,k}\tau  \,\ - \,\ \text{canonical extension of }\, R_{j,m,k}\tau \,\
\text{across $\,C_{\tilde j}$}
\end{equation}
vanishes. We note that $\,s_{j,m,k,\tilde j}\,$ is a $\,\mathcal L_{\l,\d}$-valued distribution, defined on a
neighborhood of $\,C_{\tilde j}\,$, specifically on $\,X_0 \cup C_{\tilde j}\,$. By definition,
\begin{equation}
\label{schw19} s_{j,m,k,\tilde j} \ \ \text{is supported on $\,C_{\tilde j}\,$.}
\end{equation}
Because of the canonical nature\begin{footnote}{Specifically (\ref{inf1}).}\end{footnote} of the canonical
extension, $\,s_{j,m,k,\tilde j}\,$ inherits invariance properties from $\,R_{j,m,k}\tau\,$. In particular,
\begin{equation}
\label{schw21} \ell(n)\,s_{j,m,k,\tilde j} \ = \ \chi^R_{j,m,k}(n^{-1})\,s_{j,m,k,\tilde j}\ \ \ \text{for}\ \ n \in
N_{j+1}(\R) \,;
\end{equation}
because $\,R_{j,m,k}\tau\,$ satisfies the corresponding equation; cf. (\ref{Rjkmtau2}). On the infinitesimal level,
this implies
\begin{equation}
\label{schw22} \ell(E_{\tilde j, \tilde j + 1})\,s_{j,m,k,\tilde j} \ = \ -2\,\pi\,i\,k_{\tilde j}\,s_{j,m,k,\tilde j}
\,,
\end{equation}
except when $\,j=0\,$, $\,\tilde j = 1\,$; this exceptional case will be treated afterwards. We shall show that
(\ref{schw22}) forces $\,s_{j,m,k,\tilde j} = 0\,$, by applying proposition \ref{hypertanprop} with $\,w=s_{\tilde
j}\,$ and $\,Z = E_{\tilde j,\tilde j+1}\,$. In this situation $\,w \fn_- w^{-1} \cap \fn \,$ is spanned by
$\,E_{\tilde j,\tilde j+1}\,$ and $\,w \fn w^{-1} \cap \fn_- \,$ by $\,E_{\tilde j +1,\tilde j}\,$. Since $\,[
E_{\tilde j,\tilde j+1},E_{\tilde j+1,\tilde j}] \in \fa\,$, hypothesis i) of the proposition is satisfied. In view
of (\ref{mvf2}), to verify ii), we must show that $\,(\Ad n_1 - 1)E_{\tilde j,\tilde j+1}\,$ annihilates
$\,s_{j,m,k,\tilde j}\,$,\,\ for every $\,n_1 \in N_{\tilde j}(\R)\,$; note that $\,N_{\tilde j}(\R)\,$ plays the
role of $\,wN(\R) w^{-1}\cap N(\R)\,$ in the present context. But
\begin{equation}
\label{schw23}
\big( \Ad N_{\tilde j}(\R) - 1\big) E_{\tilde j,\tilde j+1}\ \ \subset \ \ \R\,E_{\tilde j-1,\tilde j+1}\ \oplus \ \R\,E_{\tilde j,\tilde j +2}\ \oplus \ \fn''\,;
\end{equation}
here $\,\fn''\,$ denotes the Lie algebra of $\,N''(\R)\,$, the second derived subgroup of $\,N(\R)\,$. Since
$\,\tilde j \neq j, j+1\,$, $\,\big( \Ad N_{\tilde j}(\R) - 1\big) E_{\tilde j,\tilde j+1}\,$ annihilates
$\,\chi^R_{j,m,k}\,$,\,\ and hence, by (\ref{schw22}), $\,s_{j,m,k,\tilde j}\,$ -- that is the hypothesis ii). Thus
$\,\ell(E_{\tilde j,\tilde j})\,$ lowers the normal degree of $\,s_{j,m,k,\tilde j}\,$, contradicting
(\ref{schw22}) unless $\,s_{j,m,k,\tilde j}=0\,$.

We now turn to the case $\,j=0\,$, $\,\tilde j = 1\,$. At the beginning of this proof, we pointed out that the
action of the one parameter group $\,h_1(t)\,$ -- which is generated by $\,E_{1,2}\,$ -- on $\,\,R_{0,1,k}\tau\,$
drops to an action of $\,\R/\Z\,$, and the average of $\,\,R_{0,1,k}\tau\,$ over $\,\R/\Z\,$ vanishes
(\ref{cuspi0}). In view of \cite[proposition 7.20]{inforder}, these properties are inherited by $\,s_{0,1,k,1}\,$:
\begin{equation}
\label{schw24}
\ell(h_1(1))\,s_{0,1,k,1} \ = \ s_{0,1,k,1}\,,\ \ \ \ \text{and}\ \ \ \
\int_{\R/\Z}\ell(h_1(t))\,s_{0,1,k,1}\,dt \ = \ 0\,.
\end{equation}
We now apply proposition \ref{hypertanprop} to the Fourier coefficients of this $(\,\R/\Z)$-action. Since
$\,h_1(t)\,$ preserves $\,C_1\,$, normalizes $\,N_1(\R)\,$, and acts on $\,E_{1,2}\,$ via a character, not only
does $\,s_{0,1,k,1}\,$ satisfy the hypothesis ii) of proposition \ref{hypertanprop} with $\,Z=E_{1,2}\,$, but its
Fourier coefficients also do. We can then argue exactly as before, and conclude that all the Fourier coefficients
vanish, except possibly the one corresponding to the trivial character. But (\ref{schw24}) asserts the vanishing of
the ``constant" Fourier coefficient, so $\,s_{0,1,k,1}=0\,$ in the case $\,j=0\,$, $\,\tilde j = 1\,$, as well.
\end{proof}

\begin{lem}
\label{inforderlem4} For $\,0 \leq j \leq n-2\,$ and all choices of $\,m\,$ as well as $\,k$, the
$\,R_{j,m,k}\tau\,$ vanish to infinite order along $\,C_{j+1}\,$. Moreover, this is the case uniformly in $\,m\,$
and $\,k\,$, in the sense of \cite{inforder}.
\end{lem}

\begin{proof}
We argue by induction on $\,j\,$. Lemma \ref{inforderlem3} already contains this assertion for $\,j = 0\,$. Let us
suppose then that $\,j \geq 1\,$, and that the assertion is correct at the previous step. By induction and lemma
\ref{inforderlem1},
\begin{equation}
\label{ind1} R_{j-1,m,k}\tau\circ \Phi_j\ \ \ \text{vanishes to infinite order at $\,\infty\,$} \,,
\end{equation}
uniformly in $\,m\,$ and $\,k\,$. Both for $\,j=1\,$ and $\,j>1\,$, lemma \ref{RSlem3} relates the behavior of the
$\,R_{j-1,m,k}\tau\,$ at $\,C_j \cong \{x_j=\infty\}\,$ to that of the $\,\ell(h_j(k_{j+1}/m))S_{j,m,k}\tau\,$ at
$\,\{x_j=0\}\,$, though with different choices of $\,m\,$ and $\,k\,$ on the two sides; these choices can be
bounded by polynomials in $\,m\,$ and $\,\|k\|\,$,\,\ so uniformity in $\,m\,$ and $\,k\,$ is not affected.
Composing the identities with $\,\Phi_j\,$, we conclude that
\begin{equation}
\label{ind2} \ell\big(h_j({\textstyle\frac{k_{j+1}}{m}})\big)S_{j,m,k}\tau\circ \Phi_j\ \ \ \text{vanishes to infinite
order at the origin},
\end{equation}
uniformly in $\,m\,$ and $\,k\,$. Hence by (\ref{rhosigma}),
\begin{equation}
\label{ind3} \sigma_{j,m,k}\ \ \ \text{vanishes to infinite order at the origin},
\end{equation}
again uniformly in $\,m\,$ and $\,k\,$. Combining this with proposition \ref{sigmarhoftprop}, with (\ref{inf2}),
and with the relationship (\ref{rhosigma}) between the $\,\rho_{j,m,k}\,$ and the $\,R_{j,m,k}\tau\circ
\Phi_{j+1}\,$, we find that every $\,R_{j,m,k}\tau\circ \Phi_{j+1}\big|_\R\,$ has a canonical extension across
$\,\infty\,$. That is also true in the uniform sense, since (\ref{inf2}) preserves uniformity; see \cite[lemma
7.15]{inforder}. Hence, by \lemref{inforderlem1},
\begin{equation}
\label{ind4} R_{j,m,k}\tau\ \ \ \text{has a canonical extension across $\,C_{j+1}\,$} \,,
\end{equation}
uniformly in $\,m\,$ and $\,k\,$. It remains to be shown that all the $\,R_{j,m,k}\tau\,$ vanish to infinite order
along $\,C_{j+1}\,$. We argue by contradiction, and suppose
\begin{equation}
\label{ind5}R_{j,m,k}\tau \ \ \text{does not vanish to infinite order along $\,C_{j+1}\,$} \,,
\end{equation}
for at least one choice of $\,m\,$ and $\,k\,$. We must derive a contradiction from (\ref{ind4}) and (\ref{ind5}).

Recall the definition of the locally closed submanifolds $\,C^i_j\,$ in \lemref{inforderlem}. By downward induction
on $\,i\,$, for $\,0\le i \le j\,$, we shall show that
\begin{equation}
\begin{aligned}
\label{ind6}
&\text{the}\, \ R_{i,m,k}\tau\, \ \text{have canonical extensions across $\,C^{i+1}_{j+1}$, uniformly in $\,m,\,k\,$},
\\
&\text{but for some $\,m,\,k\,,$}\, \ R_{i,m,k}\tau\, \ \text{does not vanish to infinite order on $\,C^{i+1}_{j+1}\,$}.
\end{aligned}
\end{equation}
For $\,i=j\,$, this is the case by (\rangeref{ind4}{ind5}). Suppose than that $\,0 \le i \le j-1\,$ and that
(\ref{ind6}) is satisfied for all larger values of $\,i\,$. The actions of the one parameter groups
$\,h_{i+1}(t)\,$, $\,h_{i+2}(t)\,$ on $\,X\,$ preserve $\,C^{i+2}_{j+1}\,$, and drop to actions of the compact
group $\,\R/\Z\,$ on $\,N(\Z)$-invariant objects, such as $\,P_{i+1,m,k}\tau\,$. According to
(\rangeref{Pjkmtau2}{Pjkmtau3}), the $\,R_{i+1,m,k}\tau\,$ and $\,S_{i+1,m,k}\tau\,$ are the Fourier coefficients
of $\,P_{i+1,m,k}\tau\,$ with respect to the two actions. Thus, by \cite[lemma 7.2]{inforder}, the fact that the
$\,R_{i+1,m,k}\tau\,$ have canonical extensions, uniformly in $\,m\,$ and $\,k\,$, implies the corresponding
assertion about the $\,P_{i+1,m,k}\tau\,$, and then \cite[Proposition 7.20]{inforder} allows us to draw the
analogous conclusion about the $\,S_{i+1,m,k}\tau\,$. Differences between these distributions and their canonical
extensions inherit the invariance properties of their ``parents". Since the various components of a Fourier
expansion are necessarily linearly independent, if some $\,R_{i+1,m,k}\tau\,$ fails to vanish to infinite order,
that must also be the case for some $\,S_{i+1,m,k}\tau\,$. We conclude:
\begin{equation}
\begin{aligned}
\label{ind7}
&\!\!\text{the} \ S_{i+1,m,k}\tau\, \ \text{have canonical extensions across $C^{i+2}_{j+1}$, uniformly in $\,m,\,k\,$},
\\
&\!\!\text{but for some $\,m,\,k\,,$}\, \ S_{i+1,m,k}\tau\, \ \text{does not vanish to infinite order on $\,C^{i+2}_{j+1}\,$}.
\end{aligned}
\end{equation}
Translation by elements of $\,N(\R)\,$ -- or, for that matter, by diagonal matrices -- does not affect vanishing to
infinite order along Schubert cells. Thus (\ref{ind7}) remains correct for the renormalized quantities
$\,\ell(h_{i+1}(k_{i+2}/m))S_{i+1,m,k}\tau\,$. In case $\,i > 0\,$, lemma \ref{RSlem3} relates these renorma\-lized
quantities to the $\,\ell(h_{i+1}(-k_{i}/m))R_{i,m,k}\tau\,$,\,\ though with different choices of $\,m\,$ and
$\,k\,$,\,\ via translation by
\begin{equation}
\label{ind8} \Phi_{i+1} \ttwo{0}{-m_{i+2}/m}{m/m_{i+2}}{0}\ \ = \ \ \Phi_{i+1} \ttwo{0}{1}{-1}{0}\Phi_{i+1}
\ttwo{-m/m_{i+2}}{0}{0}{-m_{i+2}/m} \,.
\end{equation}
The preceding statement remains correct even for $\,i=0\,$, provided we replace
$\,\ell(h_{i+1}(-k_{i}/m))R_{i,m,k}\tau\,$ by $\,\ell(h_{1}(-\bar k_2m_2/m))R_{0,m,k}\tau\,$; to see this, note
that
\begin{equation}
\label{ind9} {\ttwo{0}{m_{2}/m}{-m/m_{2}}{\bar k_2}}^{-1}\ \ = \ \ {\ttwo{1}{-\bar
k_{2}m_2/m}{0}{1}}^{-1}\ttwo{0}{-m_{2}/m}{m/m_{2}}{0} \,.
\end{equation}
The first matrix on the right side of (\ref{ind8}) is $\,s_{i+1}\,$, as defined in (\ref{flag3.5}), and the second
lies in the diagonal subgroup. We had just remarked that translation by elements of $\,N(\R)\,$ or the diagonal
subgroup does not affect vanishing to infinite order along Schubert cells. The uniformity is not affected by the
different choices of $\,m\,$ and $\,k\,$ on the two sides, as was remarked earlier, nor by translation by a
diagonal matrix with entries that are bounded by multiples of $\,\|k\|\,$, and translation by
$\,h_{i+1}(k_{i+2}/m)\,$ or $\,h_{i+1}(-k_{i}/m)\,$. This establishes (\ref{ind6}), though so far only with
$\,s_{i+1}C^{i+2}_{j+1}\,$ in place of $\,C^{i+1}_{j+1}\,$.

If $\,i>0\,$, $\,s_{i+1}C^{i+2}_{j+1}\,$ is invariant under the one parameter group $\,h_i(t)\,$ -- not invariant
under $\,h_{i+1}(t)\,$,\,\ however -- so for $\,i>0\,$ we can argue as we did in the case of the
$\,R_{i+1,m,k}\tau\,$ and conclude that the $\,P_{i,m,k}\tau\,$ have canonical extensions across
$\,s_{i+1}C^{i+2}_{j+1}$, uniformly in $\,m\,$ and $\,k\,$ as always. But each $\,P_{i,m,k}\tau\,$ is
$\,N(\Z)$-invariant, and therefore has a canonical extension across all $\,N(\Z)$-translates of
$\,s_{i+1}C^{i+2}_{j+1}\,$.\,\ According to \lemref{inforderlem}, the $\,N(\R)$-translates constitute a Zariski
open cover of $\,C^{i+1}_{j+1}\,$. We claim that even the $\,N(\Z)$-translates cover $\,C^{i+1}_{j+1}\,$. If that
were not true, some non-empty Zariski closed subset of $\,s_{i+1}C^{i+2}_{j+1}\,$ would have to be
$\,N(\Z)$-invariant and, in view of the Zariski density of $\,N(\Z)\,$ in $\,N(\R)\,$, even $\,N(\R)$-invariant.
But then the $\,N(\R)$-translates of $\,s_{i+1}C^{i+2}_{j+1}\,$ could not cover $\,C^{i+1}_{j+1}\,$~--
contradiction! Having a canonical extension, whether in a uniform sense or not, is a local property. The
$\,P_{i,m,k}\tau\,$ therefore have canonical extensions across $\,C^{i+1}_{j+1}=\cup_{n\in N(\Z)}\,
n\,s_{i+1}C^{i+2}_{j+1}\,$, uniformly in $\,m\,$ and $\,k\,$. We now can conclude (\ref{ind6}) by taking Fourier
components with respect to the action of $\,h_i(t)\,$. If $\,i=0\,$ we need to modify the argument slightly. The
$\,R_{0,1,k}\tau\,$ themselves are $\,N(\Z)$-invariant, so we can argue as before and conclude that the
$\,R_{0,1,k}\tau\,$ have canonical extensions across not only $\,s_{1}C^{2}_{j+1}\,$, but all of $\,C^1_{j+1}\,$,
uniformly in $\,k\,$. That completes the verification of (\ref{ind6}) for all $\,i\,$, $\,1\leq i \leq j\,$.

We now apply (\ref{ind6}) with $\,i=0\,$:\,\ there exists some $\,k\,$ such that $\,R_{0,1,k}\tau\,$ has a
canonical extension across $\,C^1_{j+1}\,$ but does not vanish there to infinite order. The same must then be true
for at least one of the Fourier coefficients of $\,R_{0,1,k}\tau\,$ with respect to the action of $\,\R/\Z\,$ via
$\,h_1(t)\,$. In view of (\ref{cuspi0}) the non-zero Fourier coefficients are Whittaker distributions, i.e.,
$\,N(\Z)$-equivariant extensions of characters
\begin{equation}
\label{ind12}
N(\R) \, \ni \, n \ \ \mapsto \ e(k\,x)\ \ \text{if $\,n=n_{x,y}n''\,$ with $\,n''\in N''(\R)\,$} \ \ \ (\,k \in \Z^{n-1}_{\neq 0}\,)
\end{equation}
to distribution vectors in $\,C^{-\infty}(X,\mathcal L_{\l,\d})\,$.\,\ It is known, of course, that these extend
uniquely in an $\,N(\Z)$-equivariant manner \cite{chm}, as follows readily also from proposition
\ref{hypertanprop}. We have arrived at the required contradiction to (\rangeref{ind4}{ind5}).
\end{proof}

Together, \lemref{inforderlem1} and \lemref{inforderlem4} imply that all the $\,R_{j,m,k}\tau \circ \Phi_{j+1}\,$
vanish to infinite order at $\,\infty\,$. Proposition \ref{fliprelationsprop} follows, as was pointed out earlier.

\section{Classical proof of the formula}\label{sec:firstproof}

We begin by recalling some analytic ingredients from our paper \cite{inforder} that were used in \cite{voronoi}, in
particular some results from \cite[\S 6]{inforder} that apply directly to our context as well.  Let
$\Sch_{\text{\rm{sis}}}(\R)$ denote the space of all finite linear combinations of the products $(\sgn x)^\eta
|x|^{\a}(\log|x|)^j \phi(x)$, where $\eta\in \Z/2\Z$, $\a\in \C$, $j\in \Z_{\ge 0}$, and $\phi$ is an element of
the Schwartz space $\Sch(\R)$. By hypothesis, the test function $f$ in \thmref{mainthm} is an element of this
space, as its transform $F$ is also asserted to be by (\ref{MfMF1}).

Propositions~\ref{sigmarhoftprop} and \ref{fliprelationsprop} display relations involving Fourier transforms and
$x\mapsto x\i$ amongst some of the Fourier components of the automorphic distribution. These relations, chained
together, ultimately lead to a distributional identity which is equivalent to the statement of the summation
formula in \thmref{mainthm}.  For this it is convenient to use the operators
\begin{equation}\label{Tad}
    T_{\a,\eta} \ \ = \ \  {\mathcal F}\(x\mapsto \phi(x\i)\sgn(x)^\eta |x|^{-\a-1}\), \ \  \a\in\C \text{~and~} \eta\in\Z/2\Z \,,
\end{equation}
 which are the subject of \cite[\S6]{inforder}. This Fourier integral converges uniformly for $\Re{\a}$ sufficiently large when $\phi\in\Sch(\R)$, and
 extends to an operator which maps $\Sch_{\text{\rm{sis}}}(\R)$ to itself by \cite[Theorem 6.6]{inforder}.  Its adjoint operator is 
  given by
\begin{equation}\label{Tadjoint}
    T_{\a,\eta}^*\sigma(x) \ \ = \ \ \sgn(x)^\eta\,|x|^{\a-1}\,\widehat{\sigma}(\smallf{1}{x}) \,,
\end{equation}
which is shown there to extend to and preserve the space of tempered distributions  vanishing
to infinite order at the origin. We also require the slightly more general operators
\begin{equation}\label{T1star}
\aligned & {\mathcal T}_{j,a,b} \ \  = \ \ {\mathcal F}\(x\mapsto f(\smallf{b}{ax})
\sgn(bx)^{\d_j+\d_{j+1}}|a|^{-1}|b|^{\l_j-\l_{j+1}}|x|^{-\l_j+\l_{j+1}-1} \)\,,
\\
  & \qquad {\mathcal T}^*_{j,a,b}\sigma(x) \ \  = \ \ \sgn(a x)^{\d_j+\d_{j+1}} \, | a x|^{\l_j-\l_{j+1}-1}\, \widehat{\sigma}(\smallf{b}{ ax}) \\
  &  \qquad \qquad  \qquad \ \ \, = \ \ \mu_j(ax)\,\widehat{\sigma}(\smallf{b}{ax})
\endaligned
\end{equation}
for $j=1,\ldots,n-1$, where $\mu_j$ is defined in (\ref{mudef}) and both $a$ and $b$ are nonzero.  Their
analytic properties follow from those of $T_{\a,\eta}$ and its adjoint by simple rescalings.  In
particular, ${\mathcal T}_{j,a,b}$ maps $\Sch_{\text{\rm{sis}}}(\R)$ to itself, and ${\mathcal T}^*_{j,a,b}$
preserves  the space of tempered distributions which vanish to infinite order at the origin. Using the relation
\begin{equation}\label{mellandfour}
    M_\eta\widehat{\phi}(s) \ \ = \ \ (-1)^\eta \, G_\eta(s)\,M_\eta \phi(1-s)
\end{equation}
for an arbitrary element $\phi\in\Sch(\R)$ (\cite[(4.58)]{inforder}), the proof of \cite[Lemma 6.19]{inforder}
 can be modified easily  to show that
\begin{equation}\label{mellwithtjab1}
\aligned
&M_{\eta}({\mathcal T}_{j,a,b}\phi)(s) \ \ =  \ \ (-1)^\eta\, \sgn(a)^{\eta+\d_j+\d_{j+1}} \, |a|^{s+\l_j-\l_{j+1}-1} \  \times \\
& \qquad  \qquad  \qquad    \ \ \ \ \  \times \, \sgn(b)^{\eta}\, |b|^{-s} \, G_\eta(s) \,
(M_{\eta+\d_j+\d_{j+1}}\phi)(s+\l_j-\l_{j+1})\,,
\endaligned
\end{equation}
and consequently
\begin{equation}\label{mellwithtjab2}
\aligned
    & M_\eta({\mathcal T}_{1,a_1,b_1} {\mathcal T}_{2,a_2,b_2}\cdots
    {\mathcal T}_{n-1,a_{n-1},b_{n-1}}\phi)(s) \ \ =
    \ \ (-1)^{(n-1)\eta+(n-1)\d_1 +\d_n} \, \times \\
& \qquad   \  \times  \, \(\prod_{j=1}^{n-1} \sgn(a_j)^{\eta+\d_1+\d_{j+1}}|a_j|^{s+\l_1-\l_{j+1}-1}
G_{\eta+\d_1+\d_j}(s+\l_1-\l_j)\) \times \\
& \qquad \  \times \, \(\prod_{j=1}^{n-1} \sgn(b_j)^{\eta+\d_1+\d_j}|b|^{-s-\l_1+\l_j}\) \,
M_{\eta+\d_1+\d_n}\phi(s+\l_1-\l_n)\,.
\endaligned
\end{equation}

To complement the distributions $\sigma_{j,m,k}$ and $\rho_{j,m,k}$ from the previous section, we now introduce
some auxiliary distributions related to them :
\begin{equation}\label{taur}
   \tau_{R}(t) \ \ = \ \  \sum_{r\,\neq\, 0}\,c_{c_{n-2},\cdots,c_1,r}\,e(r(t-\smallf{a}{q}))
\end{equation}
depends   on the parameters $c_1,\cdots,c_{n-2}$ and $\f aq$, which are fixed in the statement of
\thmref{mainthm}; however, the analogous distributions
\begin{equation}\label{Deltal}
    \D_{L;k_2,\ldots,k_{n-1},\theta}  (t)   \ \ = \ \ \ \sum_{r\,\neq\, 0}\, c_{r,k_2,\ldots,k_{n-1}} \,e(r\,\theta)\,\d_{r}(t)
\end{equation}
and
\begin{equation}\label{taul}
    \tau_{L;k_2,\ldots,k_{n-1},\theta}(t) \ \ = \ \     \widehat{\D}_{L;k_2,\ldots,k_{n-1},\theta}(t) \ \ = \ \ \sum_{r\,\neq\, 0} \,c_{r,k_2,\ldots,k_{n-1}} \,e(r(\theta-t))
\end{equation}
do depend on different but similar parameters, which are therefore indicated in the notation.  The first formula in
proposition \ref{fliprelationsprop} can be restated in terms of this notation as
\begin{equation}\label{s1fliprestated}
    \sigma_{1,m,k} \ \ = \ \      {\mathcal T}^*_{1,\smallf{m}{m_2},\smallf{m_2}{m}} \ \D_{L;m_2,k_3,\ldots,k_{n-1},\f{\overline{k_2}}{m/m_2}}\,.
\end{equation}
Likewise, the second formula in proposition \ref{fliprelationsprop} relates $\tau_R$ to $\sigma_{n-2}$ by
\begin{equation}\label{rhonfliprestated}
    \tau_R \ \ = \ \ \sum_{\ell \imod {c_1q}}e(\smallf{\ell \, \bar{a}}{q}) \ {\mathcal T}^*_{n-1,q,c_1}\, \sigma_{n-2,qc_1,
    (c_{n-2},\ldots,c_2,0,\ell)}\,,
\end{equation}
as can be seen by relating it to $\rho_{n-2,c_1q,(c_{n-2},\ldots,c_2,c_1\bar{a},0)}$ and applying
proposition~\ref{sigmarhoftprop}.  In this last formula as well as elsewhere in this section, we set the indices
$k_j$ of $\sigma_{j,m,k}$ and $k_{j+1}$ of $\rho_{j,m,k}$ to zero, as we may. Finally, the third formula in
proposition \ref{fliprelationsprop} can also be restated in a way similar to (\ref{s1fliprestated}):
\begin{equation}\label{rhosigfliprestated}
\aligned
  &   \sigma_{j,m,k} \ \ = \ \  \sum_{\ell\,\imod {\f{mk_{j-1}}{m_{j+1}}}} e(\smallf{\ell\, \overline{k_{j+1}}\,m_{j+1}}{m}) \  \times \\ &
   \qquad \qquad  \qquad \
  \times \ {\mathcal T}^*_{j,\f{m}{m_{j+1}},k_{j-1}} \, \sigma_{j-1,\f{mk_{j-1}}{m_{j+1}},(k_1,\ldots,k_{j-2},0,\ell,m_{j+1},
    k_{j+2},\ldots,k_{n-1})}\,.
\endaligned
\end{equation}
Here we use that the $\sigma_{j,m,k}$ all vanish to infinite order at the origin; see
proposition~\ref{fliprelationsprop}.  The distributions $\D_{L;\cdots}$ obey a stronger property:~ they vanish
identically on the interval $(-1,1)$.
 Chaining together formulas
(\ref{s1fliprestated}-\ref{rhosigfliprestated}) now allows us to calculate $\tau_R$ in terms of the action of the
${\mathcal T}^*_{j,a,b}$ on (\ref{Deltal}).  We next parametrize $\ell$ in (\ref{rhonfliprestated}) as $d_1\ell_1$,
where $d_1$ is ranges over divisors of $c_1q$ and $\ell_1$ ranges over $(\Z/\f{qc_1 }{d_1}\Z)^*$, so that
\begin{equation}\label{chain1}
    \tau_R \ \ = \  \  \sum_{d_1|qc_1}\,\sum_{\ell_1\in (\Z/\f{qc_1}{d_1}\Z)^*}e(\smallf{d_1\ell_1\bar{a}}{q}) \  {\mathcal T}^*_{n-1,q,c_1} \, \sigma_{n-2,qc_1,(c_{n-2},\ldots,c_2,0,d_1\ell_1)}.
\end{equation}
Now consider equation (\ref{rhosigfliprestated}) with $j=n-2$. With this parametrization of $\ell$, the quantity
$m_{n-1}$ -- the GCD of $qc_1$ and $d_1\ell_1$ -- is equal to $d_1$, and $\overline{k_{n-1}}$ can be taken to be
$\overline{\ell_1}$, the modular inverse of $\ell_1$ in $(\Z/\f{qc_1}{d_1}\Z)^*$.  Using (\ref{rhosigfliprestated})
for $2 \le h \le n-2$, with $d_{h-1}|\smallf{qc_1\cdots c_{h-1}}{d_1\cdots d_{h-2}}$ and
$\ell_{h-1}\in(\Z/\smallf{qc_1\cdots c_{h-1}}{d_1\cdots d_{h-1}}\Z)^*$, we obtain successive relations
\begin{equation}\label{chain2}
\aligned
    &\sigma_{n-h,\f{qc_1\cdots c_{h-1}}{d_1\cdots d_{h-2}},(c_{n-2},\ldots,c_h,0,d_{h-1}\ell_{h-1},d_{h-2},\ldots,d_1)}
    \ \ = \\
    & \ \ \ \
\sum_{d_h|\f{qc_1\cdots c_h }{d_1\cdots d_{h-1}}}\  \sum_{\ell_h \in (\Z/\f{qc_1\cdots c_h}{d_1\cdots d_h}\Z)^*}
e\(\f{~d_h\ell_h\overline{\ell_{h-1}}~}{\smallf{qc_1\cdots c_{h-1}}{d_1\cdots d_{h-1}}}\) \, \times
\\
& \qquad\ \  \  \times \, {\mathcal T}^*_{n-h,\f{qc_1\cdots c_{h-1}}{d_1\cdots d_{h-1}},c_h}\sigma_{n-h-1,
\f{qc_1\cdots c_h}{d_1\cdots d_{h-1}},(c_{n-2},\ldots,c_{h+1},0,d_h\ell_h,d_{h-1},\ldots,d_1)}\,.
\endaligned
\end{equation}
When $h=2$ the parameters on the left hand side match those on the right hand side of (\ref{chain1}).  Thus
$\tau_R$ equals both of the following expressions:
\begin{equation}\label{chain3}
\aligned
&   \sum_{r\neq 0}c_{c_{n-2},\cdots,c_1,r}\,e(-r\smallf{a}{q})\,e(rt) \ \ = \\
& \underbrace{\sum_{d_h|\f{qc_1\cdots c_h}{d_1\cdots d_{h-1}}}\ \sum_{\ell_h\in(\Z/\f{qc_1\cdots c_h}{d_1\cdots
d_h}\Z)^*}}_{\text{for all~} h \,\le\, n-2}
    e\(\f{d_1\ell_1\bar{a}}{q} + \sum_{h=2}^{n-2} \f{~d_h \ell_h \overline{\ell_{h-1}}~}{\smallf{qc_1\cdots c_{h-1}}{d_1\cdots d_{h-1}}}
     \) \times \\
     & \ \times \,
    {\mathcal T}^*_{n-1,q,c_1}
    {\mathcal T}^*_{n-2,\f{qc_1}{d_1},c_2}\cdots
    {\mathcal T}^*_{2,\f{qc_1\cdots c_{n-3}}{d_1\cdots d_{n-3}},
c_{n-2}}
     \sigma_{1,\f{qc_1\cdots c_{n-2}}{d_1\cdots d_{n-3}},(0,d_{n-2}\ell_{n-2},d_{n-3},\ldots,d_1)}\,,
    \endgathered
\end{equation}
in which both sums involve all $d_h$ and $\ell_h$ for $1\le h \le n-2$. By (\ref{s1fliprestated}) the $\sigma_1$
term in the last line of this formula is
\begin{equation}\label{chain4}
    {\mathcal T}^*_{1,\f{qc_1\cdots c_{n-2}}{d_1\cdots d_{n-2}},\f{d_1\cdots d_{n-2}}{qc_1\cdots c_{n-2}}}\
    \D_{L;d_{n-2},d_{n-3},\ldots,d_1,
    \f{\overline{\ell_{n-2}}}{~\f{qc_1\cdots c_{n-2}}{d_1\cdots d_{n-2}}~}}.
\end{equation}
Thus (\ref{chain3}) remains equal after the following modifications are performed:~a sum over $r\in \Z$ is
inserted; $\f{\overline{\ell_{n-2}}r}{~\f{qc_1\ldots c_{n-2}}{d_1\cdots d_{n-2}}~}$ is added to the argument of the
exponential; $ {\mathcal T}^*_{1,\f{qc_1\cdots c_{n-2}}{d_1\cdots d_{n-2}},\f{d_1\cdots d_{n-2}}{qc_1\cdots
c_{n-2}}}$ is added to the end of the chain of ${\mathcal T}^*$ operators; and the $\sigma_1$ term is replaced by
 $c_{r,d_{n-2},d_{n-3},\ldots,d_1}\d_r(t)$.

We will now show that equation (\ref{chain3}) is the distributional equivalent of the summation formula in
\thmref{mainthm}, by integrating both sides against the test function ${\mathcal N}\cdot g_1(t)$, where ${\mathcal
N}$ is the normalizing factor
\begin{equation}\label{Nfactor}
 {\mathcal N} \ \ = \ \    \prod_{j\,\le\,n-2}\sgn(c_{n-1-j})^{\d_1+\cdots \d_j}|c_{n-1-j}|^{\l_1+\cdots + \l_{j}}
\end{equation}
and $g_1={\mathcal F}(f(x)|x|^{-\l_n}\sgn(x)^{\d_n})$, in terms of the function $f$ in \thmref{mainthm}. In
particular,
    \begin{equation}\label{mellg1f}
    M_\d g_1(s) \ \ = \ \ (-1)^\d\,G_\d(s)\,(M_{\d+\d_n}f)(1-s-\l_n)\,,
\end{equation}
because of (\ref{mellandfour}). By our hypothesis that $f\in |x|^{\l_n}\sgn(x)^{\d_n}\Sch(\R)$, $g_1$ is an
arbitrary Schwartz function, and so may be integrated against the periodic -- and hence tempered -- distribution on
the left hand side of (\ref{chain3}). Taking into account the normalization (\ref{ctoa}), this gives precisely the
left hand side of the formula in \thmref{mainthm}.

Now recall the description of the right hand side of (\ref{chain3}) given after (\ref{chain4}). The variables
$\ell_h$ occur only in the argument of the exponential, and yield exactly the hyperkloosterman sum
$S(r,\bar{a};q,c,d)$.  The integration of the right hand side of (\ref{chain3})  equals  
\begin{equation}\label{rhschain}
  {\mathcal N}\,\cdot\sum_{\srel{d_h|\f{qc_1\cdots c_h}{d_1\cdots d_{h-1}}}{\text{for all}\,h\le n-2}}\sum_{r\,\neq\,0} \,S(r,\bar{a};\,q,c,d) \ c_{ r,d_{n-2},d_{n-3},\ldots,d_1}\,
    g_2(r)\,,
\end{equation}
 where
\begin{equation}\label{g2fromTs}
\gathered
    g_2 \  = \  {\mathcal T}_{1,\f{qc_1\cdots c_{n-2}}{d_1\cdots d_{n-2}},\f{d_1\cdots d_{n-2}}{qc_1\cdots c_{n-2}}}
   {\mathcal T}_{2,\f{qc_1\cdots c_{n-3}}{d_1\cdots d_{n-3}},
c_{n-2}} \cdots
    {\mathcal T}_{n-2,\f{qc_1}{d_1},c_2}
          {\mathcal T}_{n-1,q,c_1}\,g_1.
\endgathered
    \end{equation}
Introducing the quantity $c_{n-1}=1$ for convenience, one can use (\ref{mellwithtjab2})  to express the Mellin
transform of $g_2$ as
\begin{equation}\label{mellg2g1}
\aligned
  &  M_\d g_2(s) \ \ = \ \ (-1)^{(n-1)\d+(n-1)\d_1+\d_n}\,(M_{\d+\d_1+\d_n}g_1)(s+\l_1-\l_n)  \ \times
\\
& \times \ \sgn(\smallf{qc_1\cdots c_{n-2}}{d_1\cdots d_{n-2}})^{\d} \left|\smallf{qc_1\cdots c_{n-2}}{d_1\cdots
d_{n-2}}\right|^{s} \( \prod_{j=1}^{n-1}\sgn(c_j)^{\d+\d_1+\d_{n-j}} |c_j|^{-s-\l_1+\l_{n-j}}\ \times \right.
\\
&\left.\times \  \sgn(\smallf{qc_1\cdots c_{j-1}}{d_1\cdots d_{j-1}})^{\d+\d_1+\d_{n-j+1}}\left|\smallf{qc_1\cdots
c_{j-1}}{d_1\cdots d_{j-1}}\right|^{s+\l_1-\l_{n-j+1}-1}
    G_{\d+\d_1+\d_j}(s+\l_1-\l_j)\).\!\!\!\!\!\!\!\!\!\!\!\!\!\!\!\!\!\!\!\!\!\!\!\!\!\!\!\!\!\!
\endaligned
\end{equation}
Letting
  \begin{equation}\label{calNtilde}
    \widetilde{\mathcal N} \ \ =  \ \ \prod_{j=1}^{n-2} \, (\sgn d_{n-1-j})^{\d_1+\cdots+\d_{j+1}}\,|d_{n-1-j}|^{\l_1+\cdots+\l_{j+1}}\,,
\end{equation}
the equality of (\ref{rhschain}) with the right hand side of the formula in \thmref{mainthm} reduces to the
identity
\begin{equation}\label{reduce1}
    {\mathcal N}\cdot \,g_2( r) \ \ = \ \  \sgn(r)^{\d_1} \, |r|^{\l_1} \,
     \widetilde{\mathcal N} \,
    \left|\f{q}{rd_1\cdots d_{n-2}}\right| \, F\!\(\textstyle\f{r \, d_{n-2}^2\, d_{n-3}^3\cdots d_1^{n-1}}{q^n\, c_{n-2}\, c_{n-3}^2\cdots c_1^{n-2}}\),
\end{equation}
or the following equivalent relation between Mellin transforms:
\begin{equation}\label{reduce2}
\aligned
  & {\mathcal N}  \cdot  \f{|d_1\cdots d_{n-2}|}{  |q|}  \,  M_\d g_2(s) \ \ =  \ \ \widetilde{\mathcal N}\cdot \sgn(\smallf{ d_{n-2}^2\, d_{n-3}^3\cdots d_1^{n-1}}{q^n\, c_{n-2}\, c_{n-3}^2\cdots c_1^{n-2}})^{\d+\d_1} \ \times \\
  & \qquad\qquad\qquad \times \left|\smallf{ d_{n-2}^2\, d_{n-3}^3\cdots d_1^{n-1}}{q^n\, c_{n-2}\, c_{n-3}^2\cdots c_1^{n-2}} \right|^{1-s-\l_1}
    M_{\d+\d_1}F(s+\l_1-1) \,.  
\endaligned
\end{equation}
After substituting (\ref{mellg1f}) into (\ref{mellg2g1}), and then into the left hand side of the previous
equation, while substituting   (\ref{MfMF}) into its right hand side, both sides have identical occurrences of
$M_{\d+\d_1}f(1-s-\l_1)$ and the product $\prod_{j=1}^{n}G_{\d+\d_1+\d_j}(s+\l_1-\l_j)$. A short computation using
(\ref{sumtozero}) then verifies that the remaining terms --  powers of $(-1)$, $|q|$, $\sgn(q)$, $|c_j|$,
$\sgn(c_j)$, $|d_j|$, and $\sgn(d_j)$ -- on both sides agree.   That completes the proof of theorem
\ref{mainthm}. 

\section{Adelic proof of the formula}\label{sec:secondproof}

In this section, we give a second, self-contained derivation of the Voronoi formula using  adelic automorphic
distributions. Since it is a second proof, we will describe only the formal aspects of the calculation;
the rigorous justification can be handled using the techniques of the previous sections. 

To begin, we will describe the adelization of the classical automorphic distributions from \secref{sec:background},
considering a $GL(n,\Z)$-invariant automorphic distribution that comes from a cuspidal automorphic representation
of $GL(n)$ over $\A$, the adele group of $\Q$.  This process is formally identical to the usual adelization of
classical automorphic functions using strong approximation, though we shall present it via
 Fourier expansions
because of the  key role they play later.

Let us first review Whittaker functions for automorphic representations. We use $\psi=\prod \psi_p$ to denote the
standard additive character on $\Q\backslash \A$, whose restriction to $\R$ coincides with $e(\cdot)$.  It can be
used to form the standard character $\psi_N=\prod \psi_{N,p}$ of $N(\Q)\backslash N(\A)$, by composing $\psi$ with
the sum of the entries just above the diagonal. A famous formula of Piatetski-Shapiro and Shalika shows that the
smooth vectors can be represented as sums of left-translates of adelic Whittaker functions $W=\prod W_p$, each of
which transforms on the left under $N(\Q_p)$ by the character $\psi_{N,p}$. By convention, the Fourier coefficient
$a_k$ of such a vector is
 the renormalized
value of the finite part of the adelic Whittaker function on the matrix  $\D_k=\operatorname{diag}(k_1\cdots
k_{n-1}, k_2\cdots k_{n-1},\ldots,k_{n-1},1)\in GL(n,\Q)$:
\begin{equation}\label{adelicwhitfunctionfinite}
    W_f(\D_k) \ \ = \ \
  \prod_{p<\infty}W_p(\D_k)
    \ \ = \ \ \f{a_{k_1,\ldots,k_{n-1}}}{\prod_{j\,=\,1}^{n-1}|k_j|^{j(n-j)/2} }\,.
\end{equation}

We shall define adelic automorphic distributions by replacing $W=W_\infty W_f$ with a ``boundary Whittaker
distribution'' $B=B_\infty W_f$ according to the following procedure.
  The distribution $
B_\infty\in V_{\l,\d}^{-\infty}$, like the Whittaker function $W_\infty$ it shall replace, will also transform
 on the left under $N(\R)$ according to this character.  Up to scaling, it must
 therefore be equal to this character on $N(\R)$, and is completely described as
  such on the open Schubert cell.  We define $B_\infty$ to be its
  unique extension to
  an $N(\R)$-equivariant distribution in $V_{\l,\d}^{-\infty}$,
   which was proven to exist
  in \cite{chm}.   The motivation for this definition is as follows.  Consider the
relation (\ref{ctoa}) between the Fourier coefficients of an automorphic form associated to
 $\tau$, and the Fourier coefficients of (\ref{tauabelexp}).  The latter were
 just interpreted in terms of $W_f(\D_k)$.
The product in (\ref{ctoa}) is the reciprocal
 of the {\em unnormalized} inducing character  (i.e.~without including $\rho$)
  from (\ref{autod}) on the diagonal matrix $\D_k$; the presence of $\rho$ accounts
  for the product in (\ref{adelicwhitfunctionfinite}).  Hence $B(\D_k)=c_{k_1,\ldots,k_{n-1}}$.

  The adelic
   automorphic distribution $\tau_\A$, in analogy to
  the Piatetski-Shapiro-Shalika Fourier expansion
  of cusp forms in terms of Whittaker functions, is defined as
 the sum of $B\(\ttwo{\g}{}{}1 g\)$, where $\g$
 runs over all cosets of $N^{(n-1)}(\Q)\backslash GL(n~-~1,\Q)$,
 $N^{(n-1)}$ being the subgroup
 of unit upper triangular matrices in $GL(n-1)$.
 When $\tau$ is restricted to the factor $GL(n,\R)\hookrightarrow GL(n,\A)$,
 it corresponds to the automorphic distribution from \secref{sec:background} that
 embeds elements of $V_{-\l,\d}^\infty$, the smooth vectors in the dual principal series
 representation, to smooth vectors of the automorphic representation.
Our assumption of $GL(n,\Z)$ invariance forces the $p$-adic Whittaker functions $W_p$, for $p$ finite, to be right
invariant under the maximal compact subgroup $GL(n,\Z_p)$ of $GL(n,\Q_p)$.
 For congruence subgroups, one must alter the $p$-adic
 Whittaker functions $W_p$ for $p$ dividing the level
  (this must be done even classically, and corresponds to vector valued
  automorphic
  forms or distributions which transform under $GL(n,\Z)$ by a matrix action).
  We shall not pursue this here, except to note that this is a computational
  obstacle to deriving the Voronoi formula for arbitrary congruence subgroups
  that in principal can be solved with enough information about ramified
  Whittaker functions.  For notational convenience, we
   drop the subscript $\A$ from
  $\tau_\A$, since only this object and not $\tau$ itself will be used for the
  remainer of this section.

Our proof is based on two different formulas for the following period of the adelic automorphic distribution:
\begin{equation}\label{Vperiod}
    V(g) \ \ = \ \ \int_{N_1(\Q)\backslash N_1(\A)}\tau(ng)\,\overline{\psi_N(n)}\,dn\,,
\end{equation}
where $N_1$ is the unipotent radical of the standard $(2,1,1,\ldots,1)$ parabolic of $GL(n)$.  One easily sees
 that it has the Fourier expansion
\begin{equation}\label{Vperiodfourexp}
    V(g) \ \ = \ \ \sum_{r\,\in\,\Q^*}B\(\ttwo{r}{}{}{I_{n-1}} g\).
\end{equation}
Just like automorphic forms, the automorphic distribution $\tau$ has a contragredient dual automorphic distribution
 $\widetilde{\tau}$ defined through the map (\ref{Rjk1.2tau})
and the convention that
\begin{equation}\label{contratau}
    \widetilde{\tau}(g) \ \ = \ \ \tau(\tilde{g})\,.
\end{equation}
This convention also serves to define dual Whittaker functions and distributions   $\widetilde{W}_p$ and
$\widetilde{B}_\infty$
 which transform on the left according to the complex
conjugate character $\overline{\psi}$ of $N$, and $\widetilde{\tau}$  has a similar Fourier expansion in terms of
$\widetilde{B}=\widetilde{B}_\infty\widetilde{W}_p$.
  Since the finite Whittaker functions $W_p$ are assumed
to be right invariant under $GL(n,\Z_p)$, and $B_\infty$ is a distribution vector for a principal series
representation (\ref{autod}),
\begin{equation}\label{nminspher}
\tau, \ \widetilde{\tau}, \ V,\    B,\text{~and~~}\widetilde{B}~~\text{are all  right invariant
     under $N_{-}(\R)\times K_f$}\,,
\end{equation}
$K_f=\prod_{p<\infty}GL(n,\Z_p)$ denoting the standard maximal compact subgroup of $GL(n,\A_f)$. We also use the
customary notation $\widehat{\Z}=\prod_{p<\infty}\Z_p$ and $\widehat{\Z}^*=\prod_{p<\infty}\Z_p^*$.

\begin{prop}\label{psformulaprop}
Let $\widetilde{V}(g)$ denote the analogous period of $\widetilde{\tau}$, but with $\psi$ instead of $\psi\i$ in
the integral (\ref{Vperiod}).  Then
\begin{equation*}
    \widetilde{V}(g) \ \ = \ \ \int_{\A^{n-2}}V\(
                                               \(\begin{smallmatrix}
                                               0 & 0 & 1 \\
                                               I_{n-2} &  0 & x \\
                                               0 &  1 & 0 \\
                                             \end{smallmatrix}\)
    \tilde{g}\)dx\,.
\end{equation*}
\end{prop}
\noindent This proposition is formally equivalent to a well-known result
 in the Rankin-Selberg theory that unfortunately does not seem to be
  in the literature.
   For that reason we have chosen to give a proof of it in the appendix.

Our alternate derivation of the formula in \thmref{mainthm} uses the formula in proposition~\ref{psformulaprop}
with
\begin{equation}\label{gandgtil}
\aligned
    g \
     \ = & \ \
     \(\begin{smallmatrix}
                                               1 & -b & 0 \\
                                               0 & 1 & 0 \\
                                               0 &0  & I_{n-2} \\
                                             \end{smallmatrix}\)
                              \(\begin{smallmatrix}
                                 c_1\cdots c_{n-2} & 0 & 0 & 0 & 0 &0 &0\\
                                 0& c_1\cdots c_{n-2}  & 0 & 0 & 0&0 &0\\
                                 0& 0 & c_2\cdots c_{n-2} & 0 & 0&0 &0\\
                                    0& 0 &0& c_3\cdots c_{n-2} & 0 & 0&0 \\
                                    0& 0 &0& 0 &\ddots & 0&0 \\
                                    0 & 0 & 0 & 0 & 0 & c_{n-2} & 0 \\
                                    0 & 0 & 0 & 0& 0& 0& 1
                                             \end{smallmatrix}\) \\
                                           \text{and~~~~}   \tilde{g} \
     \ = & \ \
     \(\begin{smallmatrix}
                                             I_{n-2} & 0 & 0 \\
                                               0 & 1 & b \\
                                               0 &0  & 1 \\
                                             \end{smallmatrix}\)
                              \(\begin{smallmatrix}
                                 c_1\cdots c_{n-2} & 0 & 0 & 0 & 0 &0 &0\\
                                 0& c_1\cdots c_{n-3}  & 0 & 0 & 0&0&0 \\
                                 0& 0 & c_1\cdots c_{n-4} & 0 & 0&0 &0\\
                                    0& 0 &0& \ddots & 0 & 0&0 \\
                                    0& 0 &0& 0 &c_1 & 0&0 \\
                                    0 & 0 & 0 & 0 & 0 &  1 & 0 \\
                                    0 & 0 & 0 & 0& 0& 0& 1
                                             \end{smallmatrix}\)z\i\,,
   \endaligned
\end{equation}
where $z=c_1\cdots c_{n-2}I$ is in the center of $GL(n,\A)$ (which, in our full level situation, we may tacitly
assume $\tau$ and $\widetilde{\tau}$ are invariant under). Here the $c_j$ are elements of $\Z_{\neq 0}$, regarded
as a subset of the diagonally-embedded $\Q^*\subset \A$, and $b=t+\f{\a}{q}\in \A$, where $t \in \R$ and $\a \in
\widehat{\Z}^*$ is equal to the integer $a$ modulo $q$. Note that $\a$ is not simply equal to the diagonal
embedding of the integer $a$, but is altered at its prime factors so that it is a unit at each place.  Both sides
of the expression in the proposition are distributions in $t\in \R$.  In our calculation we will restrict $t\neq
0$, though this is only for formal convenience as the distributions can indeed be shown to vanish to infinite order
at $t=0$.  These distributions are in fact identical those in (\ref{chain3}),    but are instead packaged in a way
connected to the Jacquet-Piatetski-Shapiro-Shalika derivation of the standard $L$-function on $GL(n)$.

Recalling (\ref{Vperiodfourexp}) and that the definition of  $\widetilde{V}$ involves $\psi\i$ instead of $\psi$,
the left hand side is equal to
\begin{equation}\label{lhside1}
    \sum_{r\,\in\,\Q^*}  e(rt-r\smallf aq)\,
    \widetilde{B}
      \(
    \begin{smallmatrix}
                                 rc_1\cdots c_{n-2} & 0 & 0 & 0 & 0 &0 &0 \\
                                0& c_1\cdots c_{n-2} & 0 & 0 & 0&0 &0\\
                                 0& 0 & c_2\cdots c_{n-2} & 0 & 0&0 &0 \\
                                    0& 0 &0& c_3\cdots c_{n-2} & 0 & 0&0 \\
                                    0& 0 &0& 0 &\ddots & 0&0 \\
                                    0 & 0 & 0 & 0 & 0 & c_{n-2} & 0 \\
                                    0 & 0 & 0 & 0& 0& 0& 1 \\
                                         \end{smallmatrix}
                                         \).
\end{equation}
The presence of the minus sign for the finite factor $\f aq$ comes from the contribution of the finite places of
$b$, and is due to the fact that $\psi$ is additively-invariant under the diagonal embedding of $\Q$ inside $\A$.
By definition, the value of $\widetilde{B}$ on diagonal matrices factors
 into the contragredient Fourier coefficient  (\ref{ctoa})
from \secref{sec:background} times the value of the unnormalized inducing character. The Fourier coefficient
vanishes unless $r\in \Z$, so the above equals
\begin{equation}\label{lhside2}
   \sum_{r\,\neq\, 0}e(rt-r\smallf{a}{q})\, a_{c_{n-2},\ldots,c_1,r}\, |r|^{\l_n}\sgn(r)^{\d_n}\prod_{j=1}^{n-2}|c_j|^{\sum_{i\ge n-j}\l_i}\sgn(c_j)^{\sum_{i\ge n-j}\d_i}.
\end{equation}
Integrating in $t\in\R$ against the test function
 \begin{equation}\label{testfunction}
  \prod_{j=1}^{n-2}|c_j|^{
   -\sum_{i\ge n-j}\l_i}\sgn(c_j)^{\sum_{i\ge n-j}\d_i}
 \, \cdot \, {\mathcal F}\(f(u)|u|^{-\l_n}\sgn(u)^{\d_n}\)(t)
 \end{equation}
 gives the left hand side of the formula in \thmref{mainthm}.
 This is the exact same integration that was
  performed in the previous section to obtain the left hand side of the Voronoi formula, though in adelic terminology.

Now we examine the right hand side of the formula in proposition~\ref{psformulaprop}, with the purpose of
integrating it against
 (\ref{testfunction}).
Unlike the calculation in the previous section, this does not directly involve the distributions $\sigma_{j,m,k}$
that played a prominent role there, though  the mechanics are fundamentally related.  The product of the two
matrices inside of $V(\cdot)$ and the first factor of $\tilde{g}$ in (\ref{gandgtil}) equals
\begin{equation}\label{rhside1}
   \(\begin{smallmatrix}
                                               0 & 0 & 1 \\
                                               I_{n-2} &  0 & x \\
                                               0 &  1 & 0 \\
                                             \end{smallmatrix}\)
                                              \(\begin{smallmatrix}
                                             I_{n-2} & 0 & 0 \\
                                               0 & 1 & b \\
                                               0 &0  & 1 \\
                                             \end{smallmatrix}\)
    \ \ = \ \  \(\begin{smallmatrix}
                                             0 & 0 & 1 \\
                                               I_{n-2} & 0 & x \\
                                               0 &1  & b \\
                                             \end{smallmatrix}\).
\end{equation}
Let $D_c$ equal to the matrix $\operatorname{diag}(c_1\cdots c_{n-2},c_1\cdots c_{n-3},\cdots,c_1)\in GL(n-2,\Q)$,
and $x$ equal the column vector $(x_{n-2},\ldots,x_1)$. Using the fact $b\i=t\i+q\a\i$ is in
$\R^*\times\widehat{\Z}$, (\ref{nminspher}) implies the integrand in proposition~\ref{psformulaprop} equals
\begin{equation}\label{rhside2}
\aligned
  &  V\(
    \(\begin{smallmatrix}
                                             0 & 0 & 1 \\
                                               I_{n-2} & 0 & x \\
                                               0 &1  & b \\
                                             \end{smallmatrix}\)
                                              \(\begin{smallmatrix}
                                              D_c & 0 & 0  \\
                                              0 & 1 & 0  \\
                                               0 & 0  & 1 \\
                                             \end{smallmatrix}\)
            \(\begin{smallmatrix}
                                              I_{n-2} & 0 & 0  \\
                                              0 & 1 & 0  \\
                                               0 &-b\i  & 1 \\
                                             \end{smallmatrix}\)
                      \)                   \ \ =  \ \
                      V\(
        \begin{smallmatrix}
                                              0 & -b\i & 1  \\
                                              D_c & -b\i x & x  \\
                                               0 & 0  & b \\
                                             \end{smallmatrix}\) &
                                             \\
          &  \qquad\qquad\qquad\qquad\qquad =  \ \
          \,
                      V\(
                             \( \begin{smallmatrix}
                                              1 & 0 & b\i  \\
                                              0 & I_{n-1}& b\i x  \\
                                               0 & 0  & 1 \\
                                             \end{smallmatrix}\)
                                               \(  \begin{smallmatrix}
                                              0 & -b\i & 0  \\
                                              D_c & -b\i x & 0  \\
                                               0 & 0  & b \\
                                             \end{smallmatrix}\)
                      \)           \\
    & \qquad\qquad\qquad\qquad\qquad =  \ \
                              \psi(b\i x_{1})\,
     V\(
                                               \begin{smallmatrix}
                                              0 & -b\i & 0  \\
                                              D_c & -b\i x & 0  \\
                                               0 & 0  & b \\
                                             \end{smallmatrix}
                      \).
         \endaligned
\end{equation}
We change variables $x\mapsto bD_c x=(t+\f{\a}{q})D_c x$ in that integral, which has the overall effect of changing
the measure by $|tq|^{n-2}$ because $D_c$ is rational (the valuation here always refers to the archimedean one
unless specified otherwise).  Thus the integral equals
\begin{equation}\label{rhside3}
  |tq|^{n-2}\int_{\A^{n-2}}\psi(c_1x_1)\,V\(\(
                                               \begin{smallmatrix}
                                              b\i & 0 & 0  \\
                                              0 & D_c & 0  \\
                                               0 & 0  & b \\
                                             \end{smallmatrix}
                      \)\(
                                               \begin{smallmatrix}
                                              0 & -1 & 0  \\
                                              I_{n-2} &   -x & 0  \\
                                               0 & 0  & 1 \\
                                             \end{smallmatrix}
                      \)\)\,dx\,.
\end{equation}
In effect, this last matrix step is the only tool needed to compute the value of $V$  and, as we now demonstrate,
also reduces the range of integration.
\begin{prop}\label{lotiszero}
Let $\cal X$ denote the subset of $(x_1,\ldots,x_{n-2}) \in \A^{n-2}$ with successively nondecreasing $p$-adic
valuations at each prime, that is  
\begin{equation*}
{\cal X} \ : \ \ |x_{n-2}|_p \ \ \ge \ \  |x_{n-3}|_p \ \  \ge \ \
 \cdots  \ \ \ge |x_2|_p  \ \ \ge \ \ |x_1|_p \ \  \ge \ \ 1\,.  
\end{equation*}
Then (\ref{rhside3}) is equal to
\begin{equation}\label{rhside6}
 |tq|^{n-2}\int_{{\cal X}}\, \sum_{r\,\in\,\Q^*}\,
    \psi\(c_1x_1+
    \sum_{j\,=\,2}^{n-2}\smallf{c_jx_j}{x_{j-1}}    \,
    +\smallf{r}{bx_{n-2}}\)\,
   B(\D)\,dx\,
     ,
\end{equation}
where $\D$ is the diagonal matrix 
\begin{equation*}
\D \ \ = \ \ \operatorname{diag}(-r\smallf{c_1\cdots c_{n-2}}{bx_{n-2}},-\smallf{c_1\cdots
c_{n-2}x_{n-2}}{x_{n-3}},-\smallf{c_1\cdots c_{n-3}x_{n-3}}{x_{n-4}},\ldots,-\smallf{c_1c_2x_2}{x_1},-c_1x_1,b)\,.
\end{equation*} 
\end{prop}

\begin{proof} By (\ref{nminspher}),
$V$ is invariant under right translation by the matrix $\(
                                               \begin{smallmatrix}
                                              I_{n-2} & v & 0  \\
                                               0 &   1 & 0  \\
                                               0 & 0  & 1 \\
                                             \end{smallmatrix}
                      \)$ for $v\in\widehat{\Z}^{n-2}$, which has the effect of translating $x$ to $x-v$.
Thus the integration over $\A^{n-2}$ reduces to the quotient $(\A/\widehat{\Z})^{n-2}$.  By ignoring the set where
all $(x_i)_\R$  vanish and appealing to strong approximation, we may chose coset representatives
$x=(x_{n-2},\ldots,x_1)$ such that each $x_i\i\in \R^*\times \widehat{\Z}$. This condition for $x_1$ along with the
matrix identity
\begin{equation}\label{smash}
    \( \begin{smallmatrix}
  &   &   &   &   & -1 \\
 1 &   &   &   &   & -x_{n-2} \\
  &  \ddots &   &   &   & \vdots \\
  &   &   &  1 &   & -x_2 \\
  &   &   &   &  1 & -x_1 \\
\end{smallmatrix} \) \(
 \begin{smallmatrix}
 1 &   &   &   &  \\
  &  \ddots &   &   &  \\
  &   &  \ddots &   &  \\
  &   &   &  1 &  \\
  &   &   &  \f{1}{x_1} & 1 \\
\end{smallmatrix}
\) \ \ = \ \
 \( \begin{smallmatrix}
 1 &   &   &   &   & -1 \\
  &  \ddots &   &   &   & -x_{n-2} \\
  &   &  &  1 &   & \vdots \\
  &   &   &  &  1 & -x_2 \\
  &   &   &   &  & -x_1 \\
\end{smallmatrix} \) \(
\begin{smallmatrix}
  &   &   &  -\f{1}{x_1} &  \\
 1 &   &   &  -\f{x_{n-2}}{x_1} &  \\
  &  \ddots &   &  \vdots &  \\
  &   &  1 &  -\f{x_2}{x_1} &  \\
  &   &   &   & 1 \\
\end{smallmatrix}\)
\end{equation}
allows us to replace the last matrix in (\ref{rhside3}) by these last two, regarded as embedded into the upper left
$(n-1) \,\times\,(n-1)$ block of $GL(n)$.  The first of these matrices may be passed through to the left, and
causes $V$ to transform by $\psi(\f{c_2x_2}{x_1})$. The last matrix now has the same form as the original one, but
of size one dimension smaller.  Right translation by an integer matrix as before, along with the structure of the
character, shows that this integral is unchanged if $x_2\mapsto x_2+zx_1$, for any $z\in \widehat{\Z}$.  Thus we
may again assume that $\f{x_1}{x_2}$ is a $p$-adic unit at each place. Continuing this way, we reduce the
integration to the domain $\mathcal X$.  To summarize, we have thus factored $\ttwo{}{-1}{I_{n-2}}{x}$ as
\begin{equation}\label{smashold}
\aligned & A \,U  = \ \  \\ & \ \ \(
                                               \begin{smallmatrix}
 -\f{1}{x_{n-2}} &  -\f{1}{x_{n-3}} &  -\f{1}{x_{n-4}} &  \cdots &  
   -\f{1}{x_1} & -1 \\
  &  -\f{x_{n-2}}{x_{n-3}} &  -\f{x_{n-2}}{x_{n-4}} &  \cdots &  
     -\f{x_{n-2}}{x_1} & -x_{n-2} \\
  &   &  -\f{x_{n-3}}{x_{n-4}} &  \cdots & 
    -\f{x_{n-3}}{x_1} & -x_{n-3} \\
  &   &   &  \ddots & 
    \vdots & \vdots \\
  &   &     &   &  -\f{x_2}{x_1} & -x_2 \\
  &   &      &   &   & -x_1 \\\end{smallmatrix}
                      \)\(
                                               \begin{smallmatrix}
 1 &   &   &   &   &       \\
 -\f{x_{n-3}}{x_{n-2}} &  1 &      &   &   &    \\
  &  -\f{x_{n-4}}{x_{n-3}} &    &   &   &     \\
     &   &  \ddots &   &   &  \\
  &     &   &    -\f{x_1}{x_2} &  1 &  \\
  &     &   &     &  -\f{1}{x_1} & 1 \\
                                             \end{smallmatrix}
                      \).
  \endgathered
\end{equation}
The matrix $U$ lies in $N_{-}(\R)\times GL(n-1,\widehat{\Z})$ in
 the region $\mathcal X$, and in this
situation the last matrix in (\ref{rhside3}) can be replaced by simply $\ttwo{A}{}{}1$, for $V$ is right-invariant
under such a matrix $U$. Therefore (\ref{rhside3}) can be written as
\begin{equation}\label{rhside4}
    |tq|^{n-2}\int_{{\cal X}}
    \psi(c_1x_1)\,V\(
\tthree{b\i}{0}{0}{0}{D_c}{0}{0}{0}{b} \ttwo{A}{}{}1
    \)\,dx\,,
\end{equation}
which equals
\begin{equation}\label{rhside5}
    |tq|^{n-2}\int_{{\cal X}}
    \psi(c_1x_1)\sum_{r\,\in\,\Q^*}B\(
    \tthree{r b\i}{0}{0}{0}{D_c}{0}{0}{0}{b}
\ttwo{A}{}{}1
    \)
\end{equation}
by (\ref{Vperiodfourexp}). To obtain the expression (\ref{rhside6}) we
  change the index of summation $r$ to $rc_1\cdots c_{n-2}$,
  and commute the upper triangular matrix $A$ across to the left so that it
  transforms out according the character $\psi$ of $N$ on the left of $B$.
\end{proof}

 The value of $B(\D)$ factors into an archimedean factor (which is the value
of the inducing character on this matrix), times local Whittaker factors. Since these local Whittaker functions are
unramified, their value is determined by the $p$-adic valuations on the simple roots, and vanish unless these are
all $\le 1$:
\begin{equation}\label{simprootcondition}
    |\smallf{q r  x_{n-3}}{x_{n-2}^2}|_p\,,\ |\smallf{c_{n-2}x_{n-2}x_{n-4}}{x_{n-3}^2}|_p\,,\
    |\smallf{c_{n-3}x_{n-3}x_{n-5}}{x_{n-4}^2}|_p\,, \ldots,  \,
    |\smallf{c_3 x_3x_1}{x_2^2}|_p\,, \
    |\smallf{c_2x_2}{x_1^2}|_p \,,
    \ |c_1x_1q|_p \, \le \,   1.
\end{equation}
Let us first consider the last of these inequalities, recalling that $|x_1|_p\le1$.  It constrains $d_1x_1$ to be a
$p$-adic unit at all places, for some divisor $d_1$ of $qc_1\in\Z$.  Since $d_1\mapsto \f{qc_1}{d_1}$ is a
bijection of such divisors, we may instead write $|x_1|_p=|\f{d_1}{qc_1}|_p$, for $d_1|qc_1$.  Likewise, the
constraints $|x_1|_p\le|x_2|_p\le|\f{x_1^2}{c_2}|_p$ mean that
$1\le|\f{x_2}{x_1}|_p\le|\f{x_1}{c_2}|_p=|\f{d_1}{qc_1c_2}|_p$. Thus, $|\f{x_2}{x_1}|_p$ equals
$|\f{d_2}{qc_1c_2/d_1}|_p$ for some $d_2|\f{qc_1c_2}{d_1}$, i.e. $|x_2|_p=|\f{d_1^2d_2}{q^2c_1^2c_2}|_p$.
Continuing, we see that the range of integration ${\mathcal X}$ in (\ref{rhside6}) can be broken up as the disjoint
union over
\begin{equation}\label{divisors}\aligned
d_1\,|&\,qc_1\\
d_2\,|&\,\smallf{qc_1c_2}{d_1}\\
d_3\,|&\,\smallf{qc_1c_2c_3}{d_1d_2}\\
\vdots \\
d_{n-2}\,|&\,\smallf{qc_1\cdots c_{n-2}}{d_1\cdots d_{n-3}}
\endaligned
\end{equation}
of
\begin{equation}\label{ranges}
\aligned
  &  \left\{\
    |x_1|_p \ = \ |\smallf{d_1}{qc_1}|_p \ , \ \ |\smallf{x_j}{x_{j-1}}|_p  \ = \
    |\smallf{d_1\cdots d_j}{qc_1\cdots c_j}|_p\ \ \ \text{for~~}j\,\ge\,2\,
    \right\} \ \ =
  \\
  & \qquad\qquad  = \ \ \left\{\
    |x_1|_p \ = \ |\smallf{d_1}{qc_1}|_p \ , \ \ |x_j|_p  \ = \
    |\smallf{d_1^jd_2^{j-1}\cdots d_j}{q^jc_1^jc_2^{j-1}\cdots c_j}|_p\ \ \ \text{for~~}j\,\ge\,2
   \, \right\}.
  \endaligned
\end{equation}
The divisors in (\ref{divisors}) are precisely the ones occurring on the right hand side of the formula in
\thmref{mainthm}, so we are reduced to showing that (\ref{rhside6}) -- when integrated over (\ref{ranges}) instead
of $\mathcal X$ -- corresponds to the rest of the right hand side of that formula.  The first constraint in
(\ref{simprootcondition}) governs the integrality of $r$, namely that on the piece (\ref{ranges}) one has that $r$
times $\f{q^n c_1^{n-1}c_2^{n-2}\cdots c_{n-2}^2}{d_1^{n-1}d_2^{n-2}\cdots d_{n-2}^2}$ is an integer.   When $r$ is
divided by that quantity, the sum over it in (\ref{rhside6}) becomes a sum over nonzero integers, corresponding to
the one on the right hand side in \thmref{mainthm}. We simultaneously change variables
\begin{equation}\label{xchangevar}
    x_1 \ \mapsto \ \smallf{d_1}{qc_1}\,x_1 \ , \ \ x_j \ \mapsto \ \smallf{d_1^jd_2^{j-1}\cdots d_j}{q^jc_1^jc_2^{j-1}\cdots c_j}\,x_j\ \ \text{for~~}j\,\ge\,2\,,
\end{equation}
which incurs no change of measure factor because these ratios are rational numbers.  This converts the domain
(\ref{ranges}) to $\{x_j\in \R\times \widehat{\Z}^*\}$, i.e.~adeles which are $p$-adic units at each finite place.
Then the relevant contribution of (\ref{rhside6}) is
\begin{equation}\label{rhside7}
\aligned &    |tq|^{n-2}\,\int_{(\R\times\widehat{\Z}^*)^{n-2}}
 \psi( \smallf{x_1d_1}{q} + \sum_{j\,=\,2}^{n-2} \smallf{d_jx_jx_{j-1}\i}{\smallf{qc_1\cdots c_{j-1}}{d_1\cdots d_{j-1}}}
+ \smallf{r}{b}\smallf{d_1 d_2 \cdots d_{n-2} }{q^2 c_1 c_2 \cdots c_{n-2} }  x_{n-2}\i ) \, \times
\\
&  B(\operatorname{diag}( -\smallf{r}{b}\smallf{d_1 d_2 \cdots d_{n-2} }{q^2 x_{n-2}}, -\smallf{d_1\cdots
d_{n-2}}{q}\smallf{x_{n-2}}{x_{n-3}}, -\smallf{d_1\cdots d_{n-3}}{q}\smallf{x_{n-3}}{x_{n-4}},
\ldots,-\smallf{d_1d_2x_2}{qx_1},-\smallf{d_1x_1}{q},b) ))dx.
\endaligned
\end{equation}
We next modify the signs of $r$ and the $x_j$  so that the signs of $x_1$,
$\f{x_2}{x_3},\ldots,\f{x_{n-2}}{x_{n-3}}$, and $\f{r}{x_{n-2}}$ are all flipped.  Thus we replace (\ref{rhside7})
with
\begin{equation}\label{rhside8}
\aligned &   |tq|^{n-2}\,\int_{(\R\times\widehat{\Z}^*)^{n-2}}
 \psi( -\smallf{x_1d_1}{q} - \sum_{j\,=\,2}^{n-2} \smallf{d_jx_jx_{j-1}\i}{\smallf{qc_1\cdots c_{j-1}}{d_1\cdots d_{j-1}}}
- \smallf{r}{b}\smallf{d_1 d_2 \cdots d_{n-2} }{q^2 c_1 c_2 \cdots c_{n-2} }  x_{n-2}\i ) \, \times
\\
& \ \ \times \, B(\operatorname{diag}(\smallf{r}{b}\smallf{d_1 d_2 \cdots d_{n-2} }{q^2 x_{n-2}}, \smallf{d_1\cdots
d_{n-2}}{q}\smallf{x_{n-2}}{x_{n-3}},\smallf{d_1\cdots d_{n-3}}{q}\smallf{x_{n-3}}{x_{n-4}},
\cdots,\smallf{d_1d_2x_2}{qx_1},\smallf{d_1x_1}{q},b) )).
\endaligned
\end{equation}
We now observe that both $\psi$ and $B$ in (\ref{rhside8}) split into an archimedean factor (which is a
distribution in $t\in\R$, the archimedean part of $b$), and a nonarchimedean factor (in which we similarly replace
$b$ by $\f{\a}{q}$). In the latter, the value of $B$ is precisely equal to
\begin{equation}\label{valueofB}|r|^{-(n-1)/2}\,\(\prod_{j\,=\,1}^{n-2}
    |d_j|^{-j(n-j)/2}\) \
    a_{r,d_{n-2},\ldots,d_{n-1}}\,,
\end{equation}
because $\a$ and the $x_i$ are all $p$-adic units.

An arbitrary element of $\A_f$ has the form $x_f+z$, where $x_f$ is the finite part of a rational number $x$ and
$z\in\widehat{\Z}$; the value of $\psi$ on such an adele is equal to $e(-x)$.
 Thus if an element of $\f{qc_1\cdots
c_j}{d_1\cdots d_j}\widehat{\Z}$ is added to $x_j$, the character $\psi$ is
 unchanged, and the $(\widehat{\Z}^*)^{n-2}$
part of the integral in (\ref{rhside8}) breaks up into a sum over cosets: it equals the  hyperkloosterman sum in
\thmref{mainthm}, divided by
 \begin{equation}\label{index}
 \left| \prod_{j\le n-2} \smallf{qc_1\cdots c_j}{d_1\cdots d_j} \right| \ \ = \ \ \left|
\smallf{q^{n-2}c_1^{n-2}c_2^{n-3}\cdots c_{n-2}}{d_1^{n-2}d_2^{n-3}\cdots d_{n-2}}\right| \,,
\end{equation}
which is the  adelic measure (i.e.~index) of this set that it is trivial on within $\widehat{\Z}^{n-2}$.

At this point the calculation has shifted from adeles to reals, and essentially repeats the final steps of the
calculation in the previous section. The archimedean part (\ref{rhside8}), including the factor of $|tq|^{n-2}$, is
a distribution in $t\in \R$ which is integrated against the function (\ref{testfunction}).  In order to symmetrize
the expression which follows, we now relabel $t$, the archimedean part of $b$, as $x_{n-1}$, and the variable $u$
from (\ref{testfunction}) as $x_{n}$. This contribution is equal to the product of powers of $|c_j|$ and
$\sgn(c_j)$ there times
\begin{equation}\label{archpart1}
\aligned &
 \(\prod_{j\,=\,1}^{n-2}
|d_j|^{\sum_{i\le n-j}(\f{n+1}{2}-i-\l_i)} \,\sgn(d_j)^{\sum_{i\le n-j}\d_i}\) \
\times \\
& \ \, \times \,
    \int_{\R^n}e(-\smallf{x_1d_1}{q} - \sum_{j\,=\,2}^{n-2}
     \smallf{d_jx_jx_{j-1}\i}{~\smallf{qc_1\cdots c_{j-1}}{d_1\cdots d_{j-1}}~}
- \smallf{r}{x_{n-1}}\smallf{d_1 d_2 \cdots d_{n-2} }{q^2 c_1 c_2 \cdots c_{n-2} }  x_{n-2}\i - x_{n-1}x_n)\ \times
\\
& \ \, \times \,
 f(x_{n})\,|x_n|^{-\l_n}\,\sgn(x_n)^{\d_n}\,
|r|^{\f{n-1}{2}-\l_1}\,\sgn(r)^{\d_1}\, |qx_{n-1}|^{\l_1-\l_n-1} \ \times
 \\
& \ \,  \times \, \sgn(qx_{n-1})^{\d_1+\d_n}\, \(\prod_{j\,=\,1}^{n-2}|x_j|^{\l_{n-j-1}-\l_{n-j}-1}\,
\sgn(x_j)^{\d_{n-j-1}+\d_{n-j}}\)dx_1\cdots dx_n\,.
\!\!\!\!\!\!\!\!\!\!\!\!\!\!\!\!\!\!\!\!\!\!\!\!\!\!\!\!\!\!\!\!\!\!\!\!\!\!\!\!
\!\!\!\!\!\!\!\!\!\!\!\!\!\!\!\!\!\!\!\!\!\!\!\!\!\!\!\!\!\!\!\!\!\!\!\!
\endaligned
\end{equation}
To symmetrize the argument of the exponential, we change $x_{n-1}\mapsto x_{n-1}\i$; this changes its occurrence
outside the exponential to $|x_{n-1}|^{-\l_1+\l_n-1}\sgn(x_{n-1})^{\d_1+\d_n}$. Each term in the exponential, aside
from the first, is now a ratio of successive variables $\f{x_j}{x_{j-1}}$.  Changing variables $x_j\mapsto
x_1\cdots x_j$ converts these ratios each to $x_j$, and the expression  (\ref{archpart1}) as a whole to
\begin{equation}\label{archpart2}
\aligned & |q|^{\l_1-\l_n-1}\sgn(q)^{\d_1+\d_n}\( \prod_{j\,=\,1}^{n-2} |d_j|^{\sum_{i\le n-j}\f{n+1}{2}-i-\l_i}
\sgn(d_j)^{\sum_{i\le n-j}\d_i}\)|r|^{\f{n-1}{2}-\l_1} \ \times
\!\!\!\!\!\!\!\!\!\!\!\!\!\!\!\!\!\!\!\!\!\!\!\!\!\!\!\!\!\!\!\!\!\!\!\!\!\!\!\!
\!\!\!\!\!\!\!\!\!\!\!\!\!\!\!\!\!\!\!\!\!\!\!\!\!\!\!\!\!\!\!\!\!\!\!\!\\
& \ \times \, \sgn(r)^{\d_1} \int_{\R^n}e(-\smallf{x_1d_1}{q} - \sum_{j\,=\,2}^{n-2} \smallf{d_jx_j}{\smallf{qc_1\cdots
c_{j-1}}{d_1\cdots d_{j-1}}} -
r\smallf{d_1 d_2 \cdots d_{n-2} }{q^2 c_1 c_2 \cdots c_{n-2} }  x_{n-1}  - x_n) \ \times\\
& \ \qquad\qquad \ \  \times \, f(x_1\cdots x_{n})\,
|x_n|^{-\l_n}\,\sgn(x_n)^{\d_n}\,|x_{n-1}|^{-\l_1}\,\sgn(x_{n-1})^{\d_1} \  \times
\\
& \ \qquad \qquad\qquad\qquad \qquad \ \ \times \,\( \prod_{j\,=\,1}^{n-2}|x_j|^{-\l_{n-j}}\,
\sgn(x_j)^{\d_{n-j}}\)\,dx_1\cdots dx_n\,.
\endaligned
\end{equation}
We now come to the final change of variables, which maps $x_j\mapsto \f{qc_1\cdots c_{j-1}}{d_1\cdots d_j}x_j$ for
$ j \le n-2$, and $x_{n-1}\mapsto \f{q^2c_1\cdots c_{n-2}}{rd_1\cdots d_{n-2}}x_{n-1}$.  The argument of $f$ is
divided by the product of these ratios, $\f{rd_1^{n-1}d_2^{n-2}\cdots d_{n-2}^2}{q^nc_1^{n-2}c_2^{n-3}\cdots
c_{n-2}}$, and so this integral is therefore  by (\ref{Ffromf})  equal to the same instance of
$F(\f{rd_1^{n-1}d_2^{n-2}\cdots d_{n-2}^2}{q^nc_1^{n-2}c_2^{n-3}\cdots c_{n-2}})$ occurring in \thmref{mainthm},
times the change of variables factor
\begin{equation}\label{theendisnear}
   \(\prod_{j\,=\,1}^{n-2} |\smallf{qc_1\cdots c_{j-1}}{d_1\cdots d_j}|^{1-\l_{n-j}}
   \sgn(\smallf{qc_1\cdots c_{j-1}}{d_1\cdots d_j})^{\d_{n-j}}\)
   |\smallf{q^2c_1\cdots c_{n-2}}{rd_1\cdots d_{n-2}}|^{1-\l_1}
   \sgn(\smallf{q^2c_1\cdots c_{n-2}}{rd_1\cdots d_{n-2}})^{\d_1}
   .
\end{equation}
To finish, we multiply the products of $|q|$, $\sgn(q)$, $|r|$, $\sgn(r)$, $|c_j|$, $\sgn(c_j)$, $|d_j|$, and
$\sgn(d_j)$ from (\ref{testfunction}), (\ref{valueofB}), (\ref{archpart1}), and (\ref{theendisnear}), and divide by
(\ref{index}). Using (\ref{sumtozero}), this indeed results in $\f{|q|}{|rd_1\cdots d_{n-2}|}$, verifying the right
hand side of the formula in \thmref{mainthm}.

\appendix

\section{Appendix: Proof of Proposition~\ref{psformulaprop}}\label{appendix}

In this appendix we prove  proposition~\ref{psformulaprop}, the main identity used in the adelic proof of our
$GL(n)$ Voronoi summation formula. We do this by establishing a more general statement, namely a crucial but
unpublished ingredient in the  work of Jacquet, Piatetski-Shapiro, and Shalika on the tensor product $L$-functions
on $GL(n_1)\times GL(n_2)$ for $|n_1-n_2|>1$.  Despite not appearing in their papers, it is both the mechanism by
which their functional equation and analytic properties are established, and the source for the local integrals
which they study in detail.

The more general statement concerns a smooth cusp form $\phi$ on the quotient $GL(n,F)\backslash GL(n,\A)$, where
$F$ is a number field and $\A=\A_F$ is its ring of adeles (our application here only requires $F=\Q$).  In
particular it  does not involve a level assumption.  By convention, we define the dual form $\tilde{\phi}$ by the
formula $\widetilde{\phi}(g)=\phi(\tilde{g})$ (see (\ref{contratau})).

\begin{prop}\label{appprop}
 Fix $1<m<n$ and let $N^\circ$ denote the unipotent radical of the $(m,1,1,\ldots,1)$ standard parabolic subgroup of $GL(n)$.  Define periods
 \begin{equation*}
    \aligned
\ \ \ \    V(g) \ \ & = \ \ \int_{N^\circ( F)\backslash N^\circ(\A)}\,\phi(ng)\,\overline{\psi_N(n)}\,dn \\
   \text{and} \ \ \widetilde{V}(g) \ \ & = \ \ \int_{N^\circ( F)\backslash N^\circ(\A)} \tilde{\phi}(ng)\,\psi_N(n)\,dn \\
   & = \ \ \int_{\widetilde{N^\circ}( F)\backslash \widetilde{N^\circ}(\A)} \phi(n\tilde{g})\,\overline{\psi_N(n)}\,dn \,,
    \endaligned
 \end{equation*}
where $\widetilde{N^\circ}$ is the unipotent radical of the $(1,1,\ldots,1,m)$ standard parabolic subgroup of
$GL(n)$. Then
\begin{equation*}
   \widetilde{V}(g) \ \ = \ \ \int_{M_{n-m,m-1}(\A)} V\(\tthree{I_{m-1}}{}{}{X}{I_{n-m}}{}{}{}1 w \,\tilde{g}\)dX\,,
\end{equation*}
in which the integral converges absolutely,
 $M_{k,\ell}$ denotes $k\times \ell$ matrices, and  $w$ denotes the permutation matrix $\ttwo{}{I_{m-1}}{I_{n-m+1}}{}$.
\end{prop}

Proposition~\ref{psformulaprop} corresponds to the case $m=2$, but for distributions instead of smooth forms.
However, the distributional version stated there is equivalent to the statement here (for full-level cusp forms),
because automorphic distributions are equivalently linear functionals which control  their embeddings into spaces
of smooth functions. The periods $V$ defined here have Fourier expansions as sums of Whittaker functions, left
translated by elements of $GL(m-1, F)$ embedded into the upper left corner of $GL(n,\A)$.  Together with the
relation above, this implies the unfolding of Jacquet-Piatetski-Shapiro-Shalika's integral representation, as
quoted in \cite[Theorem 2.1]{cogdell}, for example.

Before giving the proof, it is helpful to describe some relevant aspects of integration  over quotients of
nilpotent groups which did not arise in the published special cases \cite{cogps,jpss} of proposition~\ref{appprop}.
Suppose that $G=G_1\ltimes G_2$ is the semidirect product of locally compact Hausdorff groups such that  $G_2$ is
abelian, and that the continuous action $\rho$ of $G_1$ on $G_2$ appear in the group law
\begin{equation}\label{semidirectproduct}
    (g_1,g_2)\cdot (h_1,h_2) \ \ = \ \ (g_1h_1,\,g_2+\rho(g_1)h_2)
\end{equation}
preserves its Haar measure.
 Under this assumption, the left
Haar measure on $G$ is the product measure of the left Haar measure on $G_1$ with the Haar measure on $G_2$.  The
group $G$ is furthermore unimodular if $G_1$ is.

We further suppose each $G_i$ has a discrete subgroup $\G_i$, with the semidirect product $\G=\G_1\ltimes \G_2$
itself discrete in $G$.  In particular, for each $\g_1\in \G_1$, $\rho(\g_1)$ acts bijectively on $\G_2$. Recall
that a fundamental domain for a discrete subgroup $\D$ of a topological group $H$ is an open set $S\subset H$ whose
left $\D$-translates are disjoint yet dense in $H$.  For any $h\in H$, $Sh$ is also a fundamental domain for $\D$,
but $hS$ is instead  a fundamental domain for $h\D h\i$. If $S_2$ is a fundamental domain for $\G_2\subset G_2$,
$\rho(\g_1) S_2$ is also a fundamental domain for each $\g_1\in \G_1$, because $\G_1$ normalizes $\G_2$.  It
follows from this that if $S_1$ is a fundamental domain for $\G_1\subset G_1$, then $S=S_1 \times S_2$ is a
fundamental domain for $\G\subset G$.  More generally, if $f:S_1\rightarrow G_2$ is continuous, then
$\{(s_1,f(s_1)+s_2)\,|\,s_1\in S_1,\,s_2\in S_2\}$ is also a fundamental domain for $\G\backslash G$.

We now apply the above considerations to an iterated semidirect product.  Namely, suppose that $U$ is any unipotent
radical of a parabolic subgroup of $GL(n)$, having dimension $d$.  Its adelic points $U(\A)$ can be identified,
setwise, with $d$-tuples $(u_1,\ldots,u_d) \in \A^d$, and its Haar measure is the product of Haar measures
$du_1\cdots du_d$ from each copy.  A fundamental domain for $U( F)\backslash U(\A)$ is given by any product of
fundamental domains for each $u_i\in  F\backslash \A$, one for each copy.  We always normalize our Haar measures to
give volume 1 to $( F\backslash \A)^d$ under this identification.   Because $U$ is an iterated semidirect product
of abelian groups, the remark at the end of the previous paragraph implies the following by induction:

\begin{prop}\label{leftfunddomprop}
Let $U$ be the unipotent radical of a parabolic subgroup of $GL(n)$. Then left translation by elements of $U(\A)$
maps any fundamental domain for $U( F)\backslash U(\A)$ into another.
\end{prop}

\begin{proof}[Proof of proposition~\ref{appprop}]
Let $N'$ denote the subgroup $w\widetilde{N^\circ}w\i$, so that
\begin{equation}\label{appprop3}
\widetilde{V}(g) \ \ = \ \ \int_{N'( F)\backslash N'(\A)}\,\phi(nw\tilde{g}) \, \overline{\psi_N(n)}\,dn\,.
\end{equation}
This reduces the proposition to showing
\begin{multline}\label{appprop4}
    \int_{N'( F)\backslash N'(\A) }\phi(ng) \, \overline{\psi_N(n)}\,dn \ \ = \\ \ \ \int_{M_{n-m,m-1}(\A)} \int_{N^\circ( F)\backslash N^\circ(\A)}
    \phi\(n \tthree{I_{m-1}}{}{}{X}{I_{n-m}}{}{}{}1 g\)\overline{\psi_N(n)}\,dn\,dX
\end{multline}
for an arbitrary $g\in GL(n,\A)$.  Let $U$ denote the subgroup of unit upper triangular matrices in $GL(n-m+1)$.
The character $\psi_N$ for $GL(n-m+1)$ makes sense on $U(\A)$, and agrees with its
 $GL(n)$ variant when $U$ is embedded into the lower right corner; as no confusion will arise,
 we shall use the notation $\psi_N$ for either sized matrix.
Using the notation $M_{k,\ell}( F\backslash \A)$ as shorthand for a fundamental domain for the quotient
$M_{k,\ell}( F)\backslash M_{k,\ell}(\A)$, we define the following integrals for $0\le j \le n-m$:
\begin{equation}\label{appprop5}
\aligned &    I_j \ \ = \ \  \  \int_{\begin{smallmatrix}
y\,\in\,M_{m-1,j}( F\backslash \A) \\
X_1 \, \in \, M_{n-m-j,m-1}( F\backslash \A) \\
X_2\,\in\, M_{j,m-1}(\A) \\
u\,\in\,U( F)\backslash U(\A)
\end{smallmatrix}}
   \! \phi\(\!
\tthree{I_{m-1}}{}{y}{}{I_{n+1-m-j}}{}{}{}{I_j}\! \ttwo{I_{m-1}}{}{\begin{smallmatrix} X_1 \\ X_2
\\  0
\end{smallmatrix}}u g\)\,\times \\ & \ \ \qquad \qquad \qquad \qquad\qquad
\qquad\qquad \qquad \qquad \qquad \times \ \overline{\psi_N(u)}\,du\,dX_2\,dX_1\,dy\,.
\endaligned
\end{equation}
The $y$ integration is clearly well defined on the quotient, justifying the use of the fundamental domain for it.
We now argue that the $X_1$ and $u$
integrations are as well.  To simplify notation, let us temporarily write $Y=[0\ y]\in M_{m-1,n+1-m}$ and $X=\left[\begin{smallmatrix} X_1 \\
X_2 \\  0
\end{smallmatrix}\right]\in M_{n+1-m,m-1}$, so that the argument of
$\phi$ is $\ttwo{I_{m-1}}{Y}{}{I_{n+1-m}}\ttwo{I_{m-1}}{}{X}{u} g$. Suppose first that $u$ is replaced by $\g u$,
for $\g \in U( F)$. Since $\phi$ is left invariant under $\ttwo{I_{m-1}}{}{}{\g\i}$, the argument of $\phi$ can be
replaced by $\ttwo{I_{m-1}}{Y\g}{}{I_{n+1-m}}\ttwo{I_{m-1}}{}{\g\i X}{u}$. Since $\g$ is unipotent upper
triangular, a column of $Y\g$ is formed its counterpart in $Y$ by adding multiples of preceding columns  to it.
The reverse change of variables $y\mapsto y\g\i$ preserves the subgroup $M_{m-1,j}( F)$ and the measure $dy$.
Likewise, the variables  $X_1$ and $X_2$ can be changed to convert the expression back to $I_j$.  Thus the $u$
integration is well defined.

To show that the $X_1$ integration is well defined, suppose $Q\in M_{n-m+1,m-1}( F)$ has zero entries in its bottom
$j+1$ rows, so that its nonzero entries correspond to positions in $X_1$.  Adding $Q$ to $X$  likewise has the
effect of replacing the matrix $\ttwo{I_{m-1}}{Y}{}{I_{n+1-m}}$ in the  argument of $\phi$ by
\begin{equation}\label{approp21}
\aligned & \ttwo{I_{m-1}}{}{-Q}{I_{n-m+1}} \ttwo{I_{m-1}}{Y}{}{I_{n+1-m}}\ttwo{I_{m-1}}{}{Q}{I_{n-m+1}} \,  = \,
\ttwo{I_{m-1}}{}{-Q}{I_{n-m+1}}  \ttwo{I_{m-1}}{Y}{Q}{I_{n-m+1}}
\\ & \qquad = \, \ttwo{I_{m-1}}{Y}{}{I_{n-m+1}-QY} \, = \,
\ttwo{I_{m-1}}{Y}{}{I_{n-m+1}}\ttwo{I_{m-1}}{}{}{I_{n-m+1}-QY} ,
\endaligned
\end{equation}
since $YQ=0$.   With $u'=I_{n-m+1}-QY$, one has
$\ttwo{I_{m-1}}{}{}{u'}\ttwo{I_{m-1}}{}{X}{u}=\ttwo{I_{m-1}}{}{u'X}{u'u}$. As before, the change of variables
$X\mapsto u'X$  modifies rows of $X$ by adding multiplies of lower rows to them.  This change of variables can be
undone without changing the measure, and furthermore maps any fundamental domain for $X_1$ to another.  Another
application of  proposition~\ref{leftfunddomprop} shows that the   change of variables $u\mapsto (u')\i u$ also
maps any fundamental domain for $u\in U( F)\backslash U(\A)$ to another as well.  It also preserves the character
$\psi_N(u)$, because $u'$  -- an member of the derived subgroup $[U(\A),U(\A)]$ -- lies in the kernel of $\psi_N$.
Thus the integrands in each $I_j$ are well-defined on their regions of integration.

The group $N'$ is the semidirect product of the embedding of $U$ into the lower right corner of $GL(n)$,
 and the embedding of $M_{n-m,m-1}$ into $GL(n)$ given by $X_1\mapsto
\tthree{I_{m-1}}{}{}{X_1}{I_{n-m}}{}{}{}1$.  Hence the product of fundamental domains for $U( F)\backslash U(\A)$
and $M_{n-m,m-1}( F)\backslash M_{n-m,m-1}(\A)$ serves as a fundamental domain for $N'( F)\backslash N'(\A)$, and
the product of their Haar measures  is likewise the Haar measure on $N'$ which gives this quotient volume 1.
Therefore the integral $I_0$ reduces to the left hand side of (\ref{appprop4}). When $j=n-m$,
\begin{equation}\label{approp22}
\aligned &    I_{n-m} \ \ = \ \  \  \int_{\begin{smallmatrix}
y\,\in\,M_{m-1,n-m}( F\backslash \A) \\
X_2\,\in\, M_{n-m,m-1}(\A) \\
u\,\in\,U( F)\backslash U(\A)
\end{smallmatrix}}
  \phi\(
\tthree{I_{m-1}}{}{y}{}{1}{}{}{}{I_{n-m}}\! \ttwo{I_{m-1}}{}{\begin{smallmatrix}  X_2 \\  0
\end{smallmatrix}}u g\)\,\times \\ & \ \ \qquad \qquad \qquad \qquad \qquad\qquad \qquad \qquad
\qquad \qquad \qquad\times \ \overline{\psi_N(u)}\,du\,dX_2\,dy\,.
\endaligned
\end{equation}
Let us factor the last matrix in this expression as $\ttwo{I_{m-1}}{}{}u \ttwo{I_{m-1}}{}{\begin{smallmatrix}  X_2' \\
0
\end{smallmatrix}}{I_{n-m+1}}$, in which $u[\begin{smallmatrix}  X_2' \\  0
\end{smallmatrix}]=[\begin{smallmatrix}  X_2 \\  0
\end{smallmatrix}]$.  Both the change of variables $X_2\mapsto X_2'$, as well as its inverse,
involve  adding multiples of lower rows to higher ones, which does not alter the measure $dX_2$. Applying the same
considerations about product fundamental domains and measures for $N^\circ$ as observed for $N'$ at the beginning
of this paragraph,
 we see that (\ref{approp22})
is hence equal to the right hand side of (\ref{appprop4}).

To complete the proof  we will show that $I_j=I_{j+1}$, and afterwards argue the absolute convergence.
 Starting with
(\ref{appprop5}), we enlarge $y$ by adding a column to its left, denoted by $y_1$.  Let $Y=[0\ y]$ be as above, but
with this newly enlarged $y$. Let $Q$ now denote the $(n-m+1)\times(m-1)$ matrix which has all zeroes except for
its $n-m-j$-th row, which equals the row vector $q\in F^{m-1}$.  We have that $YQ=0$ and $QY$ is strictly upper
triangular, with its $(n-m-j,n-m-j+1)$-st entry equal to $qy_1$.  Using $\phi$'s left invariance under
$\ttwo{I_{m-1}}{}{Q}{I_{n-m+1}}$, we may rewrite $I_j$ in terms of its Fourier series expansion at $y_1=0$:
\begin{equation}\label{approp23}
 \aligned
&     \sum_{q\,\in\, F^{m-1}}\int_{\begin{smallmatrix}
y\,\in\,M_{m-1,j+1}( F\backslash \A) \\
X_1 \, \in \, M_{n-m-j,m-1}( F\backslash \A) \\
X_2\,\in\, M_{j,m-1}(\A) \\
u\,\in\,U( F)\backslash U(\A)
\end{smallmatrix}}
   \! \phi\(\!\ttwo{I_{m-1}}{}{Q}{I_{n-m+1}} \ttwo{I_{m-1}}{Y}{}{I_{n-m+1}} \!
\ttwo{I_{m-1}}{}{\begin{smallmatrix} X_1 \\ X_2 \\  0
\end{smallmatrix}}u g\)\,\times \\ & \ \ \qquad \qquad \qquad \qquad
\qquad\qquad \qquad \qquad \qquad \overline{\psi_N(u)}\,\overline{\psi(qy_1)}\,du\,dX_2\,dX_1\,dy\,.
\endaligned
\end{equation}
We calculate
\begin{equation}\label{approp24}
    \ttwo{I_{m-1}}{}{Q}{I_{n-m+1}} \ttwo{I_{m-1}}{Y}{}{I_{n-m+1}} \ \ = \ \
  \ttwo{I_{m-1}}{Y}{}{I_{n-m+1}} \ttwo{I_{m-1}}{}{Q}{I_{n-m+1}}  \ttwo{I_{m-1}}{}{}{u'},
\end{equation}
where $u'=I_{n-m+1}+QY\in U$. Therefore the argument of $\phi$ in (\ref{approp23}) is equal to
\begin{multline}\label{approp25}
\ttwo{I_{m-1}}{Y}{}{I_{n-m+1}} \ttwo{I_{m-1}}{}{Q}{I_{n-m+1}} \ttwo{I_{m-1}}{}{}{u'}
\ttwo{I_{m-1}}{}{\begin{smallmatrix} X_1 \\ X_2 \\  0 \end{smallmatrix}}u g \\
\ \ = \ \ \ttwo{I_{m-1}}{Y}{}{I_{n-m+1}} \ttwo{I_{m-1}}{}{Q}{I_{n-m+1}}
\ttwo{I_{m-1}}{}{u'\left[\begin{smallmatrix} X_1 \\ X_2 \\  0
\end{smallmatrix}\right]}{I_{n-m+1}} \ttwo{I_{m-1}}{}{}{u'u} g\,.
\end{multline}
Again, changing variables $u\mapsto (u')\i u$ maps any fundamental domain for $U( F)\backslash U(\A)$ into another,
and changes the character $\psi_N(u)$ to $\psi_N(u)\psi(qy_1)\i$ -- cf. proposition~\ref{leftfunddomprop}.  The
multiplication of the unipotent matrix $u'$ on $\left[\begin{smallmatrix} X_1 \\ X_2
\\  0
\end{smallmatrix}\right]$ serves to  add multiples of lower rows to a
higher rows.  Again, since $X_1$ and $X_2$ are integrated over abelian groups,
 this can be reversed by a change of variables which
does not destroy the fundamental domain for $X_1$ or change either Haar measure.
 After combining the sum over $q$ in (\ref{approp23}) with
the integration over the bottom row of $X_1$, so that it becomes an integration over $\A^{m-1}$ instead of $(
F\backslash \A)^{m-1}$, we have thus converted the expression for $I_j$ into that for $I_{j+1}$.

Finally, we  conclude by justifying the absolute convergence of all these integrals, as well as the manipulations
which formally relate them.  For this we shall use an argument due to Jacquet-Shalika (\cite[\S6.4]{jsextsq}) in
the context of exterior square $L$-functions. The above manipulations show that the absolute value of $I_j$ is
bounded by the analogous expression to (\ref{approp23}), but with absolute values around the integrand. That
expression is in turn bounded above by the expression for $I_{j+1}$ -- but again with absolute values around the
integrand -- because it is formed by combining a union of fundamental domains together. Thus each such expression
is absolutely convergent and all manipulations are justified, provided the final integral $I_{n-m}$, or
equivalently the integral in proposition~\ref{appprop}, is.  The absolute convergence of the $I_{n-m}$ integral
itself follows from the gauge estimates in \cite[\S5]{jacshaarch}.  Indeed, although the estimates there are for
the analogous local integrals instead,
 the mechanism
of bounding the unipotent integration by a rapidly decaying function of the unipotent variable applies equally to
the periods $V$, because of the expression for them as a sum of Whittaker functions mentioned just after the
statement of this proposition.
\end{proof}

\vspace{1cm}
\begin{tabular}{lcl}
Stephen D. Miller                    & & Wilfried Schmid \\
Department of Mathematics         & & Department of Mathematics \\
Hill Center-Busch Campus          & & Harvard University \\
Rutgers, The State University of New Jersey             & & Cambridge, MA 02138 \\
 110 Frelinghuysen Rd             & & {\tt schmid@math.harvard.edu}\\
 Piscataway, NJ 08854-8019\\
 {\tt sdmiller@math.rutgers.edu}  \\

\end{tabular}

\end{document}